\documentclass[11pt]{amsart}

\pdfoutput=1
\usepackage{graphicx}
\usepackage{ifpdf}
\usepackage{amsthm}
\usepackage{amsfonts}
\usepackage{multirow}
\usepackage{amssymb}
\usepackage{wrapfig}
\usepackage{amsmath}
\usepackage{cite}
\usepackage{framed}
\usepackage[top=2 cm,bottom=2 cm,left=3cm,right=3cm]{geometry}
\usepackage{caption}

\usepackage[symbol]{footmisc}
\usepackage{yfonts}
\usepackage[textwidth=2.4cm, textsize=tiny]{todonotes}
\usepackage{color}

\definecolor{link}{rgb}{0.18,0.25,0.78}

\usepackage[font=small,labelfont=bf]{caption}
\usepackage{subcaption}

\newcommand{\scalprod}[3]{\left\langle #2,#3\right\rangle_{#1}}
\newcommand{\norm}[2]{\left\lVert#2\right\rVert_{#1}}

\makeatletter
\renewcommand{\@seccntformat}[1]{\csname the#1\endcsname.\quad}
\makeatother

\newcommand{\re}{\mathbb{R}}

\newcommand{\diverg}{{\rm div}}

\newcommand{\dom}{{\rm dom}}
\newcommand{\Sgn}{{\rm Sgn}}
\newcommand{\sgn}{{\rm sgn}}

\newcommand{\changeurlcolor}[1]{\hypersetup{urlcolor=#1}} 

\numberwithin{equation}{section}

\theoremstyle{plain}
\newtheorem{thrm}{Theorem}[section]

\theoremstyle{definition}

\theoremstyle{plain}
\newtheorem{prop}[thrm]{Proposition}

\theoremstyle{plain}

\theoremstyle{plain}
\newtheorem{lem}[thrm]{Lemma}

\theoremstyle{plain}
\newtheorem{rem}[thrm]{Remark}

\usepackage{hyperref}
\hypersetup{
  colorlinks,
  citecolor= blue,
  linkcolor= link,
  urlcolor= link}

\usepackage{lipsum}

\makeatletter
\g@addto@macro{\endabstract}{\@setabstract}
\newcommand{\authorfootnotes}{\renewcommand\thefootnote{\@fnsymbol\c@footnote}}%
\makeatother

\begin{document}

\changeurlcolor{black}

\normalsize
 \begin{center}
 \large
  \textbf{  INFIMAL CONVOLUTION REGULARISATION FUNCTIONALS OF BV AND $\boldsymbol{\mathrm{L}}^{\boldsymbol{p}}$ SPACES. \\ PART I: THE FINITE $\mathbf{p}$ CASE}. \par \bigskip \bigskip
  
   \normalsize
  \textsc{Martin Burger} \textsuperscript{$\dagger$}, \textsc{Konstantinos Papafitsoros} \textsuperscript{$\ddagger$},
  \textsc{Evangelos Papoutsellis} \textsuperscript{$\ddagger$} \textsc{and} \textsc{Carola-Bibiane Sch{\"o}nlieb} \textsuperscript{$\ddagger$} \par \bigskip 

\let\thefootnote\relax\footnote{\textbf{Emails:} \href{mailto:martin.burger@wwu.de}{\nolinkurl{martin.burger@wwu.de}}, \href{mailto:kp366@cam.ac.uk}{\nolinkurl{kp366@cam.ac.uk}}, \href{mailto:ep374@cam.ac.uk}{\nolinkurl{ep374@cam.ac.uk}}, \href{mailto:cbs31@cam.ac.uk}{\nolinkurl{cbs31@cam.ac.uk}}}

\end{center}
\noindent
\textsuperscript{$\dagger$}Institute for Computational and Applied Mathematics, University of M\"{u}nster, Germany \\
\textsuperscript{$\ddagger$}Department of Applied Mathematics and Theoretical Physics, University of Cambridge, UK\\
  \par \bigskip



\begin{abstract}
We study a general class of infimal convolution type regularisation functionals suitable for applications in image processing. These functionals incorporate a combination of the  total variation ($\mathrm{TV}$) seminorm and $\mathrm{L}^{p}$ norms. A unified well-posedness analysis is presented and a detailed study of the one dimensional model is performed, by computing exact solutions for the corresponding denoising problem and the case $p=2$. Furthermore, the dependency of the regularisation properties of this infimal convolution approach to the choice of $p$ is studied.
It turns out that in the case $p=2$ this regulariser is equivalent to Huber-type variant of total variation regularisation. We provide numerical examples for image decomposition as well as for image denoising. We show that our model is capable of eliminating the staircasing effect,   a well-known disadvantage of total variation regularisation. Moreover as $p$ increases we obtain almost piecewise affine reconstructions, leading also to a better preservation of hat-like structures.\vspace{0.1cm}

\noindent
\textbf{Keywords}: Total Variation, Infimal convolution, Denoising, Staircasing, $\mathrm{L}^{p}$ norms, Image decomposition
\end{abstract}

\section{Introduction}
 In this paper we introduce  a family of novel $\mathrm{TV}$--$\mathrm{L}^{p}$ infimal convolution type functionals with applications in image processing:
 \begin{equation}\label{uw_min}
\mathrm{TVL}_{\alpha,\beta}^{p}(u):=\inf_{w\in \mathrm{L}^{p}(\Omega)} \alpha \|Du-w\|_{\mathcal{M}}+\beta\|w\|_{\mathrm{L}^{p}(\Omega)},\quad \alpha,\beta>0 \quad \text{and} \quad p>1.
\end{equation}
Here $\|\cdot\|_{\mathcal{M}}$ denotes the Radon norm of a measure.
The functional \eqref{uw_min} is suitable to be used as a regulariser in the context of variational non-smooth regularisation in imaging applications. We study the properties of \eqref{uw_min}, its regularising mechanism for different values of $p$ and apply it successfully to image denoising.

\subsection{Context}
After the introduction of the total variation ($\mathrm{TV}$) for image reconstruction purposes \cite{Rudin}, the use of non-smooth regularisers has become increasingly popular during the last decades (cf. \cite{tvzoo}). They are typically used in the context of variational regularisation, where the reconstructed image is obtained as a solution of a minimisation problem of the type:
\begin{equation}\label{kostas_general}
\min_{u} \frac{1}{s}\|f-Tu\|_{\mathrm{L}^{s}(\Omega)}^{s}+\Psi(u).
\end{equation}

The \emph{regulariser} is denoted here by $\Psi$. We assume that the data $f$, defined on a domain $\Omega\subset \mathbb{R}^{2}$, have been corrupted through a bounded, linear operator $T$ and additive (random) noise. Different values of $s$ can be considered  for the first term of \eqref{kostas_general}, the \emph{fidelity term}. For example, models incorporating a $\mathrm{L}^{2}$ fidelity term (resp. $L^{1}$) have be shown to be efficient for the restoration of images corrupted by Gaussian noise (resp. impulse noise). Of course, other types of noise can also be considered  and in those cases the 
form of the fidelity term is adjusted accordingly. Typically, one or more parameters within $\Psi$ balance the strength of regularisation against the fidelity term in the minimisation \eqref{kostas_general}.

The advantage of using non-smooth regularisers is that the regularised images have sharp edges (discontinuities). For instance, it is a well-known fact that $\mathrm{TV}$ regularisation promotes piecewise constant reconstructions, thus preserving discontinuities. However, this also leads to blocky-like artifacts in the reconstructed image, an effect known as \emph{staircasing}.  Recall at this point that for two dimensional images $u\in \mathrm{L}^{1}(\Omega)$, the definition of the total variation functional reads:
\begin{equation}\label{kostas_TV}
\mathrm{TV}(u):=\sup \left \{\int_{\Omega}u\, \mathrm{div}\phi\,dx: v\in C_{c}^{\infty}(\Omega,\mathbb{R}^{2}),\;\|\phi\|_{\infty}\le 1 \right\}.
\end{equation}
The total variation uses only first-order derivative information in the regularisation process. This can be  seen from that fact that for $\mathrm{TV}(u)<\infty$ the distributional derivative $Du$ is a finite Radon measure and  $\mathrm{TV}(u)=\|Du\|_{\mathcal{M}}$. Moreover if $u\in \mathrm{W}^{1,1}(\Omega)$ then $\mathrm{TV}(u)=\int_{\Omega}|\nabla u|\,dx$, i.e., the total variation is the $\mathrm{L}^{1}$ norm of the gradient of $u$. Higher-order extensions of the total variation functional are widely explored in the literature e.g. \cite{ChambolleLions, chan2001high, CEP07, LLT03, LT06, lefkimmiatis2010hessian, Piffet, TGV, mineJMIV}. The incorporation of second-order derivatives is shown to reduce or even eliminate the staircasing effect.  The most successful regulariser of this kind is the second order total generalised variation (TGV) introduced by Bredies et al. \cite{TGV}. Its definition reads
\begin{equation}\label{kostas_TGV}
\mathrm{TGV}_{\alpha,\beta}^{2}(u):=\min_{w\in\mathrm{BD}(\Omega)} \alpha \|Du-w\|_{\mathcal{M}}+\beta\|\mathcal{E}w\|_{\mathcal{M}}.
\end{equation}
Here $\alpha,\beta$ are positive parameters and $\mathrm{BD}(\Omega)$ is the space of functions of bounded deformation, i.e., the space of all $\mathrm{L}^{1}(\Omega)$ functions $w$, whose symmetrised distributional derivative $\mathcal{E}w$ is a finite Radon measure. This is a less regular space than the usual space of functions of bounded variation $\mathrm{BV}(\Omega)$ for which the full gradient $Du$ is required to be a finite Radon measure. 
Note that if  the variable $w$ in the definition \eqref{kostas_TGV} is  forced to be the gradient of another function then we obtain the classical infimal convolution regulariser of Chambolle--Lions \cite{ChambolleLions}. In that sense $\mathrm{TGV}$ can be seen as a particular instance of infimal convolution, optimally balancing first and second-order information.

In the discrete formulation of $\mathrm{TGV}$ (as well as for $\mathrm{TV}$) the Radon norm is interpreted as an $\mathrm{L}^{1}$ norm.
The motivation for the current and the follow-up paper is to explore the capabilities of  $\mathrm{L}^{p}$ norms within first-order regularisation functionals designed for image processing purposes. The use of $\mathrm{L}^{p}$ norms for $p> 1$ has been exploited in different contexts -- infinity and $p$-Laplacian (cf. e.g. \cite{ElionVese} and \cite{Kuijper} respectively).

\subsection{Our contribution}

Comparing the definition \eqref{uw_min} with the definition of $\mathrm{TGV}$ in \eqref{kostas_TGV}, we see that the Radon norm of the symmetrised gradient of $w$ has been substituted by the $\mathrm{L}^{p}$ norm of $w$, thus reducing the order of regularisation. Up to our knowledge, this is the first paper that provides a thorough analysis of $\mathrm{TV}$--$\mathrm{L}^{p}$ infimal convolution models \eqref{uw_min} in this generality.
We show that the minimisation in \eqref{uw_min} is well-defined and that $\mathrm{TVL}_{\alpha,\beta}^{p}(u)<\infty$ if and only if $\mathrm{TV}(u)<\infty$.
Hence $\mathrm{TVL}_{\alpha,\beta}^{p}$ regularised images belong to $\mathrm{BV}(\Omega)$ as desired.

In order to get more insight in the regularising mechanism of the  $\mathrm{TVL}_{\alpha,\beta}^{p}$ functional we provide a detailed and rigorous analysis of its one dimensional version of the corresponding $\mathrm{L}^{2}$ fidelity denoising problem
\begin{equation}\label{intro_denoising}
\min_{u\in\mathrm{BV}(\Omega)}\frac{1}{2}\|f-u\|_{\mathrm{L}^{2}(\Omega)}^{2}+\mathrm{TVL}_{\alpha,\beta}^{2}(u).
\end{equation}
 For the  denoising problem \eqref{intro_denoising} with $p=2$ we also compute exact solutions for simple one dimensional data. We show that the obtained solutions are piecewise smooth, in contrast to $\mathrm{TV}$ (piecewise constant) and $\mathrm{TGV}$ (piecewise affine) solutions.
Moreover, we show that for $p=2$, the $2$-homogeneous analogue of the functional \eqref{uw_min}
\begin{equation}\label{intro_phomo}
F(u)=\min_{w\in \mathrm{L}^{2}(\Omega)} \alpha \|Du-w\|_{\mathcal{M}}+\frac{\beta}{2}\|w\|_{\mathrm{L}^{2}(\Omega)}^{2},
\end{equation}
is  equivalent to a variant of \emph{Huber} $\mathrm{TV}$ \cite{Huber}, with the functional \eqref{intro_phomo} having a close connection with \eqref{uw_min} itself.
Huber total variation is a smooth approximation of total variation and even though it has been widely used in the imaging and inverse problems community, it has not been analysed adequately. Hence, as a by-product of our analysis, we compute exact solutions of the one dimensional Huber TV denoising problem.

We proceed with exhaustive numerical experiments focusing on \eqref{intro_denoising}. Our analysis is confirmed by the fact that the analytical results coincide with the numerical ones. Furthermore, we observe that even though a first-order regularisation functional is used, we are capable of eliminating the staircasing effect, similarly to Huber $\mathrm{TV}$.  By \emph{Bregmanising} our method \cite{Osher1}, we are also able to enhance the contrast of the reconstructed images, obtaining results very similar in quality to the $\mathrm{TGV}$ ones. We observe numerically that high values of $p$ promote almost affine structures similar to second-order regularisation methods. We shed more light of this behaviour in the follow-up paper \cite{partII} where we study in depth the case $p=\infty$. Let us finally note that we also consider a modified version of the functional \eqref{uw_min} where $w$ is restricted to be a gradient of another function leading to the more classical infimal convolution setting. Even though, this modified model is not so successful  in staircasing reduction, it is effective in decomposing an image into piecewise constant and smooth parts.

\subsection{Organisation of the paper}

After the introduction we proceed with the introduction of our model in Section \ref{sec:TVLp}. We prove the well-posedness of \eqref{uw_min}, we provide an equivalent definition and we prove its Lipschitz equivalence with the $\mathrm{TV}$ seminorm.
 We finish this section with a well-posedness result of the corresponding $\mathrm{TVL}_{\alpha,\beta}^{p}$ regularisation problem using standard tools.
 
 In Section \ref{sec:phomo}  we establish a link between the $\mathrm{TVL}_{\alpha,\beta}^{p}$ functional and its $p$-homogeneous analogue (using the $p$-th power of $\|\cdot\|_{\mathrm{L}^{p}(\Omega)}$). The $p$-homogeneous functional (for $p=2$) is further shown to be equivalent to Huber total variation.

We study the corresponding one dimensional model in Section \ref{sec:onedimensional} focusing on the $\mathrm{L}^{2}$ fidelity denoising case. More specifically, after deriving the optimality conditions using Fenchel--Rockafellar duality in Section \ref{sec:optimality}, we explore the structure of solutions in Section \ref{sec:structure}. In Section \ref{sec:exact} we compute exact solutions for the case $p=2$, considering a simple step function as data. 

In Section \ref{sec:infimal} we present a variant of our model suitable for image decomposition purposes, i.e.,  geometric decomposition into piecewise constant and smooth structures.

Section \ref{sec:numerics} focuses on numerical experiments. Confirmation of the obtained one dimensional analytical results
is done in Section \ref{sec:numerics:1d}, while two dimensional denoising experiments are performed in Section \ref{sec:numerics:2d} using the split Bregman method.
There, we show that our approach can lead to  elimination of the staircasing effect and we also show that by using a Bregmanised version we can also enhance the contrast,  achieving  results very close to $\mathrm{TGV}$, a method considered state of the art in the context of variational regularisation. We finish the section with some image decomposition examples and we summarise our results in Section \ref{sec:conclusion}.

In the appendix, we remind the reader of some basic facts from the theory of Radon measures and $\mathrm{BV}$ functions.

\section{Total variation and $\mathrm{L}^{p}$ regularisation}\label{sec:TVLp}

In this section we introduce  the $\mathrm{TV}$--$\mathrm{L}^{p}$ functional \eqref{uw_min} as well as some of its main properties.
For $\alpha,\beta>0$ and $1<p\le \infty$, we define  $\mathrm{TVL}^{p}_{\alpha,\beta}: \mathrm{L}^{1}(\Omega)\rightarrow\overline{\re}$ as follows:
\begin{equation*}
\mathrm{TVL}^{p}_{\alpha,\beta}(u):=\min_{w\in \mathrm{L}^{p}(\Omega)} \alpha\norm{\mathcal{M}}{Du-w}+\beta\norm{\mathrm{L}^{p}(\Omega)}{w}.
\end{equation*}
The next proposition asserts that the minimisation in \eqref{uw_min} is indeed well-defined. We omit the proof, which is based on standard coercivity and weak lower semicontinuity techniques:

\begin{prop}
Let $u\in\mathrm{ BV(\Omega)}$ with $1< p\leq\infty$ and $\alpha,\beta>0$. Then the minimum in the definition \eqref{uw_min} is attained.
\end{prop}

Another useful formulation of the definition \eqref{uw_min} is the  \emph{dual} formulation:
\begin{equation}
\mathrm{TVL}^{p}_{\alpha,\beta}(u)=\sup\left\{\int_{\Omega}u\,\diverg \phi \,dx:\; \phi\in C^{1}_{c}(\Omega), \norm{\infty}{\phi}\leq\alpha,\; \norm{\mathrm{L}^q(\Omega)}{\phi}\leq\beta \right\},
\label{sup_uw}
\end{equation}
 The next Proposition shows  that the two expressions coincide indeed.
 \begin{prop} Let $u\in\mathrm{ BV(\Omega)}$ and $1<p\leq\infty$ then 
 \begin{align*}
& \min_{w\in\mathrm{L}^{p}(\Omega)}\alpha\norm{\mathcal{M}}{Du-w}+\beta\norm{\mathrm{L}^{p}(\Omega)}{w} = 
\sup\left\{\int_{\Omega}u\,\diverg \phi\,dx:\; \phi\in C^{1}_{c}(\Omega), \norm{\infty}{\phi}\leq\alpha,\;\norm{\mathrm{L}^q\Omega)}{\phi}\leq\beta \right\}.
 \end{align*}
 \begin{proof} First notice that in \eqref{sup_uw}, we can replace $C^1_c(\Omega)$ by $C^1_0(\Omega)$, since $\overline{C^1_c(\Omega)}=C^1_0(\Omega)$ with the closure taken with respect to the uniform norm. We define 
 \begin{align*}
& X=C^1_0(\Omega),\\
& F_{1}:X\rightarrow \overline{\re}\mbox{, with } F_{1}(\phi)=\mathbb{I}_{\left\{\norm{\mathrm{L}^{q}(\Omega)}{\cdot}\leq\beta\right\}}(\phi),\\
& F_{2}:X\rightarrow \overline{\re}\mbox{, with } F_{2}(\phi)=\mathbb{I}_{\left\{\norm{\infty}{\cdot}\leq\alpha\right\}}(\phi)-\int_{\Omega}u\,\diverg\phi \, dx.
 \end{align*}
 Then, we can rewrite \eqref{sup_uw} as
 $$\mathrm{TVL}^{p}_{\alpha,\beta}(u)=
 -\inf_{\substack{\phi\in X \\ \norm{\infty}{\phi}\leq\alpha \\ \norm{\mathrm{L}^{q}(\Omega)}{\phi}\leq\beta}}
 \left\{-\int_{\Omega}u\,\diverg \phi \, dx\right\}
 =-\inf_{\phi\in X} F_{1}(\phi)+F_{2}(\phi).$$ 
 The Fenchel--Rockafellar duality theory, see \cite{Ekeland}, allows to establish a relation between the primal problem
 \[-\inf_{\phi\in X} F_{1}(\phi)+F_{2}(\phi), \]
  and its dual
 \[ \min_{w\in X^{*}}F_{1}^{*}(-w)+F_{2}^{\ast}(w).\]
 Here $F_{1}^{*}$ and $F_{2}^{*}$ denote the convex conjugate of $F_{1}$ and $F_{2}$ respectively.
  In order to obtain such a connection, we follow  \cite{attouch_brezis} where it suffices to show that $$\bigcup_{\lambda\geq0}\lambda\left(\dom F_{2} - \dom F_{1}\right)$$ is a closed vector space. Indeed, we have that 
   $$\bigcup_{\lambda\geq0}\lambda\left(\dom F_{2} - \dom F_{1}\right)\subset X,$$
 and for every $\phi \in X$, we can write $\phi=\lambda(\lambda^{-1}\phi-0)$ with $\norm{\infty}{\lambda^{-1}\phi}\leq\alpha$ and $0\in\dom F_{1}$. Hence, $\underset{\lambda\geq0}{\bigcup}\lambda\left(\dom F_{2} - \dom F_{1}\right)=X$ is a closed vector space and there is no duality gap i.e., 
 $$\inf_{\phi\in X} \left\{F_{1}(\phi)+F_{2}(\phi)\right\}+\min_{w\in X^{*}}\left\{F_{1}^{*}(w)+F_{2}^{\ast}(w)\right\}=0.$$
 Finally, we have
 $$F_{1}^{*}(-w)=\sup_{\substack{\phi\in C^{1}_{0}(\Omega)\\  \norm{\mathrm{L}^{q}(\Omega)}{\phi}\leq\beta   }}\scalprod{}{\phi}{w}
 =\beta\norm{\mathrm{L}^{p}(\Omega)}{w},$$ 
 and similarly, $$F_{2}^{*}(w)=\sup_{\substack{\phi\in C^{1}_{0}(\Omega)\\  \norm{\infty}{\phi}\leq\alpha   }}\scalprod{}{w}{\phi}+\scalprod{}{u}{\phi'}=\alpha\norm{\mathcal{M}}{Du-w}.$$
 Thus the desired equality is proven.
 \end{proof}
 \end{prop}

\begin{rem} The dual formulation of $\mathrm{TVL}^{p}_{\alpha,\beta}:\mathrm{L}^{1}(\Omega)\rightarrow\overline{\re}$ is useful since one can easily derive that $\mathrm{TVL}^{p}_{\alpha,\beta}$ is lower semicontinuous with respect to the strong $\mathrm{L}^{1}$ topology since it is a pointwise supremum of continuous functions. 
\label{rem_tvlp}
\end{rem}

The following lemma shows that the  $\mathrm{TVL}^{p}_{\alpha,\beta}$ functional is Lipschitz equivalent to the total variation seminorm.
  
\begin{lem} Let $u\in\mathrm{L^{1}(\Omega)}$ and $1<p\leq\infty$. Then $\mathrm{TVL}^{p}_{\alpha,\beta}(u)<\infty$ if and only if $u\in\mathrm{BV}(\Omega)$ and there exist constants $0<C_{1},C_{2}<\infty$ such that 
\begin{equation}
C_{2}\norm{\mathcal{M}}{Du}\leq \mathrm{TVL}^{p}_{\alpha,\beta} (u) \leq C_{1} \norm{\mathcal{M}}{Du}.
\label{tvlp_ineq}
\end{equation}
\begin{proof}
Let $u\in\mathrm{BV}(\Omega)$. Using \eqref{uw_min} we have that
$$\mathrm{TVL}^{p}_{\alpha,\beta}(u)\leq\alpha\norm{\mathcal{M}}{Du-w}+\beta\norm{\mathrm{L}^{p}(\Omega)}{w}.$$
for every $w\in\mathrm{L}^{p}(\Omega)$.
Setting $w=0$ and $C_{1}=\alpha$,  we obtain
 $$\mathrm{TVL}^{p}_{\alpha,\beta}(u)\leq C_{1} \norm{\mathcal{M}}{Du}.$$
For the other direction, we have for any $w\in\mathrm{L}^{p}(\Omega)\subset\mathrm{L^{1}(\Omega)}$ by the triangle inequality
\begin{align*}
\norm{\mathcal{M}}{Du} & \leq \norm{\mathcal{M}}{Du-w}+\norm{\mathrm{L}^{1}(\Omega)}{w}\leq \norm{\mathcal{M}}{Du-w}+|\Omega|^{\frac{1}{q}}\norm{\mathrm{L}^{p}(\Omega)}{w}\\
                                         & \leq C ( \norm{\mathcal{M}}{Du-w} + \norm{\mathrm{L}^{p}(\Omega)}{w}),
\end{align*}
with $C=\max(1,|\Omega|^{\frac{1}{q}})$. For appropriate choice of $C_2$ we obtain
$$ C_{2}\norm{\mathcal{M}}{Du}\leq\alpha\norm{\mathcal{M}}{Du-w} +\beta\norm{\mathrm{L}^{p}(\Omega)}{w},$$ which yields the left-hand side inequality.
\end{proof}
\label{lem_tvlp_ineq}
\end{lem}

Having shown the basic properties of the  $\mathrm{TVL}^{p}_{\alpha,\beta}$ functional, we can use it as a regulariser for variational imaging problems, by minimising
\begin{equation}\label{kostas_general_TVLp}
\min_{u\in\mathrm{BV}(\Omega)} \frac{1}{s} \|f-Tu\|_{\mathrm{L}^{s}(\Omega)}^{s}+\mathrm{TVL}^{p}_{\alpha,\beta}(u),\quad s\geq 1,
\end{equation}
where $T: \mathrm{L}^{2}(\Omega)\to \mathrm{L}^{2}(\Omega)$ is a bounded, linear operator and $f\in \mathrm{L}^{2}(\Omega)$. 
We conclude our analysis with existence and uniqueness results for the minimisation problem \eqref{kostas_general_TVLp}.

\begin{thrm} Let $1<p\leq\infty$ and $f\in\mathrm{L}^{2}(\Omega)$. If $T(\mathcal{X}_{\Omega})\ne 0$ then there exists a solution $u\in\mathrm{L}^{2}(\Omega)\cap\mathrm{BV}(\Omega)$ for the problem \eqref{kostas_general_TVLp}.  If $s>1$ and $T$ is injective then the solution is unique.
\begin{proof} 
The proof is a straightforward application of the direct method of calculus of variations. We simply take advantage of
\eqref{tvlp_ineq} and the compactness theorem in $\mathrm{BV}(\Omega)$ along with the lower semicontinuity property of $\mathrm{TVL}^{p}_{\alpha,\beta}$. We also refer the reader to the corresponding proofs in \cite{Luminita, mineJMIV}.
\end{proof}
\label{tvlp_exis}
\end{thrm} 

Since we are mainly interested in studying the regularising properties of $\mathrm{TVL}^{p}_{\alpha,\beta}$, from now on we focus on the case where $s=2$ and $T$ is the identity function (denoising task) where rigorous analysis can be carried out. We thus define the following problem
\begin{equation*}
\min_{u\in\mathrm{BV}(\Omega)}\frac{1}{2}\norm{\mathrm{L}^{2}(\Omega)}{f-u}^{2}+ \mathrm{TVL}^{p}_{\alpha,\beta}(u),
\end{equation*}
or equivalently 
\begin{equation}
\min_{\substack{u\in\mathrm{BV}(\Omega)\\ w\in \mathrm{L}^{p}(\Omega)}}\frac{1}{2}\norm{\mathrm{L}^{2}(\Omega)}{f-u}^{2}+ \alpha \|Du-w\|_{\mathcal{M}}+ \beta \|w\|_{\mathrm{L}^{p}(\Omega)}.
\tag{$\mathcal{P}$}
\label{tvlp_min}
\end{equation}

\section{The $p$-homogeneous analogue and relation to Huber $\mathrm{TV}$}\label{sec:phomo}

Before we proceed to a detailed analysis of the one dimensional version of \eqref{tvlp_min}, in this section we consider its $p$-homogeneous analogue
\begin{equation}\label{kostas_phomo}\tag{$\mathcal{P}_{p-hom}$}
\min_{\substack{u\in\mathrm{BV}(\Omega) \\ w\in\mathrm{L}^{p}(\Omega)}} \frac{1}{2}\|f-u\|_{\mathrm{L}^{2}(\Omega)}^{2}+\alpha \|Du-w\|_{\mathcal{M}}+\frac{\beta}{p}\|w\|_{\mathrm{L}^{p}(\Omega)}^{p},\quad 1<p<\infty.
\end{equation}
We  show in Proposition \ref{kostas_phomo_1homo} that there is a strong connection between the models \eqref{tvlp_min} and \eqref{kostas_phomo}.
The reason for the introduction of \ref{kostas_phomo} is that, in certain cases,  it is technically easier to derive exact solutions for \eqref{kostas_phomo} rather than for \eqref{tvlp_min}
straightforwardly, see Section \ref{sec:exact}.   Moreover, here we can  guarantee the uniqueness of the optimal $w^{\ast}$, since 
$$w^{\ast}=\underset{w\in\mathrm{L}^p(\Omega)}{\operatorname{argmin}}\;\alpha\norm{\mathcal{M}}{Du-w}+\frac{\beta}{p}\norm{\mathrm{L}^{p}(\Omega)}{w}^p,$$ 
and thus $w^{\ast}$ is unique as a minimiser of a  strictly convex  functional since $1<p<\infty$. Hence, compared to $\eqref{tvlp_min}$, an optimal solution pair of \eqref{kostas_phomo} is unique.
The next proposition says that, unless $f$ is a constant function then the optimal $w$ in \eqref{kostas_phomo} cannot be zero but nonetheless converges to $0$ as $\beta\to\infty$. In essence, this means that one cannot obtain $\mathrm{TV}$ type solutions with the $p$-homogeneous model.

\begin{prop}
Let  $1<p<\infty$, $f\in\mathrm{L}^{2}(\Omega)$ and let $(w^{\ast},u^{\ast})$ be an optimal solution pair of the $p$-homogeneous problem \eqref{kostas_phomo}. Then $w^{\ast}=0$ if and only if $f$ is a constant function.  For general data $f$, we have that  $w^{\ast}\to 0$ in $\mathrm{L}^{p}(\Omega)$ for $\beta\to \infty$.
\begin{proof}
It follows immediately that if $f$ is constant then $(0,f)$ is the optimal pair for $\eqref{kostas_phomo}$.
Suppose now that $(w^{\ast},u^{\ast})$ solve \eqref{kostas_phomo}. It is easy to check that the following also hold:
\begin{align}
w^{\ast} &\in \underset{w\in\mathrm{L}^{p}(\Omega)}{\operatorname{argmin}}\; \alpha\|\nabla u^{\ast}-w\|_{\mathrm{L}^{1}(\Omega)}+\frac{\beta}{p} \|w\|_{\mathrm{L}^{p}(\Omega)}^{p},\label{kostas_min1}\\
u^{\ast} &= \underset{u\in\mathrm{BV}(\Omega)}{\operatorname{argmin}} \; \frac{1}{2} \|f-u\|_{\mathrm{L}^{2}(\Omega)}^{2}+\alpha \|Du-w^{\ast}\|_{\mathcal{M}}.\label{kostas_min2}
\end{align}
In particular, \eqref{kostas_min1} implies that 
\begin{equation}
0\in \alpha \partial \|\nabla u^{\ast}-\cdot\|_{\mathrm{L}^{1}(\Omega)}(w^{\ast})+\beta |w^{\ast}|^{p-2}w^{\ast},\label{kostas_sub1}
\end{equation}
Suppose now that $w^{\ast}=0$. Then \eqref{kostas_min2} becomes 
\begin{equation}\label{kostas_min2_2}
u^{\ast} = \underset{u\in\mathrm{BV}(\Omega)}{\operatorname{argmin}} \; \frac{1}{2} \|f-u\|_{\mathrm{L}^{2}(\Omega)}^{2}+\alpha \|Du\|_{\mathcal{M}}.
\end{equation}
Furthermore, since $(0,u^{\ast})$ solve \eqref{kostas_phomo}, then for every $u\in C_{c}^{\infty}(\Omega)$ and $\epsilon>0$, the pair $(\epsilon\nabla u,u^{\ast}+\epsilon u )$ is suboptimal for \eqref{kostas_phomo}, i.e.,
\begin{align*}
 \frac{1}{2}\|f-u^{\ast}\|_{\mathrm{L}^{2}(\Omega)}^{2}+\alpha \|Du^{\ast}\|_{\mathcal{M}}&\le  \frac{1}{2}\|(f-u^{\ast})-\epsilon u\|_{\mathrm{L}^{2}(\Omega)}^{2}+\alpha \|D(u^{\ast}+\epsilon u)-\epsilon \nabla u\|_{\mathcal{M}}+\frac{\beta}{p}\|\epsilon\nabla u\|_{\mathrm{L}^{p}(\Omega)}^{p},
\end{align*}
from which we take
\begin{align*}
 \frac{1}{2}\|f-u^{\ast}\|_{\mathrm{L}^{2}(\Omega)}^{2}&\le  \frac{1}{2}\|(f-u^{\ast})-\epsilon u\|_{\mathrm{L}^{2}(\Omega)}^{2}+\frac{\beta}{p}\|\epsilon\nabla u\|_{\mathrm{L}^{p}(\Omega)}^{p}\\
0&\le    \frac{\epsilon^{2}}{2}\|u\|_{\mathrm{L}^{2}(\Omega)}^{2}-\epsilon\int_{\Omega}(f-u^{\ast})u\,dx+\frac{\beta\epsilon^{p}}{p}\|\epsilon\nabla u\|_{\mathrm{L}^{p}(\Omega)}^{p}.
\end{align*}
By dividing the last inequality by $\epsilon$ and taking the limit $\epsilon\to 0$ we have that $\int_{\Omega}(f-u^{\ast})u\,dx\le 0$. By considering the analogous perturbations $u^{\ast}-\epsilon u$ , we obtain similarly that $\int_{\Omega}(f-u^{\ast})u\,dx\ge 0$ and thus
\[\int_{\Omega}(f-u^{\ast})u\,dx= 0,\quad \forall u\in C_{c}^{\infty}(\Omega).\]
Hence $u^{\ast}=f$ and since $u^{\ast}$ solves \eqref{kostas_min2_2} this can only happen when $f$ is a constant function.

%
For the last part of the proposition, (supposing $f\ne 0$), simply observe that for every $u\in\mathrm{BV}(\Omega)$ and $w\in \mathrm{L}^{p}(\Omega)$ we have that
\begin{equation*}
\frac{1}{2}\|f-u^{\ast}\|_{\mathrm{L}^{2}(\Omega)}^{2}+\alpha \|Du^{\ast}-w^{\ast}\|_{\mathcal{M}}+\frac{\beta}{p} \|w^{\ast}\|_{\mathrm{L}^{p}(\Omega)}^{p}
\le \frac{1}{2}\|f-u\|_{\mathrm{L}^{2}(\Omega)}^{2}+\alpha \|Du-w\|_{\mathcal{M}}+\frac{\beta}{p} \|w\|_{\mathrm{L}^{p}(\Omega)}^{p},
\end{equation*}
and setting $u=w=0$, we obtain
\[\frac{1}{p} \|w^{\ast}\|_{\mathrm{L}^{p}(\Omega)}^{p}\le\frac{1}{2\beta} \|f\|_{\mathrm{L}^{2}(\Omega)}^{2},\]
and thus $\|w^{\ast}\|_{\mathrm{L}^{p}(\Omega)}^{p}\to 0$ when $\beta\to \infty$.
\end{proof}
\label{kostas_w_0}
\end{prop}

 We can further establish a connection between the $1$-homogeneous $\eqref{tvlp_min}$ and the $p$-homogeneous model \eqref{kostas_phomo}:
 
\begin{prop}
Let  $1<p<\infty$ and $f\in\mathrm{L}^{2}(\Omega)$ not a constant. A pair $(w^{\ast},u^{\ast})$ is a solution of \eqref{kostas_phomo}
 with parameters $(\alpha,\beta_{p-hom})$ if and only if it is also a solution of \eqref{tvlp_min} with parameters $(\alpha,\beta_{1-hom})$ where $\beta_{1-hom}=\beta_{p-hom}\|w^{\ast}\|_{\mathrm{L}^{p}(\Omega)}^{p-1}$.
\begin{proof}
Since $f$ is not a constant by the previous proposition we have that $w^{\ast}\ne 0$. 
Note that for an arbitrary function $u\in\mathrm{BV}(\Omega)$:
\begin{align*}
w^{\ast} & \in \underset{w\in\mathrm{L}^{p}(\Omega)}{\operatorname{argmin}}\;\alpha\norm{\mathcal{M}}{Du-w}+\frac{\beta_{p-hom}}{p}\norm{\mathrm{L}^{p}(\Omega)}{w}^{p}\Leftrightarrow\\
0 & \in\alpha\partial\norm{\mathcal{M}}{Du-\cdot}(w^{\ast})+\beta_{p-hom}|w^{\ast}|^{p-2}w^{\ast}\Leftrightarrow\\
0 & \in\alpha\partial\norm{\mathcal{M}}{Du-\cdot}(w^{\ast})+\frac{\beta_{1-hom}}{\norm{\mathrm{L}^{p}(\Omega)}{w^{\ast}}^{p-1}  }|w^{\ast}|^{p-2}w^{\ast}\Leftrightarrow
\\
w^{\ast} & \in \underset{w\in\mathrm{L}^p(\Omega)}{\operatorname{argmin}}\;\alpha\norm{\mathcal{M}}{Du-w}+\beta_{1-hom}\norm{\mathrm{L}^{p}(\Omega)}{w}.
\end{align*}
This means that $w^{\ast}$ is an admissible solution for both problems \eqref{tvlp_min} and \eqref{kostas_phomo}, with the corresponding set of parameters $(\alpha,\beta_{1-hom})$ and $(\alpha,\beta_{p-hom})$ respectively. The fact that the same holds for $u^{\ast}$ as well, comes from the fact that in both problems we have
\[u^{\ast}\in\underset{u\in\mathrm{BV}(\Omega)}{\operatorname{argmin}} \frac{1}{2}\|f-u\|_{\mathrm{L}^{2}(\Omega)}^{2}+\alpha \|Du-w^{\ast}\|_{\mathcal{M}}.\]
\end{proof}
\label{kostas_phomo_1homo}
\end{prop}

Finally,  it turns out that for $p=2$, problem \eqref{kostas_phomo} is essentially equivalent to the widely used Huber total variation regularisation, \cite{Huber}. We show that in the next proposition.

\begin{prop}
Consider the functional $F:\mathrm{BV}(\Omega)\to \mathbb{R}$ with 
\begin{equation}
F(u)=\min_{w\in\mathrm{L}^{2}(\Omega)} \alpha \|Du-w\|_{\mathcal{M}}+\frac{\beta}{2}\|w\|_{\mathrm{L}^{2}(\Omega)}^{2}.
\end{equation}
Then 
\[F(u)=\int_{\Omega}\varphi(\nabla u)\,dx + \alpha |D^{s}u|(\Omega),\]
where $\varphi:\mathbb{R}^{d}\to\mathbb{R}$ with
\[\varphi(x)=
\begin{cases}
\alpha |x|-\frac{\alpha^{2}}{2\beta}, & |x|\ge \frac{\alpha}{\beta},\\
\frac{\beta}{2} |x|^{2}, 		      & |x|\le \frac{\alpha}{\beta}. 	
\end{cases}
\]
\label{equiv_tv_huber}
\end{prop} 
\begin{proof}
We have
\begin{align*}
F(u)&= \min_{w\in\mathrm{L^{2}(\Omega)}} \alpha \|Du-w\|_{\mathcal{M}}+\frac{\beta}{2}\|w\|_{\mathrm{L}^{2}(\Omega)}^{2}\\
        &= \alpha |D^{s}u|(\Omega)+\alpha \min_{w\in \mathrm{L}^{2}(\Omega)} \int_{\Omega} |\nabla u-w|+\frac{\beta}{2\alpha}|w|^{2}\,dx.
\end{align*}
So we focus on the minimisation problem
\begin{equation}\label{hubermin1}
\min_{w\in \mathrm{L}^{2}(\Omega)} \int_{\Omega} |\nabla u-w|+\frac{\beta}{2\alpha}|w|^{2}\,dx.
\end{equation}
Baring in mind that (as it can easily checked) for $c\in\mathbb{R}^{d}$ and $\lambda>0$, 
\[\underset{y\in\mathbb{R}^{d}}{\operatorname{argmin}}\; |c-y|+\frac{\lambda}{2}|y|^{2}=
\begin{cases}
\frac{1}{\lambda}\frac{c}{|c|} & \mbox{ if } |c|\ge \frac{1}{\lambda},\\
c & \mbox{ if } |c|< \frac{1}{\lambda},
\end{cases}
\]
and
\[\min_{y\in \mathbb{R}^{d}} |c-y|+\frac{\lambda}{2}|y|^{2}=
\begin{cases}
|c|-\frac{1}{2\lambda} & \mbox{ if } |c|\ge \frac{1}{\lambda},\\
\frac{\lambda}{2}|c|^{2} & \mbox{ if } |c|< \frac{1}{\lambda},
\end{cases}
\]
it is straightforwardly verified setting $\lambda=\beta/\alpha$ that the function
\[w^{\ast}=\frac{\alpha}{\beta}\frac{\nabla u}{|\nabla u|}\mathcal{X}_{\{|\nabla u|\ge \frac{\alpha}{\beta}\}}+\nabla u \mathcal{X}_{\{|\nabla u|< \frac{\alpha}{\beta}\}},\] 
belongs to $\mathrm{L}^{\infty}(\Omega)\subset \mathrm{L}^{2}(\Omega)$ and solves \eqref{hubermin1} with optimal value equal to $\frac{1}{\alpha}\int_{\Omega}\varphi(\nabla u)\,dx$.
\end{proof}

\section{The one dimensional case}\label{sec:onedimensional}
In order to get more insight into the structure of solutions of the problem \eqref{tvlp_min}, in this section we study its one dimensional version. As above, we focus on the finite $p$ case, i.e., $1<p<\infty$. The case $p=\infty$ leads to several additional complications and will be  subject of a forthcoming paper \cite{partII}. For this section $\Omega\subset\re$ is an open and bounded interval, i.e., $\Omega=(a,b)$. 
Our analysis follows closely the ones in \cite{Bredies} and \cite{papafitsoros} where the one dimensional $\mathrm{L}^{1}$--$\mathrm{TGV}$ and $\mathrm{L}^{2}$--$\mathrm{TGV}$ problems are studied respectively.

\subsection{Optimality conditions}\label{sec:optimality}
\label{sec_opt_cond}
In this section, we derive the optimality conditions  for the one dimensional problem \eqref{tvlp_min}. We initially start our analysis by defining the predual problem \eqref{predual}, proving existence and uniqueness for its solutions. 
We will employ again the Fenchel--Rockafellar duality theory in order  to find a connection between their solution pairs.

We define the predual problem \eqref{predual} as ($q$ H\"older conjugate to $p$):

\begin{equation*}
\tag{$\mathcal{P}^{*}$}-\inf\left\{\int_{\Omega}f\phi'dx+\frac{1}{2}\int_{\Omega}(\phi')^{2}dx: \;\phi\in\mathrm{H}^{1}_{0}(\Omega), \norm{\mathrm{L}^{q}(\Omega)}{\phi}\leq\beta, \norm{\infty}{\phi}\leq\alpha\right\}.
\label{predual}
\end{equation*}
Existence and uniqueness can be verified by standard arguments:

\begin{prop} For $f \in L^2(\Omega)$, the predual problem \eqref{predual}  admits a unique solution in $\mathrm{H}_{0}^{1}(\Omega).$
\label{existence}
\end{prop}

Observe now that we can also write down the predual problem \eqref{predual} using the following equivalent formulation:
\begin{equation}
-\inf_{(\phi,\Omega)\in X}F_{1}(\phi,\xi) + F_{2}(K(\phi,\xi)),
\label{sum2fun}
\end{equation}
where $X=\mathrm{H_{0}^{1}(\Omega)}\times\mathrm{H_{0}^{1}(\Omega)}$, $Y=\mathrm{H_{0}^{1}(\Omega)}\times \mathrm{L^{2}(\Omega)}$ and 
\begin{equation}
\begin{aligned}
& K : X\rightarrow Y \mbox{,  }K(\phi,\xi)=(\xi-\phi,\xi'),\\
& F_{1}: X \rightarrow\overline{\re} \mbox{, with  }F_{1}(\phi,\xi)= \mathbb{I}_{\left\{\norm{\mathrm{L}^q(\Omega)}{\cdot}\leq\beta\right\}}(\phi)+\mathbb{I}_{\left\{\norm{\infty}{\cdot}\leq\alpha\right\}}(\xi),\\
& F_{2}: Y \rightarrow \overline{\re} \mbox{, with }F_{2}(\phi,\psi)= \mathbb{I}_{\left\{0\right\}}(\phi)+\int_{\Omega}f\psi\, dx+\frac{1}{2}\int_{\Omega}\psi^{2}dx.
\end{aligned}
\label{functionals}
\end{equation}
We denote the infimum in \eqref{predual}  as $\inf\mathcal{P}^{*}$. Then, it is immediate that
 $$-\inf\mathcal{P}^{*}=-\inf_{(\phi,\xi)\in X}F_{1}(\phi,\xi) + F_{2}(K(\phi,\xi)).$$
The dual problem of \eqref{sum2fun}, see \cite{Ekeland}, is defined as
\begin{equation}
\min_{(w,u)\in Y^{*}} F_{1}^{*}(- K^{\star}(w,u))+F_{2}^{*}(w,u),
\label{dual}
\end{equation}
where $K^{\star}$ here denotes the adjoint of $K$.
Let $(\sigma,\tau)$ be elements of $\mathrm{H}_{0}^{1}(\Omega)^{*}\times \mathrm{H}_{0}^{1}(\Omega)^{*}$ acting as distributions.
 For the convex conjugate of $F_{1}$, we write
\begin{equation}
F_{1}^{*}(\sigma,\tau)  = \sup_{\substack{(\phi,\xi)\in X \\ \norm{\mathrm{L}^q(\Omega)}{\phi}\leq\beta \\ \norm{\infty}{\xi}\leq\alpha}} \scalprod{}{\sigma}{\phi}+\scalprod{}{\tau}{\xi} =
                                          \beta \sup_{\substack{\phi\in\mathrm{H}_{0}^{1}(\Omega)\\ \norm{\mathrm{L}^q(\Omega)}{\phi}\leq1}} \scalprod{}{\sigma}{\phi} + \alpha\sup_{\substack{\xi\in\mathrm{H}_{0}^{1}(\Omega)\\ \norm{\infty}{\xi}\leq1}} \scalprod{}{\tau}{\xi}.
\label{F1conj1}                          
\end{equation}
However, by standard density arguments we have:
\begin{equation}
F_{1}^{*}(\sigma,\tau)=  \beta \sup_{\substack{\phi\in C_{c}^{\infty}(\Omega)\\ \norm{\mathrm{L}^q(\Omega)}{\phi}\leq1}} \scalprod{}{\sigma}{\phi} + \alpha\sup_{\substack{\xi\in C^{\infty}_{c}(\Omega)\\ \norm{\infty}{\xi}\leq1}} \scalprod{}{\tau}{\xi}= \beta\norm{\mathrm{L}^p(\Omega)}{\sigma}+ \alpha\norm{\mathcal{M}}{\tau}.
\label{F1conj2}
\end{equation}
Moreover, let $K^{\star}:Y^{*}\rightarrow X^{*}$ with 
\begin{align*}
\scalprod{}{-K^{\star}(w,u)}{(\phi,\xi)} & =-\scalprod{}{(w,u)}{K(\phi,\xi)} =-\scalprod{}{(w,u)}{(\xi-\phi,\xi')}\\
                                                              & = -\scalprod{}{w}{\xi}+\scalprod{}{w}{\phi}-\scalprod{}{u}{\xi'} = \scalprod{}{Du-w}{\xi}+\scalprod{}{w}{\phi}.
\end{align*}
Hence, we obtain
\begin{equation}
F_{1}^{*}(-K^{\star}(w,u))= \beta\norm{\mathrm{L}^p(\Omega)}{w}+ \alpha\norm{\mathcal{M}}{Du-w},
\label{F1conj3}
\end{equation}
and 
\begin{align}
F_{2}^{*}(w,u) & =\sup_{\substack{(\phi,\psi)\in Y \\ \phi=0}}\scalprod{}{w}{\phi}+\scalprod{}{u}{\psi}-\scalprod{}{f}{\psi}-\frac{1}{2}\int_{\Omega}\psi^{2}\notag\\
                         & = \sup_{\psi\in L^{2}(\Omega)}\scalprod{}{u-f}{\psi}-\frac{1}{2}\int_{\Omega}\psi^{2}dx\notag\\
                         & := \left( \frac{1}{2}\norm{\mathrm{L}^{2}(\Omega)}{\cdot}^{2}\right)^{*}(u-f) = \frac{1}{2}\norm{\mathrm{L}^2(\Omega)}{u-f}^{2}.\label{F2conj}
\end{align}
Therefore, we have proved the following:
\begin{prop} The dual problem of \eqref{predual}  is equivalent to the \emph{primal} problem \eqref{tvlp_min} in the sense that $(w,u)\in Y^{*}$ is a solution of the dual of \eqref{predual}   if and only if $(w,u)\in  \mathrm{L}^{p}(\Omega)\times\mathrm{BV}(\Omega)$ is a solution of \eqref{tvlp_min}.
\label{equivalence}
\end{prop}
It remains to verify that we have no duality gap  between the  two minimisation problems \eqref{tvlp_min} and \eqref{predual}. The proof of the following proposition follows the proof of the corresponding proposition in \cite{Bredies}. We slightly modify for our case.
\begin{prop} Let $F_{1}, F_{2}, K$ be defined  as in \eqref{functionals}. Then
\begin{equation}
Y=\bigcup_{\lambda\geq0}\lambda(\dom F_{2} - K (\dom F_{1}))
\label{FRcond}
\end{equation}
is a closed vector space and thus \cite{attouch_brezis}
\begin{equation}
\min_{(\phi,\xi)\in X}F_{1}((\phi,\xi))+F_{2}(K(\phi,\xi)) + \min_{(w,u)\in Y^{*}} F_{1}^{*}(- K^{\star}(w,u))+F_{2}^{*}((w,u))=0.
\label{duality_gap}
\end{equation}
\begin{proof}
Let $(\phi,\psi)\in Y$ and define $\psi_{0}(x)=c_{1}+c_{2}x$, where $c_{1},c_{2}$ are constants that are uniquely determined by the following conditions
\begin{equation}
\int_{\Omega}\psi_{0}(x)\, dx=\int_{\Omega}\psi(x)\, dx\mbox{,  }\int_{\Omega}x\psi_{0}(x)\,dx=\int_{\Omega}x\psi(x)+\phi(x)\, dx.
\end{equation} 
 Let $\xi(x)=\int_{a}^{x}(\psi_{0}-\psi)(y) \,dy$. Since by construction, $\xi'=\psi_{0}-\psi\in \mathrm{L}^{2}(\Omega)$ with $\xi(a)=\xi(b)=0 $, we have that $\xi\in\mathrm{H}_{0}^{1}(\Omega)$. Furthermore, let $\phi= - (\phi+\xi)\in\mathrm{H}_{0}^{1}(\Omega)$ and $(\phi,\xi)\in X$ with
$$(\phi,\psi)=(\xi-\phi,\psi_{0}-\xi')=(0,\psi_{0})-(\xi-\phi,\xi')=(0,\psi_{0})-K(\phi,\xi).$$ 
Choosing appropriately $\lambda>0$ such that $\norm{\mathrm{L}^q(\Omega)}{\lambda^{-1}\phi}\leq\beta$, $\norm{\infty}{\lambda^{-1}\xi}\leq\alpha$, we can write 
$$(\phi,\psi)=\lambda((0,\lambda^{-1}\psi_{0})-K(\lambda^{-1}\phi,\lambda^{-1}\xi)),$$
 with $\dom F_{2} =  \{0\} \times \mathrm{L}^{2}(\Omega)$ and $ \dom F_{1} = \{(\phi,\xi): \norm{\mathrm{L}^q(\Omega)}{\phi}\leq\beta, \norm{\infty}{\xi}\leq\alpha\}$. Since $(\phi,\psi)\in Y$ were chosen arbitrarily, \eqref{FRcond} holds. 
\end{proof}
\label{geom_condition}
\end{prop} 

Since there is no duality gap, we can find a relationship between the solutions of \eqref{predual}  and \eqref{tvlp_min}  via the optimality conditions,  see \cite[Prop. 4.1 (III)]{Ekeland}.

%

\begin{thrm}[Optimality conditions] Let $1<p<\infty$ and $f\in \mathrm{L}^{2}(\Omega)$. A pair $(w,u)\in\mathrm{L}^{p}(\Omega) \times\mathrm{BV}(\Omega)$  is a solution of \eqref{tvlp_min} if and only if  there exists a function $\phi\in\mathrm{H_{0}^{1}(\Omega)}$ such that 
\begin{equation}
\begin{aligned}
& \phi' = u-f,\\
& \phi \in \alpha\Sgn(Du-w),\\
\end{aligned}
\label{conditions1}
\end{equation}
and
\begin{equation}
\begin{aligned}
\begin{cases}
  \phi\in\left\{ \tilde{\phi}\in \mathrm{L}^{q}(\Omega) : \norm{\mathrm{L}^{q}(\Omega)}{\tilde{\phi}}\leq\beta\right\} &  \mbox{ if }w=0,\\
  \phi=\beta\frac{|w|^{p-2}w}{\norm{\mathrm{L}^{p}(\Omega)}{w}^{(p-1)}} &  \mbox{ if }w\neq0.
\end{cases}
\end{aligned}
\label{conditions2}
\end{equation}

\begin{proof}
Since there is not duality gap, the optimality conditions read \cite[Prop. 4.1 (III)]{Ekeland}:
\begin{align}
(\phi,\xi) & \in \partial F_{1}^{*}(-K^{\star}(w,u)),\label{opt1}\\
K(\phi,\xi) & \in \partial F_{2}^{*} (w,u),\label{opt2}
\end{align}
for every $(\phi,\xi)$ and $(w,u)$ solutions of \eqref{predual}  and \eqref{tvlp_min} respectively. Note that in dimension one we have $\mathrm{H}_{0}^{1}(\Omega)\subseteq C_{0}(\Omega)$. Hence, for every $(\sigma,\tau)\in X^{*}$, we have the following:
\begin{align*}
& F_{1}^{*}(\sigma,\tau) \geq F_{1}^{*}(-K^{\star}(w,u))+\scalprod{}{(\sigma,\tau)+K^{\star}(w,u)}{(\phi,\xi)}\Leftrightarrow\\
& \alpha\norm{\mathcal{M}}{\tau}+\beta\norm{\mathrm{L}^{p}(\Omega)}{\sigma} \geq \alpha\norm{\mathcal{M}}{Du-w}+\beta\norm{\mathrm{L}^{p}(\Omega)}{w}+\scalprod{}{(\sigma,\tau)-(w,Du-w)}{(\phi,\xi)}\Leftrightarrow\\
& \alpha\norm{\mathcal{M}}{\tau}+\beta\norm{\mathrm{L}^{p}(\Omega)}{\sigma} \geq \alpha\norm{\mathcal{M}}{Du-w}+\beta\norm{\mathrm{L}^{p}(\Omega)}{w}+\scalprod{}{(\sigma-w)}{\phi}+\scalprod{}{\tau-(Du-w)}{\xi}\Leftrightarrow\\
& \begin{cases}
& \alpha\norm{\mathcal{M}}{\tau} \geq\alpha\norm{\mathcal{M}}{Du-w}+\scalprod{}{\tau-(Du-w)}{\xi}\mbox{, }\forall\tau\in\mathrm{H}^{1}_{0}(\Omega)^{*},\\
& \beta\norm{\mathrm{L}^{p}(\Omega)}{\sigma} \geq \beta\norm{\mathrm{L}^{p}(\Omega)}{w}+\scalprod{}{\sigma-w}{\phi}\mbox{, }\forall\sigma\in\mathrm{H}^{1}_{0}(\Omega)^{*}
\end{cases}
\qquad\Leftrightarrow\\
& \begin{cases}
& \alpha\norm{\mathcal{M}}{\tau} \geq\alpha\norm{\mathcal{M}}{Du-w}+\scalprod{}{\tau-(Du-w)}{\xi}\mbox{, }\forall\tau\in \mathcal{M}(\Omega),\\
& \beta\norm{\mathrm{L}^{p}(\Omega)}{\sigma} \geq \beta\norm{\mathrm{L}^{p}(\Omega)}{w}+\scalprod{}{\sigma-w}{\phi}\mbox{, }\forall\sigma\in\mathrm{L}^{p}(\Omega)
\end{cases}
\qquad\Leftrightarrow\\
&\begin{cases}
& \xi\in\alpha\partial\norm{\mathcal{M}}{\cdot}(Du-w),\\
& \phi\in \beta \partial \norm{\mathrm{L}^{p}(\Omega)}{\cdot}(w),
\end{cases}
\end{align*}
and using \eqref{sign2} we can simplify the last expressions with:
\begin{equation}
\xi \in \alpha\Sgn(Du-w),
\end{equation}
and
\begin{equation}
\begin{aligned}
\begin{cases}
\phi\in\{ \tilde{\phi}\in \mathrm{L}^{q}(\Omega) : \norm{\mathrm{L}^{q}(\Omega)}{\tilde{\phi}}\leq\beta\} &\text{ if }w=0,\\
\phi=\beta\frac{|w|^{p-2}w}{\norm{\mathrm{L}^{p}(\Omega)}{w}^{(p-1)}} &\text{ if } w\neq0.
\end{cases}
\label{cond_proof1}
\end{aligned}
\end{equation}
Indeed,  the $\mathrm{L}^{p}$ norm is an one homogeneous functional and thus its subdifferential  reads: 
$$\partial\norm{\mathrm{L}^{p}(\Omega)}{\cdot}(w)=\left\{ z\in (\mathrm{L}^{p}(\Omega))^{*} : \scalprod{}{z}{w}=\norm{\mathrm{L}^{p}(\Omega)}{w}\mbox{,  }\scalprod{}{z}{\sigma}\leq\norm{\mathrm{L}^{p}(\Omega)}{\sigma}\mbox{,  }\forall\sigma\in\mathrm{L}^{p}(\Omega)\right\}.$$
Clearly, for $w=0$, the above expression reduces to $\norm{\mathrm{L}^{p}(\Omega)}{\sigma}\geq\scalprod{}{z}{\sigma}\mbox{, }\forall\sigma\in \mathrm{L}^{p}(\Omega)$ which is valid for any $z\in \mathrm{L}^{q}(\Omega)$ with $\norm{\mathrm{L}^{q}(\Omega)}{z}\leq1$, i.e., the unit ball of $\mathrm{L}^{q}(\Omega)$. If $w\neq0$ then the subdifferential  reduces to the G\^ateaux derivative of  the $\mathrm{L}^{p}$ norm,  i.e., $\partial\norm{\mathrm{L}^{p}(\Omega)}{\cdot}(w)=\frac{|w|^{p-2}w}{\norm{\mathrm{L}^{p}(\Omega)}{w}^{p-1}}$.
Finally, from \eqref{opt2} we have for every $(\hat{w},\hat{u})\in Y^{*}$ 
\begin{align*}
& F_{2}^{*}(\hat{w},\hat{u})  \geq F_{2}^{*}(w,u) +\scalprod{}{K(\phi,\xi)}{((\hat{w},\hat{u})-(w,u)}\Leftrightarrow\\
& \frac{1}{2}\int_{\Omega}(f-\hat{u})^{2} \, dx \geq  \frac{1}{2}\int_{\Omega}(f-u)^{2}\, dx+\scalprod{}{(\xi-\phi,\xi')}{(\hat{w}-w,\hat{u}-u)}\Leftrightarrow\\
& \frac{1}{2}\int_{\Omega}(f-\hat{u})^{2} \, dx \geq  \frac{1}{2}\int_{\Omega}(f-u)^{2}\, dx+\scalprod{}{(\xi-\phi}{\hat{w}-w}+\scalprod{}{\xi'}{\hat{u}-u}\Leftrightarrow\\
& \begin{cases} 
& \scalprod{}{\xi-\phi}{\hat{w}-w} \leq0\mbox{, }\forall\hat{w}\in (\mathrm{H}^{1}_{0}(\Omega))^{*},\\
& \frac{1}{2}\int_{\Omega}(f-u)^{2}+\scalprod{}{\xi'}{\hat{u}-\hat{u}}\leq\frac{1}{2}\int_{\Omega}(f-\hat{u})^{2}\mbox{, }\forall\hat{u}\in (\mathrm{L}^{2}(\Omega))^{*}\\
\end{cases}
\qquad\Leftrightarrow\\
& \begin{cases} 
& \xi = \phi, \\
& \xi' \in\partial \left(\frac{1}{2}\norm{2}{f-\cdot}^{2}\right)(u) = u -f .
\end{cases}
\end{align*}

Combining all the above results, we obtain the optimality conditions \eqref{conditions1} and \eqref{conditions2}.
\end{proof}
\end{thrm}

\begin{rem}
We observe that if $w=0$ then the conditions \eqref{conditions1} coincide with the optimality conditions for the
 $\mathrm{L}^{2}$--$\mathrm{TV}$ minimisation problem (ROF) with parameter $\alpha$, i.e.,
\begin{equation}
\min_{u\in\mathrm{BV(\Omega)}}\frac{1}{2}\norm{2}{f-u}^2+\alpha\norm{\mathcal{M}}{Du},
\label{rof}
\end{equation}
see also \cite{Ring}. On the other hand when $w\neq0$, the additional condition \eqref{conditions2} depends on the value of $p$ and as we will see later it allows a certain degree of smoothness in the final solution $u$. 
\end{rem}

\subsection{Structure of the solutions}\label{sec:structure}
\label{section_stru}

The optimality conditions \eqref{conditions1} and \eqref{conditions2} are an important asset since we  can determine exactly the structure of the solutions for the problem \eqref{tvlp_min} as this is determined by  the regularising parameters $\alpha, \beta$ and the value of $p$. 

We initially discuss the cases where the solution $u$ of $\eqref{tvlp_min}$ is a solution of a corresponding ROF minimisation problem i.e., $w=0$.

\begin{prop}[ROF-solutions] Let $q$ be the conjugate exponent of $p\in (1,\infty]$ as this is defined in \eqref{conj_expon}. If
\begin{equation}
\frac{\beta}{\alpha}\geq |\Omega|^{\frac{1}{q}},
\label{cond1}
\end{equation}
 then the solution $u$ of \eqref{tvlp_min} coincides with the solution of the ROF minimisation problem \eqref{rof} and $w=0$.
\begin{proof}
Let $(w^{*},u^{*})$ be a solution pair for ($\mathcal{P}$), then for every $(w,u)\in \mathrm{L}^{p}(\Omega)\times\mathrm{BV}(\Omega)$, 
$$\frac{1}{2}\norm{\mathrm{L}^2(\Omega)}{f-u^{*}}^{2}+\alpha\norm{\mathcal{M}}{Du^{*}-w^{*}}+\beta\norm{\mathrm{L}^p(\Omega)}{w^{*}}\leq \frac{1}{2}\norm{\mathrm{L}^2(\Omega)}{f-u}^{2}+\alpha\norm{\mathcal{M}}{Du-w}+\beta\norm{\mathrm{L}^p(\Omega)}{w}.$$
 Setting $w=0$, we get 
\begin{equation}
\frac{1}{2}\norm{\mathrm{L}^2(\Omega)}{f-u^{*}}^{2}+\alpha\norm{\mathcal{M}}{Du^{*}-w^{*}}+\beta\norm{\mathrm{L}^p(\Omega)}{w^{*}}\leq \frac{1}{2}\norm{\mathrm{L}^2(\Omega)}{f-u}^{2}+\alpha\norm{\mathcal{M}}{Du}.
\label{rof_type1}
\end{equation}
Since $w\in\mathrm{L}^{p}(\Omega)$ with $p\in(1,\infty]$, using the inequality \eqref{Lp_inclusion} we have that $$\norm{\mathcal{M}}{w^{*}}=\norm{\mathrm{L}^{1}(\Omega)}{w^{*}}\leq|\Omega|^{1-\frac{1}{p}}\norm{\mathrm{L}^p(\Omega)}{w^{*}},$$ 
and using the condition \eqref{cond1} we get
\begin{align}
\alpha\norm{\mathcal{M}}{Du^{*}} & \leq\alpha\norm{\mathcal{M}}{Du^{*}-w^{*}}+\alpha\norm{\mathcal{M}}{w^{*}}\notag\\
                             & \leq\alpha\norm{\mathcal{M}}{Du^{*}-w^{*}}+\beta|\Omega|^{\frac{1}{p}-1}\norm{\mathcal{M}}{w^{*}}\notag\\ 
														 & \leq\alpha\norm{\mathcal{M}}{Du^{*}-w^{*}}+\beta\norm{\mathrm{L}^p(\Omega)}{w^{*}}.\label{rof_type2}
\end{align}

From \eqref{rof_type1} and \eqref{rof_type2} we conclude that for every $u\in\mathrm{BV(\Omega)}$, 
$$\frac{1}{2}\norm{\mathrm{L^2(\Omega)}}{f-u^{*}}^{2}+\alpha\norm{\mathcal{M}}{Du^{*}}\leq\frac{1}{2}\norm{\mathrm{L^2(\Omega)}}{f-u}^{2}+\alpha\norm{\mathcal{M}}{Du},$$
 i.e., $u^{*}$ solves
  $$\min_{u\in\mathrm{BV(\Omega)}}\frac{1}{2}\norm{\mathrm{L}^2(\Omega)}{f-u}^{2}+\alpha\norm{\mathcal{M}}{Du}=\min_{\substack{u\in \mathrm{BV}(\Omega)\\w=0}} \frac{1}{2}\norm{\mathrm{L^2(\Omega)}}{f-u}^{2}+\alpha\norm{\mathcal{M}}{Du-w}+\beta\norm{\mathrm{L}^p(\Omega)}{w}.$$
\end{proof}
\label{equivalent_ROF}
\end{prop}
In fact what we have essentially proved above is that when the condition \eqref{cond1} holds then
\[\mathrm{TVL}_{\alpha,\beta}^{p}(u)=\alpha \|Du\|_{\mathcal{M}},\quad \forall u \in\mathrm{BV}(\Omega).\]
Notice also that when \eqref{cond1} holds then we can show that $w=0$ is an admissible solution but in general we cannot prove that this solution is unique.

The condition \eqref{cond1} is valid for any dimension $d\geq1$. It provides a rough threshold for obtaining ROF type solutions in terms of the regularising parameters $\alpha,\beta$ and the image domain $\Omega$. However, the condition is not sharp  since as we will see  in the following sections we can obtain a sharper estimate for specific data $f$. \par

The following proposition in the spirit of \cite{Bredies, papafitsoros} gives more insight into the structure of solutions of \eqref{tvlp_min}.
 
\begin{prop} Let $f\in\mathrm{BV(\Omega)}$ and suppose that $(w,u)\in \mathrm{L}^{p}(\Omega)\times\mathrm{BV}(\Omega)$ is a solution pair for \eqref{tvlp_min} with $p\in(1,\infty)$. Suppose that $u>f$ (or $u<f$ ) on an open interval $I\subset\Omega$ then $(Du-w)\lfloor I = 0$ i.e., $u'=w$ on $I$ and $|D^{s}u|(I)=0$.
\label{prop_struct}
\end{prop}

The above proposition is formulated rigorously via the use of \emph{precise representatives} of $\mathrm{BV}$ functions, see \cite{Ambrosio}, but  for the sake of simplicity we rather not get into the details here. Instead we refer the reader to \cite{Bredies, papafitsoros} where the analogue propositions are shown for the $\mathrm{TGV}$ regularised solutions and whose proofs are similar to the one of Proposition \ref{prop_struct}. 

We now consider the case where the solution is constant in $\Omega$, which in fact coincides with the mean value $\tilde{f}$ of the data f:
\begin{equation}
\tilde{f} := \underset{\mbox{$u$ constant}}{\operatorname{argmin}}\;\frac{1}{2}\norm{\mathrm{L}^2(\Omega)}{f-u}^{2}=\frac{1}{|\Omega|}\int_{\Omega}f\,dx.
\label{L2_regr}
\end{equation}

\begin{prop}[Mean value solution] If the following conditions hold
\begin{equation}
\begin{aligned}
& \alpha\geq \|f-\tilde{f}\|_{\mathrm{L}^1(\Omega)},\\
& \beta\geq|\Omega|^{\frac{1}{q}}\|f-\tilde{f}\|_{\mathrm{L}^1(\Omega)},
\end{aligned}
\label{L2_regr_cond}
\end{equation}
then the solution of \eqref{tvlp_min} is constant and  it is equal to $\tilde{f}$.
\begin{proof} Clearly, if $u$ is a constant solution of \eqref{tvlp_min}, then $Du=0$ and from  \eqref{tvlp_ineq} we get $\mathrm{TVL}_{\alpha,\beta}^{p}(u)=0$. Hence, we have $u=\tilde{f}$.  In order to have $u=\tilde{f}$, from the optimality conditions \eqref{conditions1} and \eqref{conditions2}, it suffices to find a function $\phi\in \mathrm{H}^{1}_{0}(\Omega)$ such that $\phi(a)=\phi(b)=0$ and
 $$\phi'=f-\tilde{f},\quad\norm{\infty}{\phi}\leq\alpha,\quad\norm{\mathrm{L}^q(\Omega)}{\phi}\leq\beta.$$
Letting $\phi(x)=\int_{a}^{x} (f(s)-\tilde{f}) \,ds$, then obviously $\phi(a)=\phi(b)=0$ and
 $$|\phi(x)| \leq \int_{a}^{x} |f(s)-\tilde{f}| \,ds \leq \|f-\tilde{f}\|_{\mathrm{L}^1(\Omega)}<\infty.$$
  Therefore, $\norm{\infty}{\phi}\leq\|f-\tilde{f}\|_{\mathrm{L}^1(\Omega)}$. Also, since $\mathrm{L}^{\infty}(\Omega)\subset \mathrm{L}^{q}(\Omega)$ we obtain 
   $$\norm{\mathrm{L}^{q}(\Omega)}{\phi}\leq |\Omega|^{\frac{1}{q}}\norm{\infty}{v}\leq  |\Omega|^{\frac{1}{q}}\|f-\tilde{f}\|_{\mathrm{L}^1(\Omega)}.$$ 
 Hence, it suffices to choose $\alpha$ and $\beta$ as in \eqref{L2_regr_cond}. 
\end{proof}
\end{prop}

\begin{figure}[h]
 \includegraphics[scale=0.33]{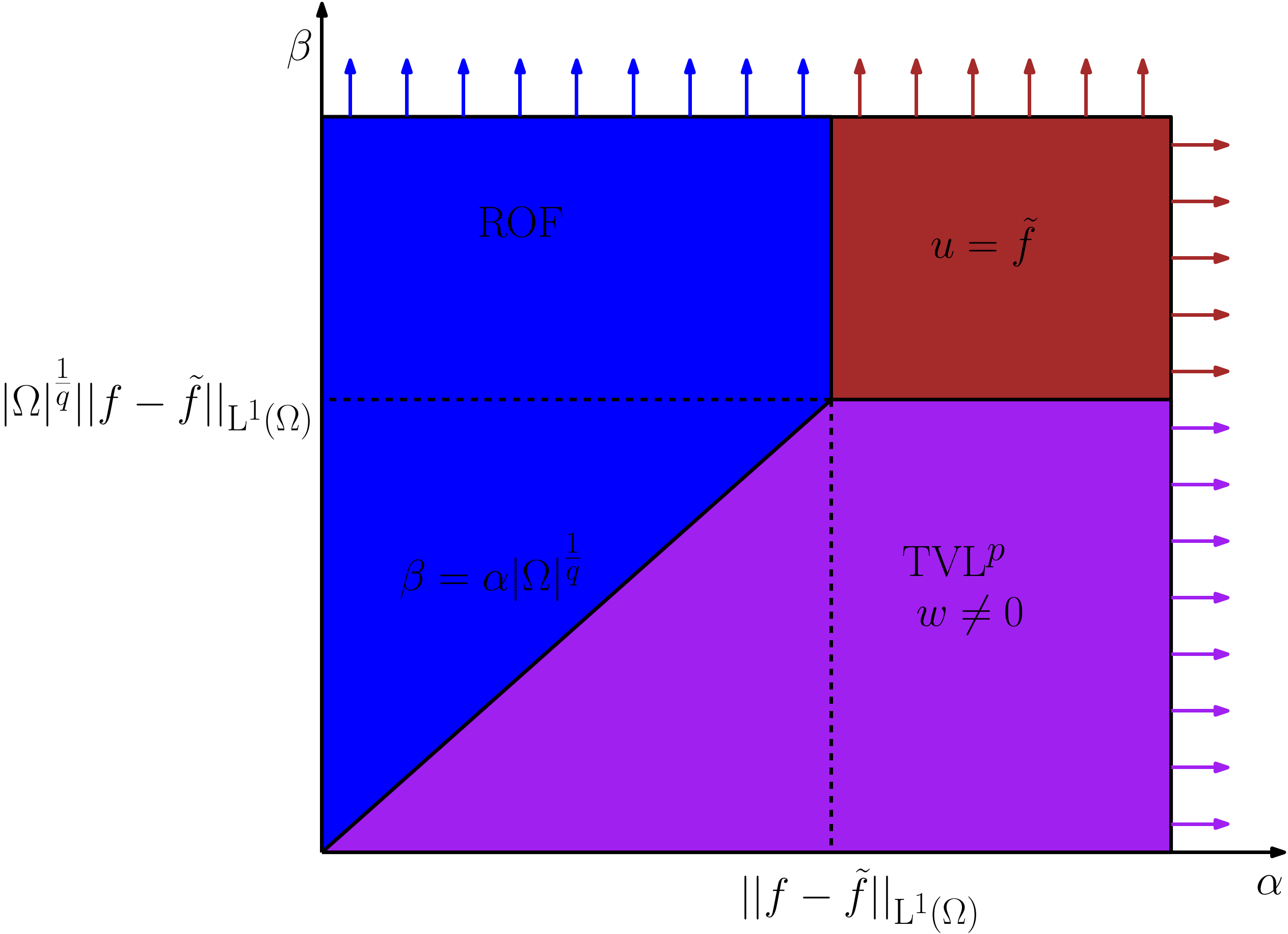}
 \caption{ Characterisation of solutions of \eqref{tvlp_min} for any data $f$: The blue/red areas correspond to the ROF type solutions ($w=0$) and the purple area corresponds to the $\mathrm{TVL}^{p}$ solutions ($w\neq0)$ for $1<p<\infty$. We note that the red/purple areas are potentially larger/smaller as the conditions we have derived are not sharp.}
 \label{TVS_p_graph}
\end{figure}

In Figure \ref{TVS_p_graph}, we summarise our results so far. There, we have partitioned the set $\{\alpha>0,\beta>0\}$ into different areas that correspond to different types of solutions of the problem \eqref{tvlp_min}. 
 The brown area, arising from thresholds \eqref{L2_regr_cond} corresponds to the choices of $\alpha$ and $\beta$ that produce constant solutions while the blue area corresponds to ROF type solutions, according to threshold \eqref{cond1}. 
 Therefore, we can determine the area where the non-trivial solutions are obtained i.e., $w\neq0$, see purple region. Note that since the conditions \eqref{cond1} and \eqref{L2_regr_cond} are not sharp the red  and the purple areas are potentially larger or smaller respectively  than it is shown in Figure \ref{TVS_p_graph}. 
 
 The following proposition reveals more information about the structure of solutions in the case $w
 \ne 0$.

\begin{prop}[$\mathrm{TVL}^p$-solutions]Let $f\in\mathrm{BV(\Omega)}$ and suppose that $(w,u)\in\mathrm{L}^{p}(\Omega)\times\mathrm{BV}(\Omega)$ is a solution pair for \eqref{tvlp_min} with $p\in(1,\infty)$ and $w\neq0$. Suppose that $u>f$ (or $u<f$) on an open interval $I\subset\Omega$ then the solution $u$ of \eqref{tvlp_min} is obtained by 
\begin{equation}
-C(|u'(x)|^{p-2}u'(x))'+u(x)=f(x),\quad \forall x\in I \quad\mbox{where } C=\frac{\beta}{\norm{\mathrm{L}^{p}(\Omega)}{w}^{p-1}}.
\label{pLaplacian}
\end{equation}
\begin{proof} Since $1<p<\infty$, $w\neq0$ using Proposition \ref{prop_struct} and the second optimality condition of \eqref{conditions2}, we have that
 $$\phi=\beta\frac{|u'|^{p-2}u'}{\norm{\mathrm{L}^{p}(\Omega)}{w}^{p-1}}.$$ 
 Hence, by \eqref{conditions1} we obtain \eqref{pLaplacian} where $C=\frac{\beta}{\norm{\mathrm{L}^{p}(\Omega)}{w}^{p-1}}$.  
\end{proof}
\label{tvlp_1_hom}
\end{prop}

\subsection{Exact solutions of \eqref{tvlp_min} for a  step function}\label{sec:exact}
\label{section_exact}

In what follows we compute explicit solutions of the $\mathrm{TVL}^{p}$ denoising model \eqref{tvlp_min} for the case $p=2$ for a simple data function.
We define the step function in $\Omega=(-L,L)$, $L>0$ as: 
\begin{equation}
f(x)=
\begin{cases}
0&\text{ if   }x\in(-L,0],\\
h&\text{ if }x\in(0,L).
\end{cases}
\label{step_function}
\end{equation}
We first investigate conditions under which we obtain ROF type solutions, that is $w=0$.
\subsubsection{ROF type solutions}
We are initially interested in solutions that respect the discontinuity at $x=0$ and are piecewise constant. From the optimality conditions \eqref{conditions1}--\eqref{conditions2}, it suffices to find a function $v\in\mathrm{H}^{1}_{0}(\Omega)$ such that 
\begin{equation}
\phi(-L)=\phi(L)=0,\quad \norm{\infty}{\phi}\leq\alpha, \quad \phi(0)=\alpha,
\label{exact_cond1} 
\end{equation}
and it is also piecewise affine.
It is easy to see that by setting $\phi(x)=\frac{\alpha}{L}(L-|x|)$,  the conditions \eqref{exact_cond1} are satisfied and the solution $u$ is piecewise constant. The first condition of \eqref{conditions2} implies that $\norm{\mathrm{L}^q(\Omega)}{\phi}\leq\beta\Leftrightarrow \frac{\beta}{\alpha}\geq(\frac{2L}{q+1})^{\frac{1}{q}}$  and provides  a necessary and sufficient condition that need to be fulfilled in order for $u$ to be piecewise constant, that is to say 
\begin{equation}
u(x)=
\begin{cases}
\frac{\alpha}{L}& \mbox{ if  }x\in(-L,0]\\
h-\frac{\alpha}{L}& \mbox{ if  }x\in(0,L)
\end{cases}
\quad \Leftrightarrow \quad \frac{\beta}{\alpha}\geq\left(\frac{2L}{q+1}\right)^{\frac{1}{q}}.
\label{pc_cond}
\end{equation}
A special case of the ROF-type solution is when $u$ is constant, i.e., when  $u=\tilde{f}$, the mean value of $f$.

 We define $\phi(x)=\frac{h}{2}(L-|x|)$ and in that case we have that $\norm{\infty}{\phi}\leq\alpha\Leftrightarrow\alpha\geq\frac{hL}{2}$ and $\norm{\mathrm{L}^q(\Omega)}{\phi}\leq\beta\Leftrightarrow \beta\geq\frac{h}{2}(\frac{2L^{q+1}}{q+1})^{\frac{1}{q}}$. This implies that 
\begin{equation}
u=\tilde{f}=\frac{h}{2}\quad\Leftrightarrow \quad
\alpha\geq\frac{hL}{2}\quad\mbox{and}\quad  \beta\geq\frac{h}{2}\left(\frac{2L^{q+1}}{q+1}\right)^{\frac{1}{q}}\quad \text{with}\quad\frac{1}{p}+\frac{1}{q}=1.
\label{c_cond}
\end{equation}
Using now  \eqref{pc_cond}--\eqref{c_cond} we can draw the  exact regions in the quadrant  of $\{\alpha>,\beta>0\}$ that correspond to these two types of solutions, see the left graph in Figure \ref{tvl2_graph_1_hom} for the special case $p=2$. Notice that in these regions $w=0$ and the estimates are valid for any $p\in(1,\infty)$. 

\subsubsection{$\mathrm{TVL}^2$ type solutions}

For simplicity reasons, we examine here only the case  $p=2$ with $w\neq0$ in $\Omega$. However, we refer the reader to  Section \ref{sec:numerics:1d} where we compute numerically solutions for $p\ne 2$. Using Proposition \ref{tvlp_1_hom}, we observe that the solution is given by the following second order differential equation:
\begin{equation}
-Cu''(x)-u(x)=f(x),\quad\mbox{subject to}\quad C=\frac{\beta}{\norm{\mathrm{L}^2(\Omega)}{w}}.
\label{tvl2_sol}
\end{equation}
Even though we can tell that the solution of \eqref{tvl2_sol} has an exponential form, the fact that the constraint on $C$ depends on the solution $w$, creates a difficult computation in order to recover $u$ analytically. In order to overcome this obstacle, we consider the one dimensional version of the  $2$-homogeneous analogue of \eqref{tvlp_min} that was introduced in Section \ref{sec:phomo}:
\begin{equation}
\min_{\substack{u\in\mathrm{BV}(\Omega)\\w\in\mathrm{L}^2(\Omega)}}\frac{1}{2}\norm{\mathrm{L}^2(\Omega)}{f-u}^2+\alpha\norm{\mathcal{M}}{Du-w}+\frac{\beta_{2-hom}}{2}\norm{\mathrm{L}^{2}(\Omega)}{w}^2.
\label{tvlp_min_p_hom}
\end{equation}
Similarly to Section \ref{sec_opt_cond}, one can derive the optimality conditions for \eqref{tvlp_min_p_hom}. A pair $(w, u)$ is a solution of \eqref{tvlp_min_p_hom} if and only if 
there exists a function $\phi\in\mathrm{H}^{1}_{0}(\Omega)$ such that
\begin{equation}
\begin{aligned}
\phi' & =u-f, \\
\phi & \in\alpha\Sgn(Du-w),\\
\phi & =\beta_{2-hom}w.
\end{aligned}
\label{opt_cond_p_hom}
\end{equation}

In order to recover analytically the solutions of \eqref{tvlp_min} for $p=2$ and determine the purple region in Figure \ref{TVS_p_graph} it suffices to solve the equivalent model \eqref{tvlp_min_p_hom} where $w\neq0$. We may restrict our computations only on $I=(-L,0]\subset\Omega$ and due to symmetry the solution in $I=(0,L)$ is given by $u(x)+u(-x)=h$. The optimality condition \eqref{opt_cond_p_hom}  results to
\begin{equation}
-u''(x)+ku(x)=0,\quad \mbox{where} \quad k^2=\frac{1}{\beta}\quad \mbox{and}\quad x\in I=(-L,0]\subset\Omega.
\end{equation}
Then, we get $u(x)=c_{1}e^{kx}+c_{2}e^{-kx}$ with $\phi(x)=\frac{c_{1}}{k}e^{kx}-\frac{c_{2}}{k}e^{-kx}+c_{3}$ for all $x\in(-L,0]$. Firstly, we examine solutions that are continuous  which due to symmetry much have the value $\frac{h}{2}$ at the $x=0$, i.e., $u(0)=\frac{h}{2}$. Since $\phi\in\mathrm{H}^{1}_{0}(-L,L)$, we have $\phi(-L)=0$ and also $u'(-L)=0$. Finally, we  require that $\phi(0)<\alpha$. After some computations, we conclude that 
\begin{equation}
u(x)=
\begin{cases}
c_{1}e^{kx}+c_{2}e^{-kx} &\mbox{ if }x\in(-L,0],\\
h-c_{1}e^{-kx}-c_{2}e^{kx}&\mbox{ if }x\in(0,L)
\end{cases}\quad
\Leftrightarrow \quad\frac{\tanh(kL)}{k}<\frac{2\alpha}{h},
\label{cont_sol}
\end{equation}
where  $c_{1}=c_{2}e^{2kL}$, $c_{2}=\frac{h}{2(e^{2kL}+1)}$ and  $k=\frac{1}{\sqrt{\beta}}$.

On the other hand, in order to get solutions that  preserve the discontinuity at $x=0$, we require the following:
\begin{equation}
\begin{aligned}
& \phi(-L)=0,\quad u'(-L)=0,\\
& u(0)<\frac{h}{2},\quad \phi(0)=\alpha.
\end{aligned}
\end{equation}
Then we get
\begin{equation}
u(x)=
\begin{cases}
c_{1}e^{kx}+c_{2}e^{-kx} &\mbox{ if }x\in(-L,0],\\
h-c_{1}e^{-kx}-c_{2}e^{kx}&\mbox{ if }x\in(0,L)
\end{cases}\quad
\Leftrightarrow \quad\frac{\tanh(kL)}{k}>\frac{2\alpha}{h},
\label{disc_sol}
\end{equation}
where  $c_{1}=c_{2}e^{2kL}$, $c_{2}=\frac{\alpha k}{e^{2kL}-1}$ and  $k=\frac{1}{\sqrt{\beta}}$. Notice that the conditions for $\alpha$ and $\beta$ in \eqref{cont_sol} and \eqref{disc_sol} are supplementary and thus only these type of solutions can occur, see the quadrant of $\{\alpha>0, \beta>0\}$ as it presented
 in Figure \ref{tvl2_graph_2_hom}. Letting $g(\beta)=\sqrt{\beta}\tanh{(\frac{L}{\sqrt{\beta}})}$, if $g(\beta)<\frac{2\alpha}{h}$ then the solution is of the form \eqref{cont_sol}, see the blue region in Figure \ref{tvl2_graph_2_hom}. On the other hand in the complementary green region we obtain the solution \eqref{disc_sol}. For extreme cases where $\beta\rightarrow\infty$, i.e., $k\rightarrow0$ we obtain $\frac{\tanh(kL)}{k}\rightarrow L$, which means that there is an asymptote of $g$ at $\alpha=\frac{hL}{2}$. 
Although, we know the form of the inverse function of the hyperbolic tangent, we cannot compute analytically the inverse $f^{-1}$. However, we can obtain an approximation using a Taylor expansion which leads to 
\begin{equation}
\sqrt{\beta}\tanh{\left(\frac{L}{\sqrt{\beta}}\right)}=L - \frac{L^3}{3\beta} + \mathcal{O}\left(\frac{1}{\beta^2}\right)=\frac{2\alpha}{h}\Leftrightarrow \beta=\frac{hL^3}{3(hL-2\alpha)},
\end{equation}
where $\alpha>0$ and $\alpha\neq\frac{hL}{2}$.

Finally, we would like to describe the solution on the limiting case  $\beta\rightarrow\infty$. Letting $\beta\rightarrow\infty$ in \eqref{cont_sol}, we have that $c_{1},c_{2}\rightarrow\frac{h}{2}$ and $u(x)\rightarrow\frac{h}{2}$ for every $x\in\Omega$, which in fact is the mean value obtained from \eqref{tvlp_min}. For the discontinuous solutions, we have that $c_{1},c_{2}\rightarrow\frac{\alpha}{2L}$ and 
$$u(x)\rightarrow
\begin{cases}
\frac{\alpha}{L}&\mbox{ if  }x\in(-L,0],\\
h-\frac{\alpha}{L}&\mbox{ if  }x\in(0,L),\\
\end{cases}
$$
i.e., we converge to the solution \eqref{pc_cond}. We also get that 
\begin{equation}
w(x)=kc_{2}
\begin{cases}
e^{2kL+kx}-e^{-kx}&\mbox{ if }x\in(-L,0],\\
e^{2kL-kx}-e^{kx}&\mbox{ if }x\in(0,L],\\
\end{cases}
\end{equation}
with $\norm{\mathrm{L}^2(\Omega)}{w}=c_{2}k\sqrt{2}e^{kL}(\sinh(2kL)-2kL)^{\frac{1}{2}}$ and $c_{2}$ is given either from \eqref{cont_sol} or \eqref{disc_sol}. Then, in both cases we have  $w\rightarrow0$ as $k\rightarrow0$. Observe that the product of $\beta_{2-hom}\norm{\mathrm{L^2}(\Omega)}{w}$ is bounded as $\beta_{2-hom}\rightarrow\infty$ for both types of solutions and in fact corresponds to the bounds found in \eqref{pc_cond} and \eqref{c_cond}. Indeed, since 
$$\frac{(\sinh(2kL)-2kL)^{\frac{1}{2}}}{k^{\frac{3}{2}}}\rightarrow2\sqrt{\frac{L^3}{3}},\quad\mbox{as }k\rightarrow0,$$ 
if $\alpha>\frac{hL}{2}$ then
$$\beta_{2-hom}\norm{\mathrm{L^2}(\Omega)}{w}\rightarrow\frac{h}{2}\sqrt{\frac{2L^3}{3}},\quad\mbox{as }\beta_{2-hom}\rightarrow\infty,$$
 while if $\alpha\le\frac{hL}{2}$
$$\beta_{2-hom}\norm{\mathrm{L^2}(\Omega)}{w}\rightarrow\alpha\sqrt{\frac{2L}{3}},\quad\mbox{as }\beta_{2-hom}\rightarrow\infty.$$ 

The last result is yet another verification of Theorem \ref{kostas_phomo_1homo} and it shows that there is an 
one to one correspondence,  $\beta_{2-hom}\norm{\mathrm{L^{2}(\Omega)}}{w}\leftrightarrow\beta_{1-hom}$ and the purple region of Figure \ref{tvl2_graph_1_hom} is characterised by the solutions obtained in Figure \ref{tvl2_graph_2_hom}.

%

\begin{figure}
 \includegraphics[scale=0.33]{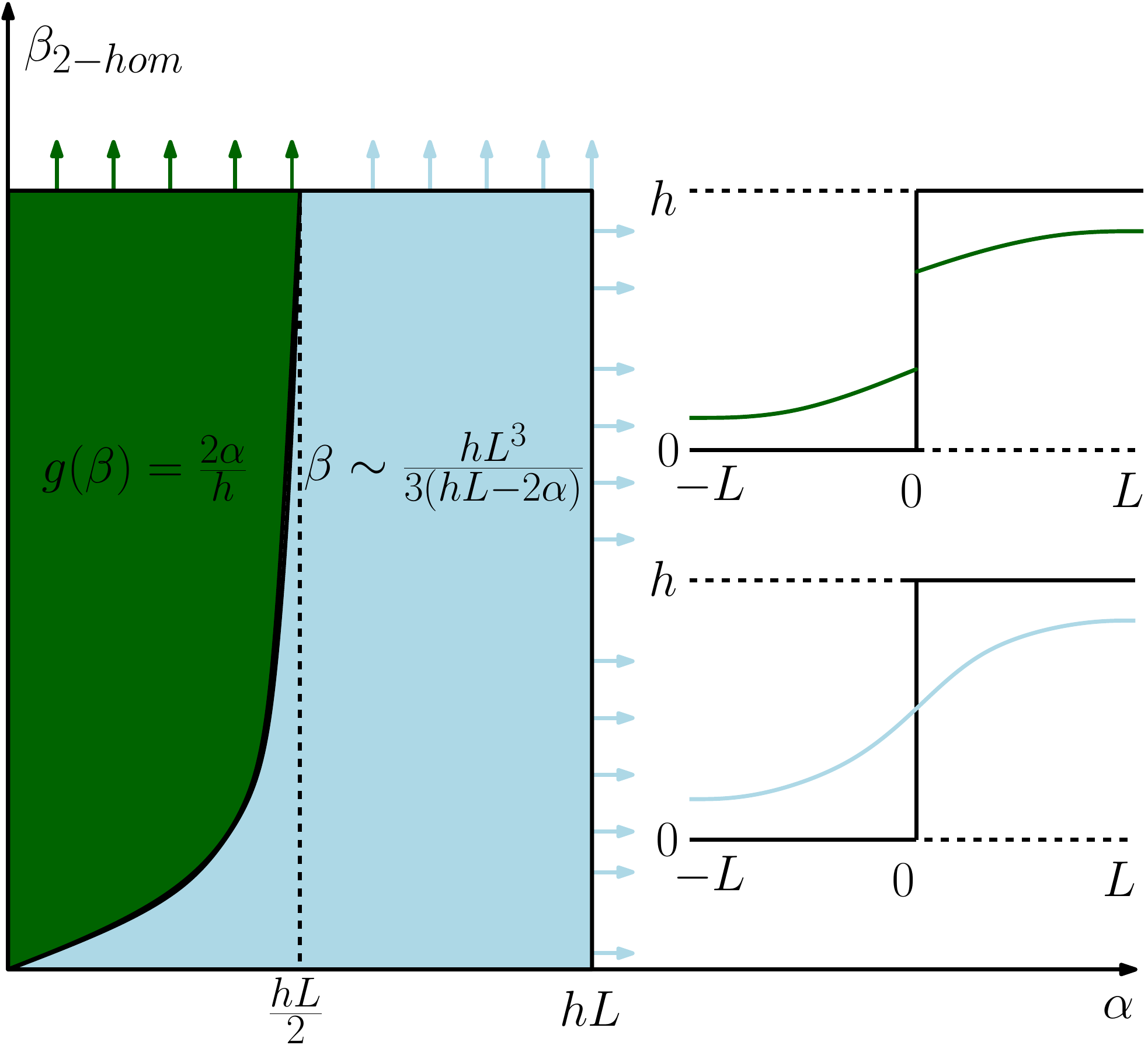}
 \caption{
Characterisation of solutions of \eqref{tvlp_min_p_hom} for data $f$ being a step function. The green region corresponds to solutions that preserve the discontinuity at $x=0$, \eqref{disc_sol}, while the blue region corresponds to continuous solutions, \eqref{cont_sol}, both having an exponential form.}
 \label{tvl2_graph_2_hom}
\end{figure}

\begin{figure}
 \includegraphics[scale=0.33]{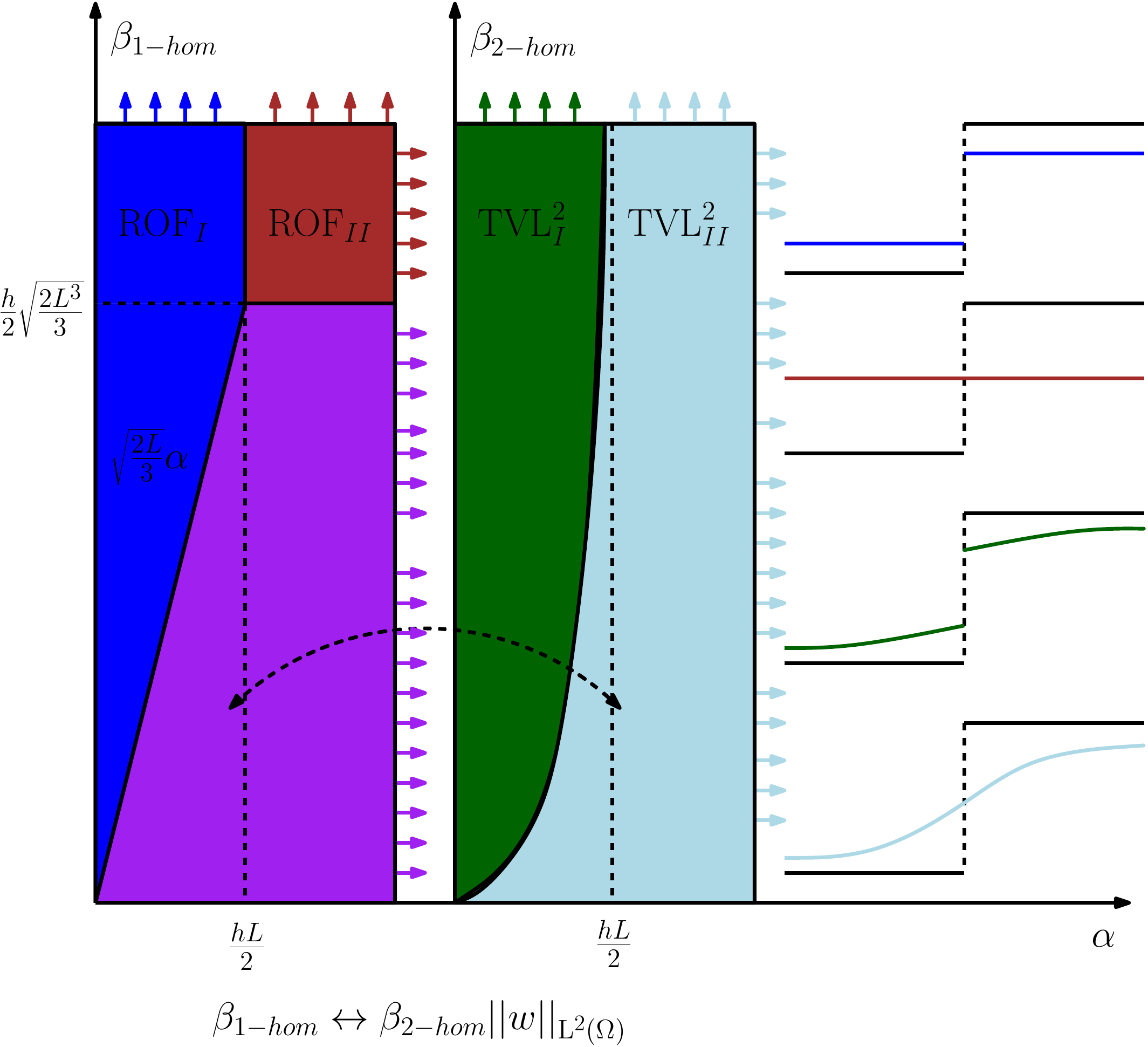}
 \caption{Characterisation of solutions of \eqref{tvlp_min} for $p=2$ for data $f$ being a step function. The type of solutions in the purple region of the left graph are exactly the solutions obtained for the 2-homogenous problem \eqref{tvlp_min_p_hom}, on the right graph.}
 \label{tvl2_graph_1_hom}
\end{figure}

\section{An image decomposition approach}
\label{sec:infimal}
In this section, we present another formulation for the problem \eqref{tvlp_min}, where we decompose an image into a $\mathrm{BV}$ part (piecewise constant) and a part that belongs to $W^{1,p}(\Omega)$ (smooth). Let $1<p\leq\infty$ and $\Omega\subset\re^{d}$ and consider the following minimisation
 problem:
\begin{equation}
\min_{\substack{u\in \mathrm{BV}(\Omega)\\ v\in \mathrm{W}^{1,p}(\Omega)}} L(u,v):= \frac{1}{2}\norm{\mathrm{L}^{2}(\Omega)}{f-u-v}^2+\alpha\norm{\mathcal{M}}{Du}+\beta\norm{\mathrm{L}^{p}(\Omega)}{\nabla v}.
\label{inf_conv}
\end{equation}
In this way, we can decompose our image into two geometric components. The second term captures the piecewise constant structures in the image, whereas the third term captures the smoothness that depends on the value of $p$. In the one dimensional setting, we can prove that the problems \eqref{tvlp_min} and \eqref{inf_conv} are 
equivalent.
\begin{prop} Let $\Omega=(a,b)\subset\re$, then a pair $(v^{\ast},u^{\ast})\in \mathrm{W}^{1,p}(\Omega)\times\mathrm{BV}(\Omega) $ is a solution of \eqref{inf_conv} if and only if $(\nabla v^{\ast}, u^{\ast}+v^{\ast})\in \mathrm{L}^p(\Omega)\times\mathrm{BV}(\Omega)$ is a solution of \eqref{tvlp_min}.
\begin{proof}
Let $\overline{u}=u+v$ then, we have the following
\begin{align*}
(v^{\ast},u^{\ast}) & \in\underset{\substack{u\in \mathrm{BV}(\Omega) \\ v\in\mathrm{W}^{1,p}(\Omega)}}{\operatorname{argmin}}\;\frac{1}{2}\norm{\mathrm{L}^{2}(\Omega)}{f-u-v}^{2}+\alpha\norm{\mathcal{M}}{Du}+\beta\norm{\mathrm{L}^{p}(\Omega)}{\nabla v}\Leftrightarrow
\end{align*}
\begin{align*}
(v^{\ast},u^{\ast}) & \in\underset{\substack{u\in \mathrm{BV}(\Omega) \\ v\in\mathrm{W}^{1,p}(\Omega)}}{\operatorname{argmin}}\;\frac{1}{2}\norm{\mathrm{L}^{2}(\Omega)}{f-u-v}^{2}+\alpha \sup_{\substack{\phi\in C_{c}^{\infty}(\Omega)\\\norm{\infty}{\phi}\leq1}}\left\{\scalprod{}{u}{\phi'}\right\}+\beta\norm{\mathrm{L}^{p}(\Omega)}{\nabla v}\Leftrightarrow\\
(v^{\ast},\overline{u^{\ast}}) & \in\underset{\substack{\overline{u}\in \mathrm{BV}(\Omega) \\ v\in\mathrm{W}^{1,p}(\Omega)}}{\operatorname{argmin}}\;\frac{1}{2}\norm{\mathrm{L}^{2}(\Omega)}{f-\overline{u}}^{2}\alpha \sup_{\substack{\phi\in C_{c}^{\infty}(\Omega)\\\norm{\infty}{\phi}\leq1}}\left\{\scalprod{}{\overline{u}}{\phi'}+\scalprod{}{\nabla v}{\phi}\right\}+\beta\norm{\mathrm{L}^{p}(\Omega)}{\nabla v}\Leftrightarrow\\
(w^{\ast},\overline{u^{\ast}}) & \in \underset{\substack{\overline{u}\in \mathrm{BV}(\Omega) \\ w =\nabla v\\ v\in\mathrm{W}^{1,p}(\Omega)}}{\operatorname{argmin}}\;\frac{1}{2}\norm{\mathrm{L}^{2}(\Omega)}{f-\overline{u}}^{2}+\alpha\norm{\mathcal{M}}{D\overline{u} - w}+\beta\norm{\mathrm{L}^{p}(\Omega)}{w}.
\end{align*}
However, we can eliminate the last constraint since
\begin{equation}
\left\{w\in\mathrm{L}^{p}(\Omega): \;\exists v\in\mathrm{W}^{1,p}(\Omega),\; w=\nabla v\right\}=\mathrm{L}^{p}(\Omega).
\label{inf_conv_cons}
\end{equation}
Indeed, let $w\in \mathrm{L}^{p}(\Omega)\subset \mathrm{L}^{1}(\Omega)$ for $p\in(1,\infty]$ and define $v(x)=\int_{a}^{x}w(s)\,ds$ for $x\in\Omega\subset\re$. Clearly,  $v'=w$ a.e and by Jensen's inequality 
$$|v(x)|^{p}=\left|\int_{a}^{x}w(s)\,ds\right|^{p}\leq\int_{a}^{x}|w(s)|^{p}\,ds<\infty,$$
 and $v\in\mathrm{W}^{1,p}(\Omega)$ for $p\in(1,\infty)$.  Finally, for the case $p=\infty$, let  $C>0$ be a constant such that $|w(x)|\leq C$ a.e. on $\Omega$. In that case we have $|v(x)|\leq\int_{a}^{x}|w(s)|\,ds\leq C|\Omega|<\infty$ and
  $$|v(y)-v(x)|\leq\int_{x}^{y}|w(s)|\,ds\leq C |y-x|, \quad \forall x,y \in\Omega,$$
   i.e., $v\in\mathrm{W}^{1,\infty}(\Omega)$, from Rademacher's theorem. Therefore,
   $$(w^{\ast},\overline{u^{\ast}})  \in \underset{\substack{\overline{u}\in \mathrm{BV}(\Omega) \\ w\in\mathrm{L}^{p}(\Omega)}}{\operatorname{argmin}}\;\frac{1}{2}\norm{\mathrm{L}^{2}(\Omega)}{f-\overline{u}}^{2}\alpha\norm{\mathcal{M}}{D\overline{u} - w}+\beta\norm{\mathrm{L}^{p}(\Omega)}{w},$$
where $\overline{u^{\ast}}=u^{\ast}+v^{\ast}$ and $w^{\ast}=\nabla v^{\ast}$.
\end{proof}
\label{prop_equivalence}
\end{prop}
Even though for $d=1$ it is true that every $\mathrm{L}^{p}$ function can be written as a gradient, this is not true for higher dimensions. In fact, as we show in the following sections, this constraint is quite restrictive and for example the staircasing effect cannot always be eliminated in the denoising process, see for instance Figure \ref{infimal_2D_1}.

The existence of minimisers of \eqref{inf_conv} is shown following again the same techniques as in Theorem \ref{tvlp_exis}. Moreover, due to the strict convexity on the fidelity term of \eqref{inf_conv}, one can prove that the sum $u+v\in\mathrm{BV(\Omega)}$ is unique for a solution  $(u,v)\in\mathrm{W}^{1,p}(\Omega)\times\mathrm{BV}(\Omega)$. This result coincides with the uniqueness of \eqref{tvlp_min} problem for $u$. Finally, if $(u_{1},v_{1}), (u_{2},v_{2})$ are two minimisers of \eqref{inf_conv}, then from the convexity of $L(u,v)$ we have for $0\leq\lambda\leq1$
$$L(\lambda(u_{1},v_{1})+(1-\lambda)(u_{2},v_{2}))\leq\lambda L(u_{1},v_{1})+(1-\lambda)L(u_{2},v_{2}).$$
Since $(u_{1},v_{1}), (u_{2},v_{2})$ are both minimisers, the above inequality is in fact an equality. Since $u_{1}+v_{1}=u_{2}+v_{2}$, we obtain
\begin{align}
& \alpha|D(\lambda u_{1} + (1-\lambda)u_{2}|(\Omega) + \beta\norm{\mathrm{L}^{p}(\Omega)}{\nabla(\lambda v_{1} + (1-\lambda)v_{2})}\notag\\
& =\alpha(\lambda|Du_{1}|(\Omega)
+(1-\lambda)|Du_{2}|)+\beta(\lambda\norm{\mathrm{L}^{p}(\Omega)}{\nabla v_{1}}+(1-\lambda)\norm{\mathrm{L}^{p}(\Omega)}{\nabla v_{2}}).
\label{equality1}
\end{align}
If we assume that 
$$\norm{\mathrm{L}^{p}(\Omega)}{\nabla(\lambda v_{1} + (1-\lambda) v_{2})}<\lambda\norm{\mathrm{L}^{p}(\Omega)}{\nabla v_{1}}+(1-\lambda)\norm{\mathrm{L}^{p}(\Omega)}{\nabla v_{2}},$$
 then we contradict the equality on \eqref{equality1}. Hence, the Minkowski inequality becomes an equality which is equivalent to the existence of $\mu>0$ such that $\nabla v_{2}=\mu\nabla v_{1}$. In other words, we have proved the following proposition that was also shown in \cite{Kim} in a similar context:

\begin{prop} Let $(u_{1},v_{1}), (u_{2},v_{2})$ be two minimisers of \eqref{inf_conv}. Then 
\begin{align}
&u_{1}+v_{1} =u_{2}+v_{2},\label{sum_unique}\\
&\exists\mu>0 \mbox{ such that }\nabla v_{2} =\mu\nabla v_{1}.
\end{align}
\label{prop_uniq1}
\end{prop}

\section{Numerical Experiments}\label{sec:numerics}

In this section we present our numerical simulations for the  problem \eqref{tvlp_min}. We begin with  the one dimensional case where we verify numerically the analytical solutions obtained in Section \ref{section_exact}. We also describe the type of structures that are promoted for different values of $p$. Finally, we proceed to the two dimensional case where we focus on image denoising tasks and in particular on the elimination of the staircasing effect.

We start by defining the discretised version of problem \eqref{tvlp_min}  
\begin{equation}
\min_{u\in\re^{n\times m}}\frac{1}{2}\norm{2}{f-u}^{2}+\mathrm{TVL}_{\alpha,\beta}^{p}(u).
\label{discrete_tvlp}
\end{equation}
Here $\mathrm{TVL}_{\alpha,\beta}^{p}:\re^{n\times m}\rightarrow\re$ is defined  as 
\begin{equation}
\mathrm{TVL}_{\alpha,\beta}^{p}(u)=\underset{w\in(\re^{n\times m})^2}{\operatorname{argmin}}\alpha\norm{1}{\nabla u - w}+\beta\norm{p}{w},
\end{equation}
where for $x\in\re^{n\times m}$, we set $\norm{p}{x}=(\sum_{i,j=1}^{n,m}|x(i,j)|^{p})^{\frac{1}{p}}$ and for $x=(x_{1},x_{2})\in(\re^{n\times m})^2$ we define
\begin{equation}
\norm{p}{x}=\left(\sum_{i,j=1}^{n,m}   \left(\sqrt{ (x_{1}(i,j))^{2} + (x_{2}(i,j))^{2} }\right)^{p}\right)^{\frac{1}{p}}.\label{lpnorm_1}
\end{equation}
We denote by $\nabla=(\nabla_{1},\nabla_{2})$ the discretised gradient with forward differences and zero Neumann boundary conditions defined as
\begin{align*}
(\nabla_{1} u)_{i,j}&=
                \begin{cases}
                 \frac{u(i+1,j)-u(i,j)}{t}&\mbox{    if   }1\leq i< n,\; 1\leq j\leq m,\\
                 0              &\mbox{    if   }i=n,\; 1\leq j\leq m,
                 \end{cases}\\
(\nabla_{2} u)_{i,j}&=
                \begin{cases}
                 \frac{u(i,j+1)-u(i,j)}{t}&\mbox{    if   }1\leq i\leq n,\; 1\leq j< m,\\
                 0             &\mbox{    if   }1\leq i\leq n,\; j=m.
                 \end{cases}          
\end{align*}
where $t$ denotes the step size. The discrete version of the divergence operator is defined as the adjoint of $\nabla$. That is, for every $w=(w_{1},w_{2})\in(\re^{n\times m})^2$ and $u\in\re^{n\times m}$, we have that $\scalprod{}{-\diverg w}{u}=\scalprod{}{w}{\nabla u}$ with 
\begin{equation}
\begin{aligned}
(\diverg w)_{i,j} &=
                \begin{cases}
                 \frac{w_{1}(i,j)-w_{1}(i,j-1)}{t}&\mbox{    if   }1\leq j\leq m,\; 1\leq i\leq n,\\
                 \frac{w_{1}(i,j)}{t}             &\mbox{    if   }j=1,\; 1\leq i\leq n,\\
		   -\frac{w_{1}(i,j-1)}{t}          &\mbox{    if   }j=m,\; 1\leq i\leq n,\\
                 \end{cases}\\                
                 &+
                \begin{cases}
                 \frac{w_{2}(i,j)-w_{2}(i-1,j)}{t}&\mbox{    if   } 1<i<n,\; 1\leq j\leq m,\\
                 \frac{w_{2}(i,j)}{t}             &\mbox{    if   }i=1,\; 1\leq j\leq m,\\
		   -\frac{w_{2}(i-1,j)}{t}          &\mbox{    if   }i=m,\; 1\leq j\leq m.
                 \end{cases}
\end{aligned}
\label{divergence}
\end{equation}

We solve the minimisation problem \eqref{discrete_tvlp} in two ways. The first one is by using the CVX
optimisation package with MOSEK solver (interior point methods). This method is efficient for small--medium scale optimisation problems and thus it is a suitable choice in order to replicate  one dimensional solutions. On the other hand, we prefer to solve large scale two dimensional versions of \eqref{discrete_tvlp} with the split Bregman method \cite{Osher} which has been widely used for the fast solution of non-smooth  minimisation problems. 

\subsection{Split Bregman for L$\mathbf{^{2}}$--TVL$\mathbf{^{p}}$}\label{sec:SB}

In this section we describe how we adapt the split Bregman algorithm to our discrete model \eqref{discrete_tvlp}. Letting $z=\nabla u -w$,  the corresponding unconstrained problem becomes 

\begin{equation}
\min_{\substack{u\in\re^{n\times m}\\ w\in(\re^{n\times m})^2\\z\in(\re^{n\times m})^2 }}\frac{1}{2}\norm{2}{f-u}^{2}+ \alpha\norm{1}{z} + \beta\norm{p}{w},\quad \mbox{such that}\quad z=\nabla u - w.
\label{discrete_tvlp_p_hom_cons}
\end{equation}

Replacing the constraint, using a Lagrange multiplier $\lambda$, we obtain the following unconstrained formulation:

\begin{equation}
\min_{\substack{u\in\re^{n\times m}\\ w\in(\re^{n\times m})^2\\z\in(\re^{n\times m})^2 }}\frac{1}{2}\norm{2}{f-u}^{2}+ \alpha\norm{1}{z} + \beta\norm{p}{w} +\frac{\lambda}{2}\norm{2}{z-\nabla u + w}^{2}.
\label{discrete_tvlp_p_hom_uncons}
\end{equation}

The  Bregman iteration, see \cite{Osher1}, that corresponds to the minimisation \eqref{discrete_tvlp_p_hom_uncons} leads to  the following two step algorithm:

\begin{align}
(u^{k+1},z^{k+1},w^{k+1}) & =\underset{u,z,w}{\operatorname{argmin}} \;\frac{1}{2}\norm{2}{f-u}^{2}+ \alpha\norm{1}{z} + \beta\norm{p}{w} + \label{kostas_b1}
      \frac{\lambda}{2}\norm{2}{b^{k}-z+\nabla u - w}^{2},\\
b^{k+1} & =b^{k}+z^{k+1}-\nabla u ^{k+1} -w^{k+1}.\label{kostas_b2}
\end{align}

Since solving \eqref{kostas_b1} at once is a difficult task,  we employ a splitting technique and  minimise alternatingly for $u, z$ and $w$. This yields the split Bregman iteration for our method:
\begin{align}
u^{k+1}&=\underset{u\in\re^{n\times m}}{\operatorname{argmin}}\;\frac{1}{2}\norm{2}{f-u}^2 + \frac{\lambda}{2}\norm{2}{b^{k}+z^{k}-\nabla u + w^{k}}^{2},\label{sub_u}\\
z^{k+1}&=\underset{z\in(\re^{n\times m})^2}{\operatorname{argmin}}\;\alpha\norm{1}{z} + \frac{\lambda}{2}\norm{2}{b^{k}+z-\nabla u^{k+1} + w^{k}}^{2},\label{sub_z}\\
w^{k+1}&=\underset{w\in(\re^{n\times m})^{2}}{\operatorname{argmin}}\;\beta\norm{p}{w} + \frac{\lambda}{2}\norm{2}{b^{k}+z^{k+1}-\nabla u^{k+1} + w}^{2},\label{sub_w}\\
b^{k+1} & =b^{k}+z^{k+1}-\nabla u ^{k+1} -w^{k+1}.
\end{align}

Next, we discuss how we  solve  each of  the subproblems \eqref{sub_u}--\eqref{sub_w}. The first-order optimality condition of \eqref{sub_u} results into the following linear system: 
\begin{equation}
\underbrace{(I-\lambda\Delta)}_\text{A}u = \underbrace{f - \lambda\diverg(b^{k} + z^{k} - w^{k})}_\text{c}.
\label{th_matrix}
\end{equation}
Here $A$ is a sparse, symmetric, positive definite  and strictly diagonal dominant matrix, thus we can easily solve \eqref{th_matrix} with an iterative solver such as conjugate gradients or Gauss--Seidel.  However, due to the zero Neumann boundary conditions, the matrix $A$ can be efficiently diagonalised by the two dimensional discrete cosine transform, 
\begin{equation}
A=W_{nm}^{\intercal} D W_{nm},
\label{decomp}
\end{equation}
where here $W_{nm}$ is the discrete cosine matrix and $D=diag(\mu_{1},\cdots,\mu_{n*m})$ is the diagonal matrix of the eigenvalues of $A$. In that case, $A$ has a particular structure of a block symmetric \textit{Toeplitz-plus-Hankel} matrix with \textit{Toeplitz-plus-Hankel} blocks and one can obtain the solution of \eqref{sub_u} by three operations involving the two dimensional discrete cosine transform \cite{hansen} as follows: Firstly, we calculate the eigenvalues of $A$ by multiplying \eqref{decomp} with $e_{1}=(1,0,\cdots,0)^{\intercal}$ from both sides and using the fact that $W_{nm}^{\intercal}W_{nm}=W_{nm}W_{nm}^{\intercal}=I_{nm}$, we get
\begin{equation}
D_{i,i}=\frac{[W_{nm}Ae_{1}]_{i}}{[W_{nm}e_{1}]_{i}},\mbox{  }i=1,2,\cdots,nm.
\label{eigenvalues}
\end{equation}
Then, the  solution of \eqref{sub_u}  is computed exactly by
\begin{equation}
u=W_{nm}^{\intercal} D^{-1} W_{nm} c.
\label{sol_u1}
\end{equation}
The solution of the subproblem \eqref{sub_z} is  obtained in a closed form  via the following \emph{shrinkage operator}, see also \cite{Osher, Wang08anew}. Indeed, for $i=1,2$ we have
\begin{equation}
z_{i}^{k+1}=\mathrm{shrink}_{\frac{\alpha}{\lambda}}(\underbrace{b_{i}^{k}-\nabla_{i} u^{k+1} + w_{i}^{k}}_{g_{i}}):=\max\left(\norm{2}{g} - \frac{\alpha}{\lambda} \right)\frac{g_{i}}{\norm{2}{g}}.
\label{shrink}
\end{equation}
Finally, we discuss the solution of the subproblem \eqref{sub_w}. In the spirit of \cite{vogel}, we solve \eqref{sub_w} by a fixed point iteration scheme. Letting $\kappa=\frac{\beta}{\lambda}$ and  $\eta=-b^{k}-z^{k+1}+\nabla u^{k+1}$, the first-order optimality condition of \eqref{sub_w}  becomes
\begin{equation}
\kappa\frac{|w|^{p-2}w}{\norm{p}{w}^{p-1}}+w-\eta=0
\label{sub_w1}
\end{equation}
For given $w^{k}$, we obtain $w^{k+1}$ by the following fixed point iteration
\begin{equation}
w_{i}^{k+1}=\frac{\eta_{i}\norm{p}{w^{k}}^{p-1}}{\kappa|w^{k}|^{p-2}+\norm{p}{w^{k}}^{p-1}},
\label{sol_w}
\end{equation} 
under the convention that $0/0=0$. We can also consider solving the $p$-homogenous analogue \eqref{kostas_phomo}, where for certain values of $p$, e.g. $p=2$, we can solve exactly \eqref{sol_w}, since in that case $w_{i}^{k+1}=\frac{\eta_{i}}{\kappa+1}$. However, we observe numerically that there is no significant  computational difference between these two methods. Let us finally mention that since we do not solve exactly all the subproblems \eqref{sub_u}--\eqref{sub_w}, we do not have a convergence proof for the split Bregman iteration. However in practice,  the algorithm converges to the right solutions after comparing them with the corresponding solutions obtained with the CVX package.

\subsection{One dimensional results}\label{sec:numerics:1d}

For this section, we set $m=1$ and thus  $u\in\re^{n\times1}$, $w\in(\re^{n\times 1})^{2}$. Initially, we compare our numerical solutions with the analytical ones, obtained in Section \ref{section_exact} for the step function, setting $p=2$, $h=100$, $L=1$ and $\Omega=[-1,1]$. The domain $\Omega$ is discretised into $2000$  points.
 We first examine the cases of where ROF solutions are obtained, i.e., the parameters $\alpha$ and $\beta$ are selected according to the conditions \eqref{pc_cond} and \eqref{c_cond}, see Figure \ref{compare_ROF_sol}. There we see that the analytical solutions coincide with the numerical ones. 
\begin{figure}[h]
\begin{center}
\begin{subfigure}[t]{0.32\textwidth}
                \centering                                                  
                \includegraphics[width=1\linewidth]{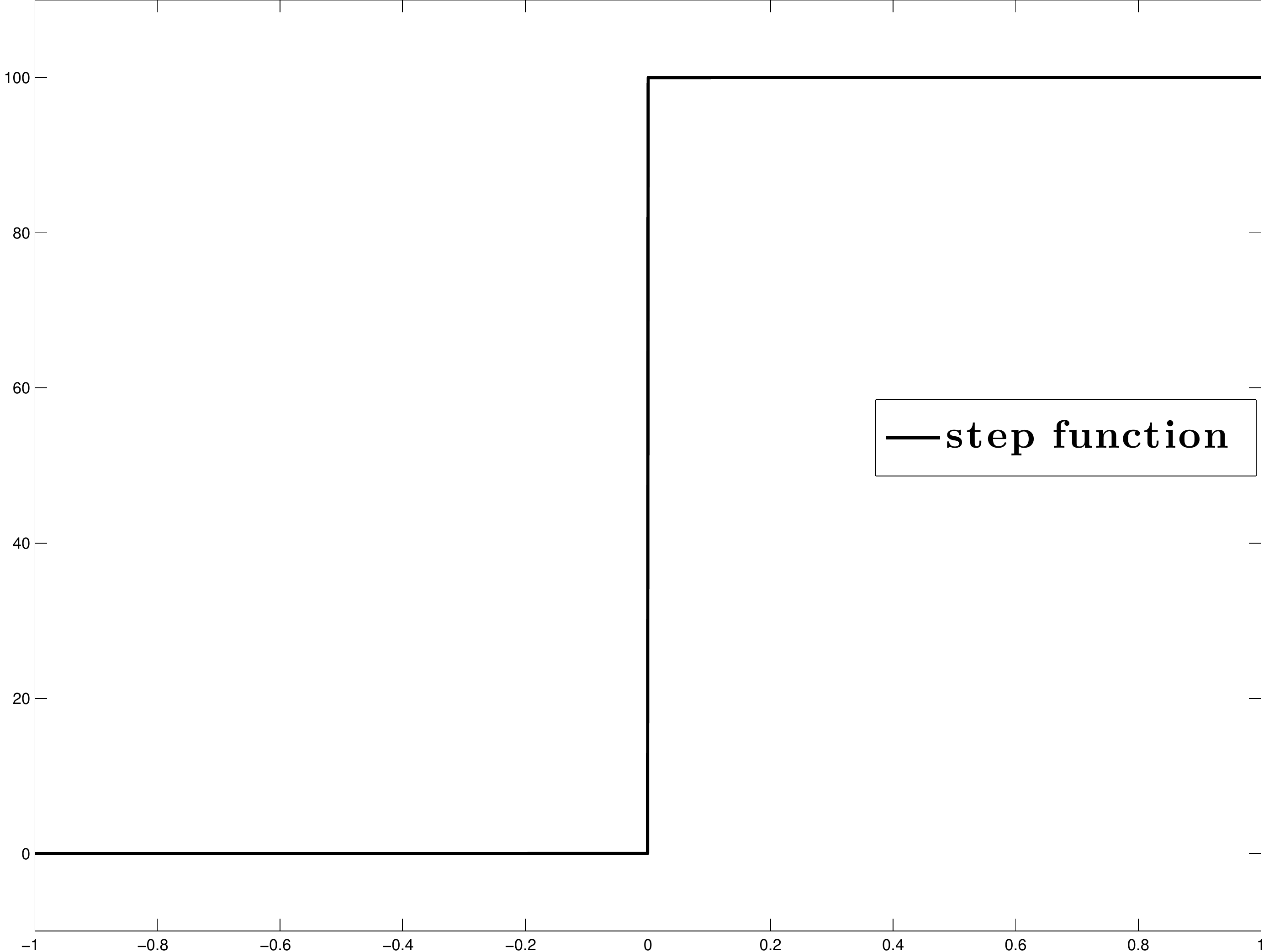}
                \caption{Original data} 
                \label{compare_ROF_sol:a}
\end{subfigure}
\begin{subfigure}[t]{0.32\textwidth}
                \centering                                                  
                \includegraphics[width=1\linewidth]{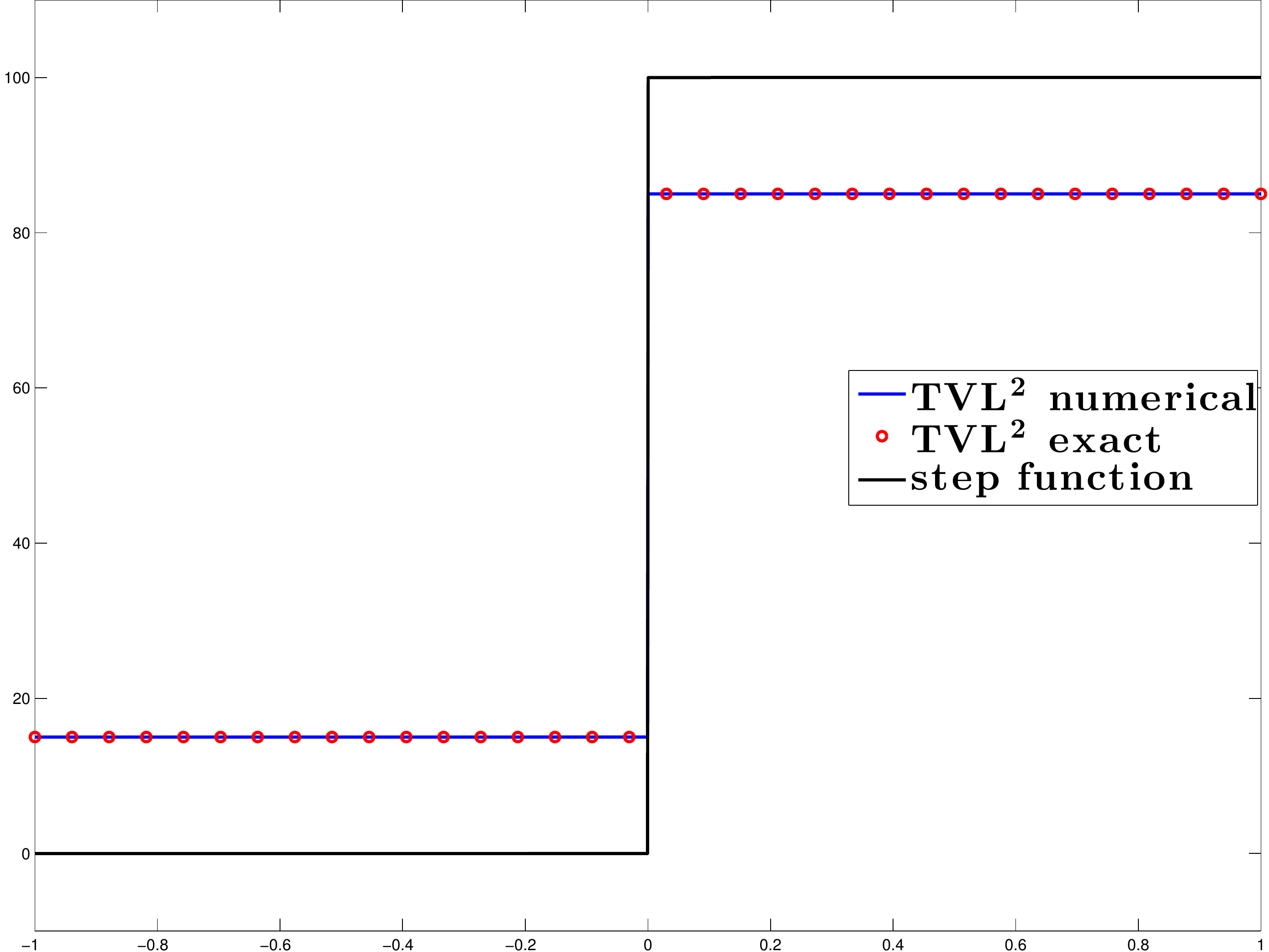}
                \caption{$(\mathrm{ROF})$: $\alpha=15$, $\beta=500$ }
              \label{compare_ROF_sol:b}
\end{subfigure}
\begin{subfigure}[t]{0.32\textwidth}
                \centering                                                  
                \includegraphics[width=1\linewidth]{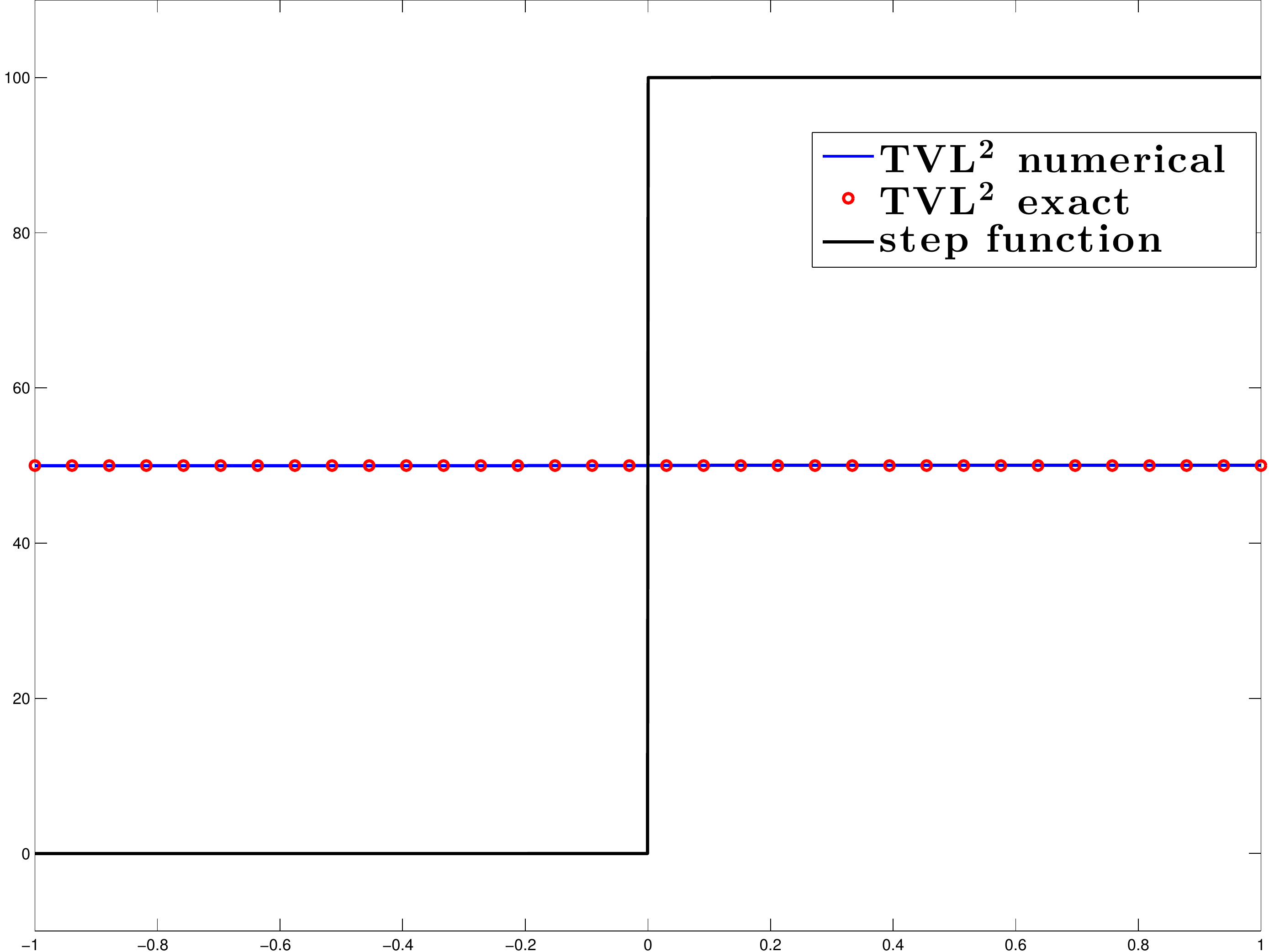}
                \caption{$(\mathrm{ROF}):$ $\alpha=60$, $\beta=1300$ }
			\label{compare_ROF_sol:c}
\end{subfigure}
\end{center}
\caption{Comparison between numerical solutions of \eqref{tvlp_min} and the corresponding analytical solutions obtained in Section \ref{section_exact}. The parameters $\alpha$ and $\beta$ are chosen so that conditions \eqref{pc_cond} and \eqref{c_cond} are satisfied.}
\label{compare_ROF_sol}
\end{figure}

Now, we proceed by computing the non-ROF solutions.
The numerical solutions are solved using the $2$-homogeneous analogue of \eqref{tvlp_min_p_hom}, since we have proved that the $1$-homogeneous and $p$-homogeneous problems are equivalent modulo an appropriate rescaling of the parameter $\beta$, see Proposition \ref{kostas_phomo_1homo}. In fact, as it is described in Figure \ref{tvl2_graph_1_hom}, in order to obtain solutions from the purple region, it suffices to seek solutions for the $2$-homogeneous \eqref{tvlp_min_p_hom}. Notice also that these solutions are exactly the solutions obtained solving a Huber TV problem, see Proposition \ref{equiv_tv_huber}. The analytical solutions are given in \eqref{cont_sol} and \eqref{disc_sol} and are compared with the numerical ones in Figure \ref{compare_TVL2_sol}, where we observe that they coincide. We also verify the equivalence between the $1$-homogeneous and $2$-homogeneous problems where $\alpha$ is fixed and $\beta$ is obtained from Proposition \ref{kostas_phomo_1homo}, see Figure \ref{compare_TVL2_sol:c}.

\begin{figure}[h]
\begin{center}
\begin{subfigure}[t]{0.32\textwidth}
                \centering                                                  
                \includegraphics[width=1\linewidth]{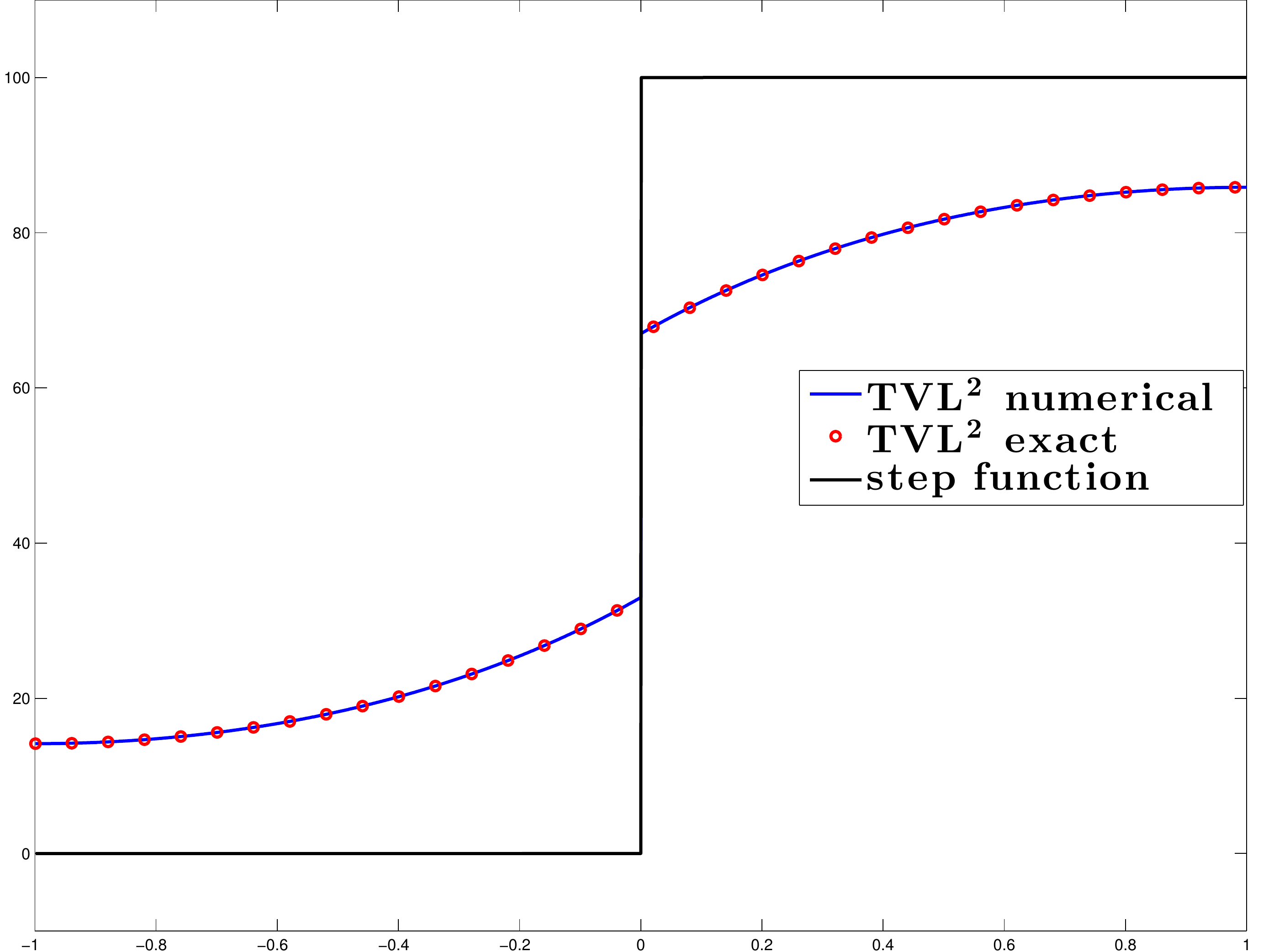}
                \caption{\centering$\mathrm{TVL^{2}}:$ $\alpha=20$, $\beta_{2-hom}=450$  } 
                \label{compare_TVL2_sol:a}
\end{subfigure}
\begin{subfigure}[t]{0.32\textwidth}
                \centering                                                  
                \includegraphics[width=1\linewidth]{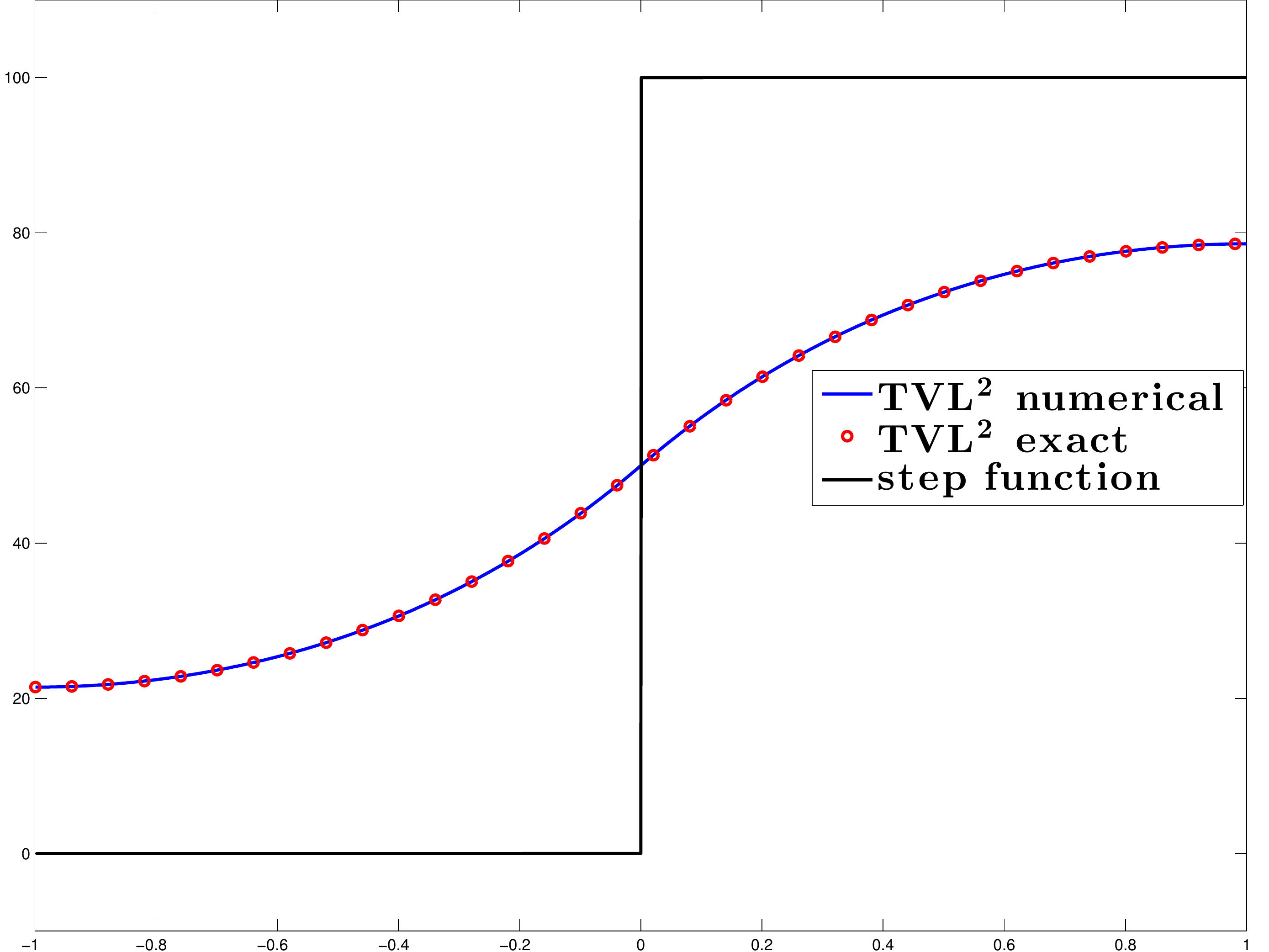}
                \caption{\centering$\mathrm{TVL^{2}}:$ $\alpha=60$, $\beta_{2-hom}=450$} 
              \label{compare_TVL2_sol:b}
\end{subfigure}
\begin{subfigure}[t]{0.32\textwidth}
                \centering                                                  
                \includegraphics[width=1\linewidth]{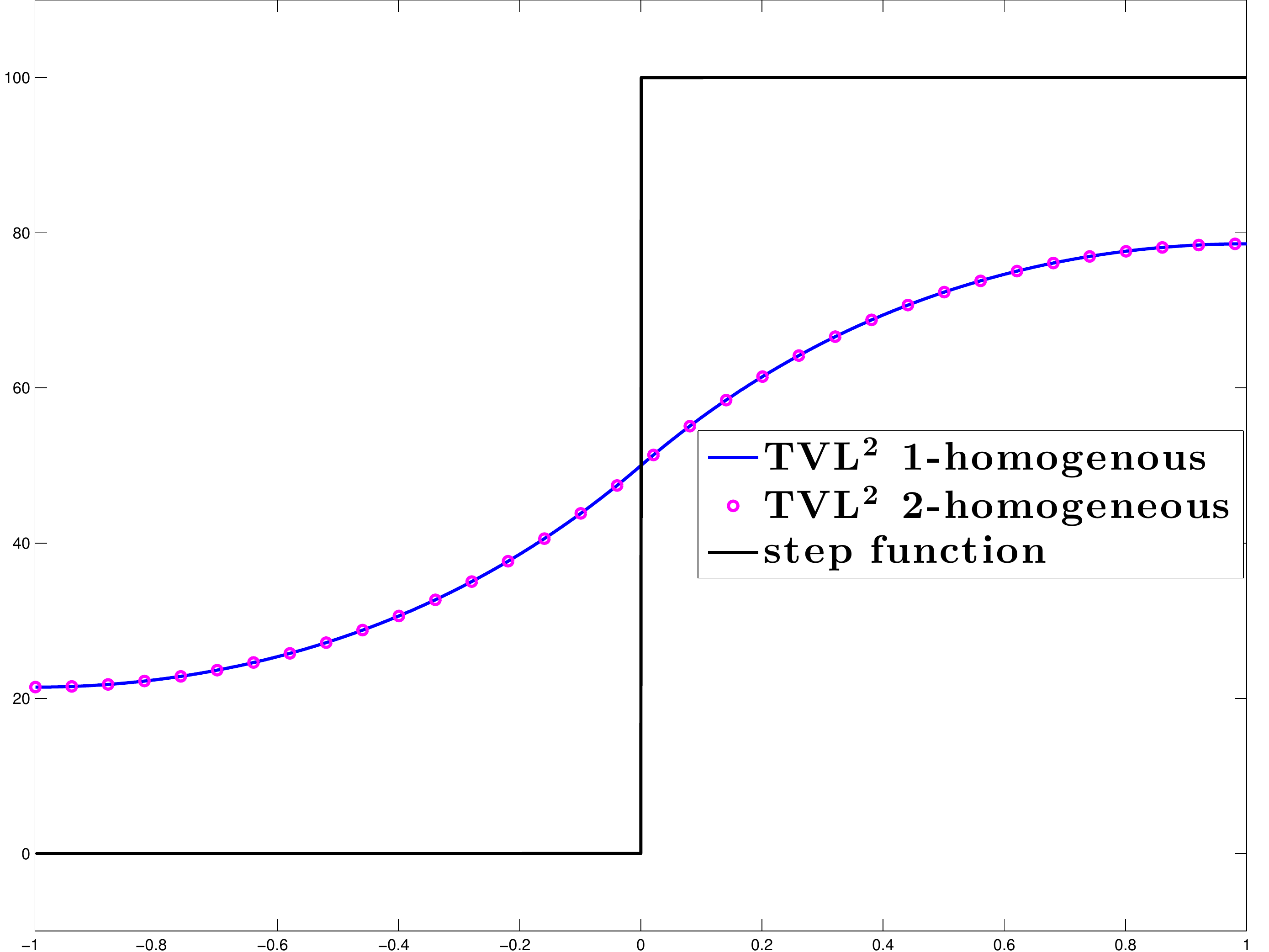}
                \caption{\centering $1$ and $2$-homogeneous:  $\alpha=15$, $\beta_{2-hom}$=450, $\beta_{1-hom}=\beta_{2-hom}\norm{2}{w}$ } 
			\label{compare_TVL2_sol:c}
\end{subfigure}
\end{center}
\caption{Comparison between numerical and analytical solutions obtained in Section \ref{section_exact} for the step function, by solving the $2$-homogeneous problem \eqref{tvlp_min_p_hom}. The parameters $\alpha$ and $\beta$ are chosen so that conditions \eqref{cont_sol} and \eqref{disc_sol} are satisfied. The last plot indicates the equivalence between the $2$-homogeneous  \eqref{tvlp_min_p_hom} and the $1$-homogeneous problem \eqref{tvlp_min}.}
\label{compare_TVL2_sol}
\end{figure}

We continue our experiments for general values of $p$ focusing on the geometric behaviour of the solutions as $p$ increases. In order to compare the solutions for $p\in(1,\infty)$, we fix the parameter $\alpha$ and choose appropriate values of $\beta$ and $p$. We choose $\alpha$ and $\beta$ so that they belong to the purple region in Figure \ref{tvl2_graph_1_hom},  i.e., $\beta<(\frac{2L}{q+1})^{\frac{1}{q}}\alpha$ and $\beta<\frac{h}{2} (\frac{2L^{q+1}}{q+1})^{\frac{1}{q}}$, hence non-ROF solutions are obtained. We set $p=\{\frac{4}{3}, \frac{3}{2}, 2, 3, 4, 10\}$ and for the solutions that preserve the discontinuity we select $\beta=\{72, 140, 430, 1350, 2400, 6800\}$ with fixed $\alpha=20$ (observe that $\beta<(\frac{2L}{q+1})^{\frac{1}{q}}\alpha$ is valid in any case), see Figure \ref{different_p_step:a}. For the continuous cases, we set $\alpha=60$ and $\beta=\{50, 110, 430, 1700, 3000, 9500\}$ (again the conditions $\alpha\geq\frac{hL}{2}$ and $\beta<\frac{h}{2} (\frac{2L^{q+1}}{q+1})^{\frac{1}{q}}$ hold), see Figure \ref{different_p_step:b}.
 We observe that for $p=\frac{4}{3}$, the solution has a similar behaviour to $p=2$, but with a steeper gradient at the discontinuity point. Moreover, the solution becomes almost constant near the boundary of $\Omega$. On the other hand, as we increase $p$, the slope of the solution near the discontinuity point reduces and it becomes almost linear with a relative small constant part near the boundary.\par
\begin{figure}[h]
\begin{center}
\begin{subfigure}[t]{7.7cm}
                \centering                                                  
                \includegraphics[width=7.7cm]{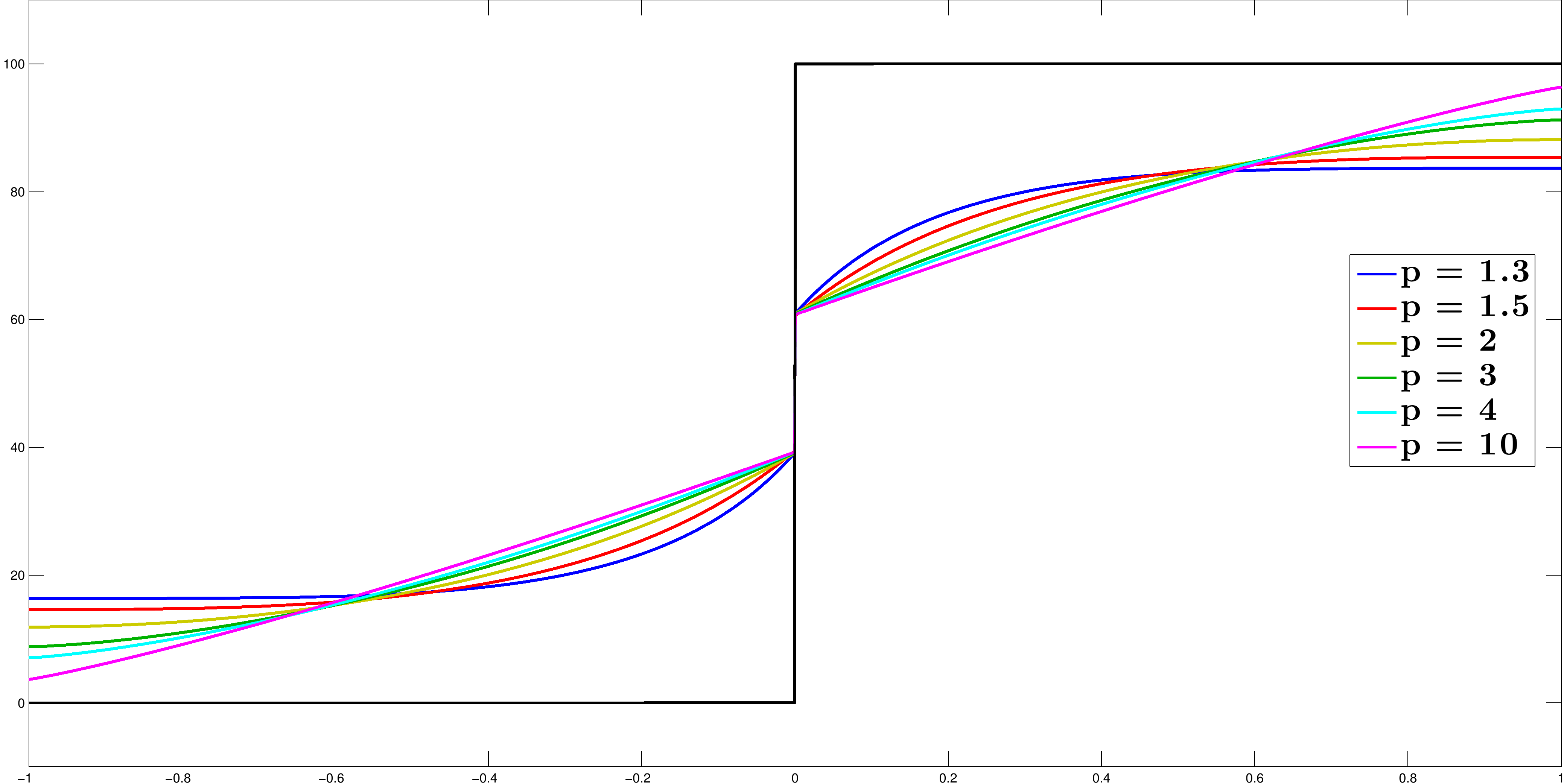}
                \caption{\centering$\mathrm{TVL}^p$ discontinuous solutions for $p=\{\frac{4}{3}, \frac{3}{2}, 2, 3, 4, 10\}$} 
                \label{different_p_step:a}
\end{subfigure}
\begin{subfigure}[t]{7.7cm}
                \centering                                                  
                \includegraphics[width=7.7cm]{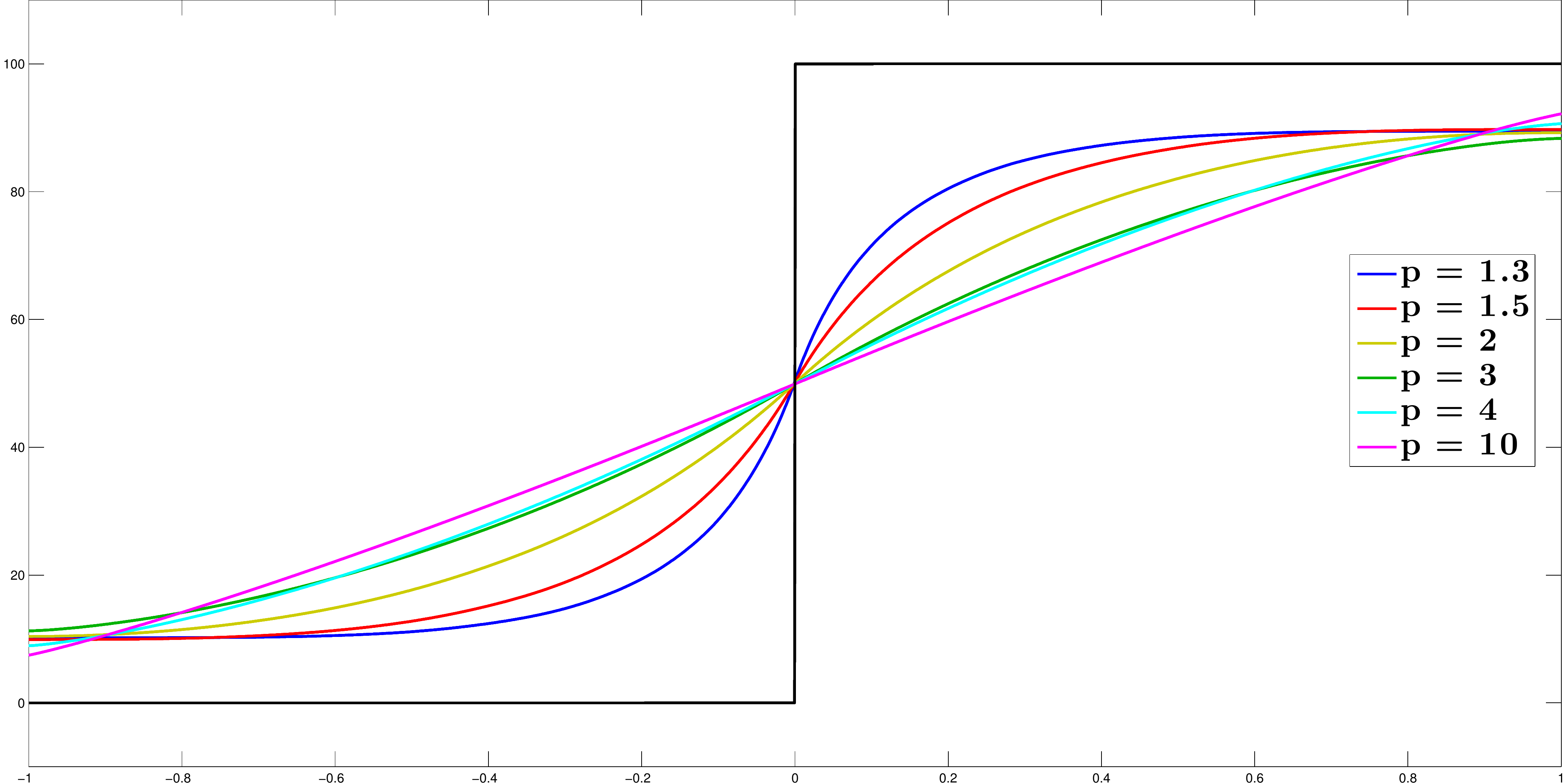}
               \caption{\centering$\mathrm{TVL}^p$ continuous solutions for $p=\{\frac{4}{3}, \frac{3}{2}, 2, 3, 4, 10\}$} 
                \label{different_p_step:b}
\end{subfigure}
\end{center}
\caption{Step function: The types of solutions for the problem \eqref{tvlp_min}  for different values of $p$.}
\label{different_p_step}
\end{figure}

\begin{figure}[h!]
                \centering                                                  
                \includegraphics[width=6cm]{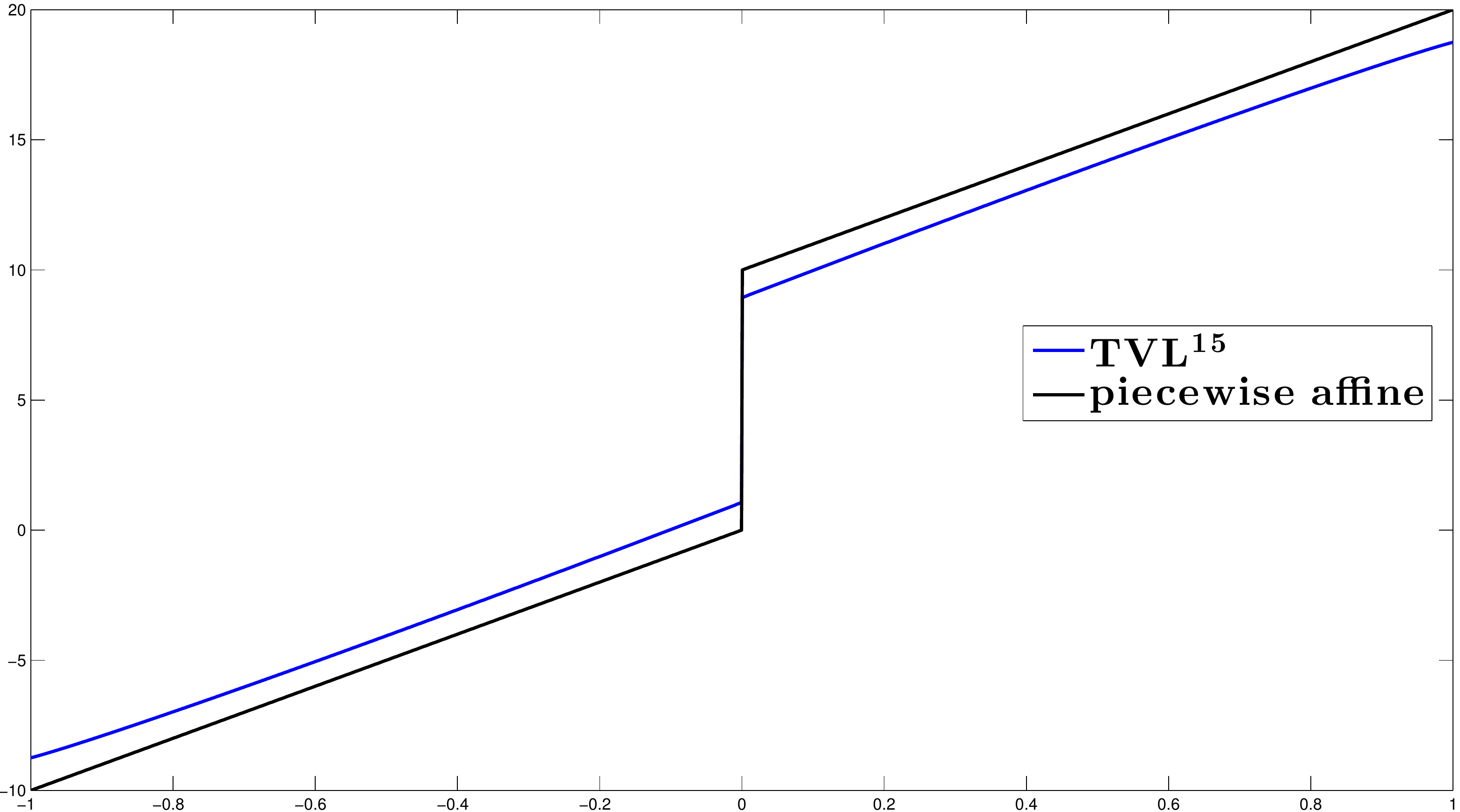}
                \caption{Piecewise affine data: $\mathrm{TVL}^{15}$ solution with $\alpha=1$, $\beta=620$.} 
                \label{affine1}
\end{figure}

The linear structure of the solutions that appears for large $p$ motivates us to examine the case of a piecewise linear data $f$ defined as
\begin{equation}
f(x)=
\begin{cases}
\lambda x&\mbox{ if }x\in(-L,0],\\
\lambda x+h&\mbox{ if }x\in(0,L],
\end{cases}
\end{equation}
see Figure \ref{affine1}. We set again $\Omega=[-1,1]$, $\lambda=\frac{1}{10}$ and the data are discretised in 2000 points. As we observe, the reconstruction for $p=15$ behaves almost linearly everywhere in $\Omega$ except near the boundary.  In the follow up paper \cite{partII}, where the case $p=\infty$ is examined in detail, the occurrence of this linear structure is justified.

In the last part of this section, we discuss the image decomposition approach presented in Section \ref{sec:infimal}. We treat a more complicated one dimensional noiseless signal with piecewise constant, affine and quadratic components and solve the discretised version of \eqref{inf_conv} using $\mathrm{CVX}$ under $\mathrm{MOSEK}$. We verify numerically the equivalence between \eqref{inf_conv} and \eqref{tvlp_min} for $p=2$, i.e., $(\nabla v,u+v)$ corresponds to $(w,\overline{u})$ where $(v,u)$ and $(w,\overline{u})$ are the solutions of \eqref{inf_conv} and \eqref{tvlp_min} respectively, see Figure \ref{infimal_convolution}. We also compare the decomposed parts $u,v$ for two different values of $p$ ($\frac{4}{3}$ and $10$). In order to have a reasonable comparison on the corresponding solutions, the parameters $\alpha, \beta$ are selected such that the residual $\norm{2}{f-u-v}$ is the same for both values of $p$. As we observe, the $v$ decomposition with $p=\frac{4}{3}$ promotes some \emph{flatness} on the solution compared to $p=2$, compare Figures \ref{infimal_convolution:b} and \ref{inf_conv:a}. On the other hand for $p=10$, the $v$ component promotes again almost affine structures, Figure \ref{inf_conv:b}. Notice, that in both cases the $v$ components are continuous. In fact, this is confirmed analytically for every $1<p<\infty$, since in dimension one $\mathrm{W}^{1,p}(\Omega)\subset C(\overline{\Omega})$.   
\begin{figure}[h]
\begin{center}
\begin{subfigure}[t]{6cm}
                \centering                                                  
                \includegraphics[width=6cm]{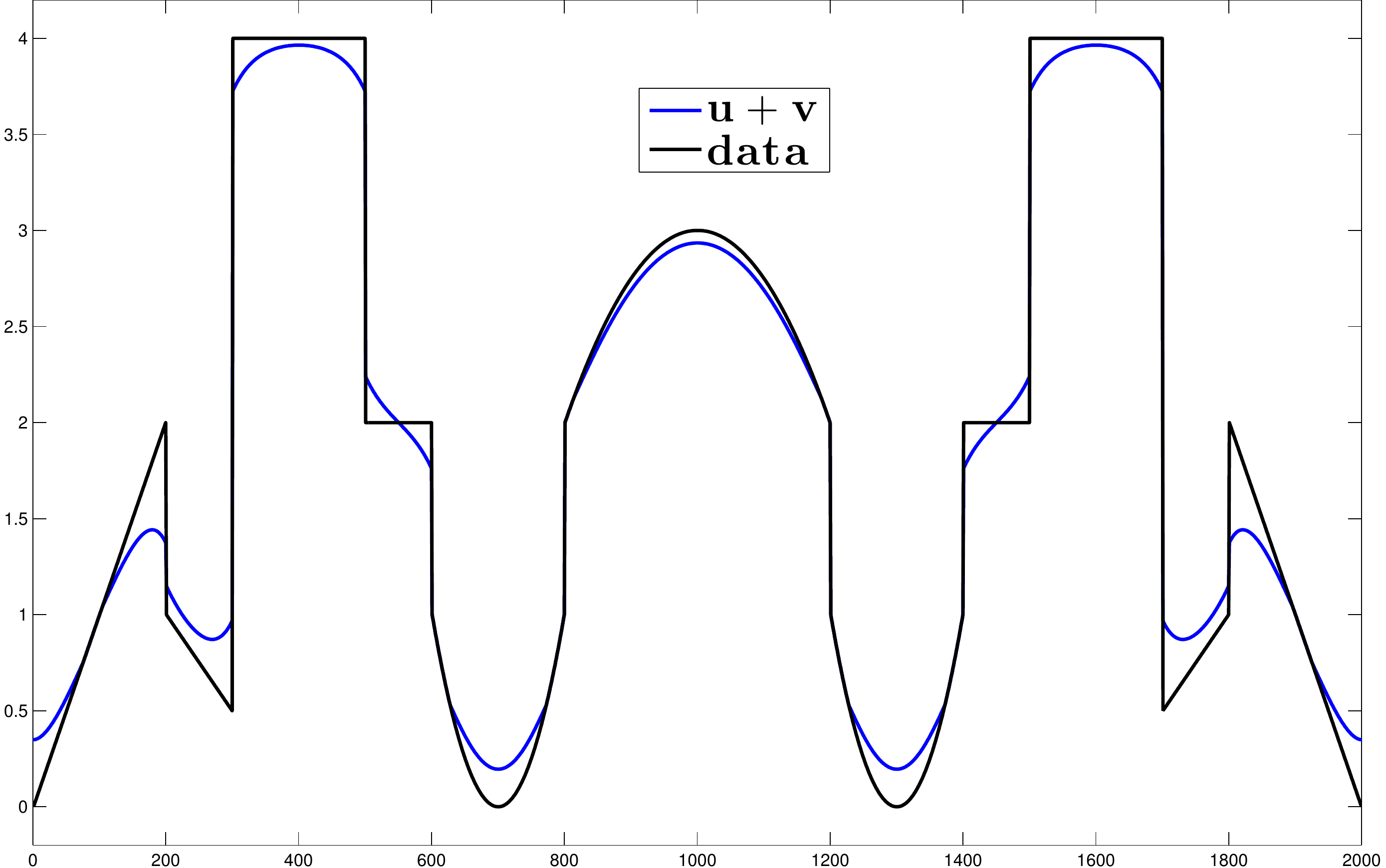}
                \caption{Solution $u+v$ of \eqref{inf_conv}} 
                \label{infimal_convolution:a}
\end{subfigure}
\begin{subfigure}[t]{6cm}
                \centering                                                  
                \includegraphics[width=6cm]{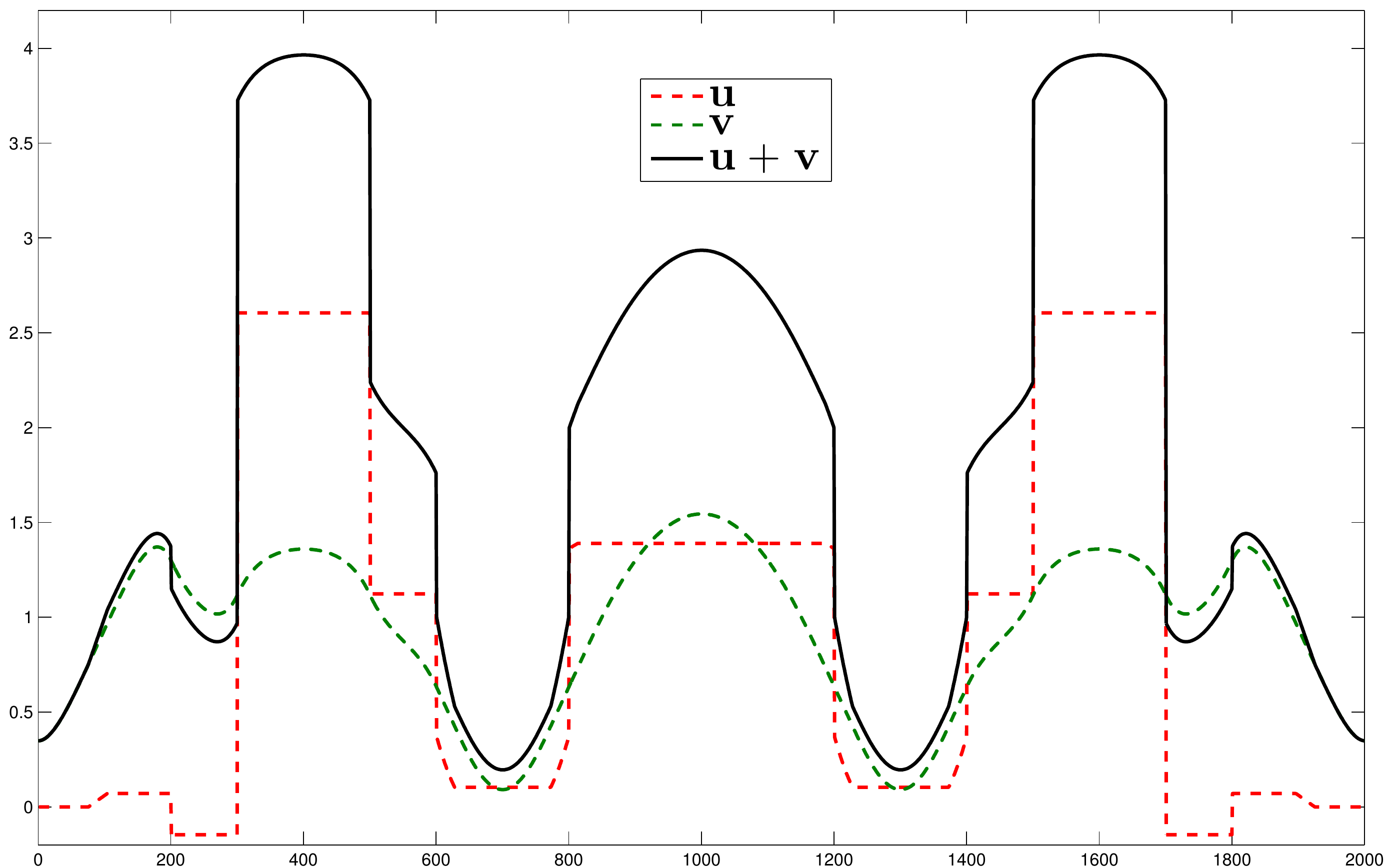}
                \caption{Decomposition into $u, v$ parts} 
                \label{infimal_convolution:b}
\end{subfigure}\\
\begin{subfigure}[t]{6cm}
                \centering                                                  
                \includegraphics[width=6cm]{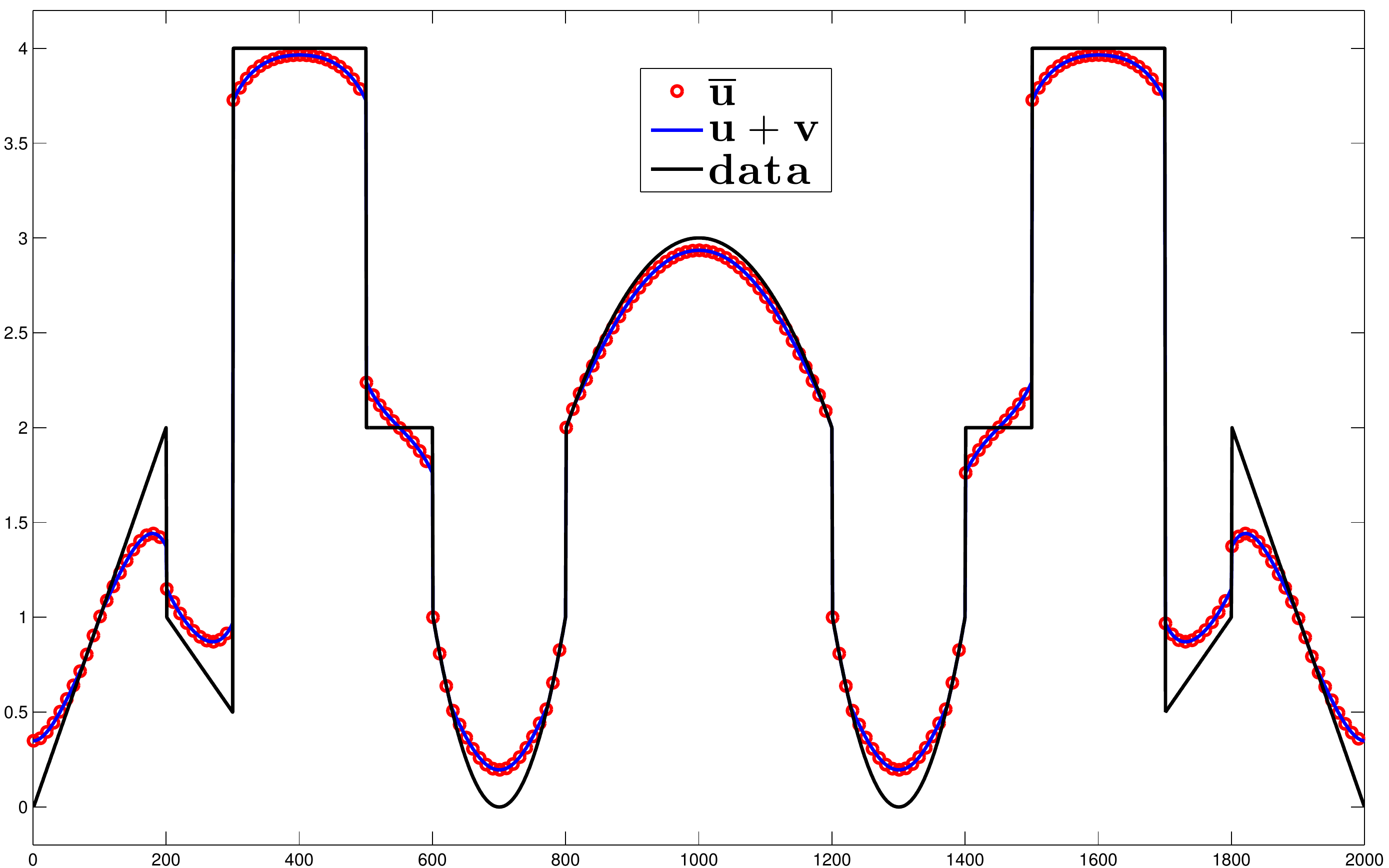}
                \caption{Equivalence of \eqref{tvlp_min} and \eqref{inf_conv}: $\overline{u}=u+v$} 
                \label{infimal_convolution:c}
\end{subfigure}
\begin{subfigure}[t]{6cm}
                \centering                                                  
                \includegraphics[width=6cm]{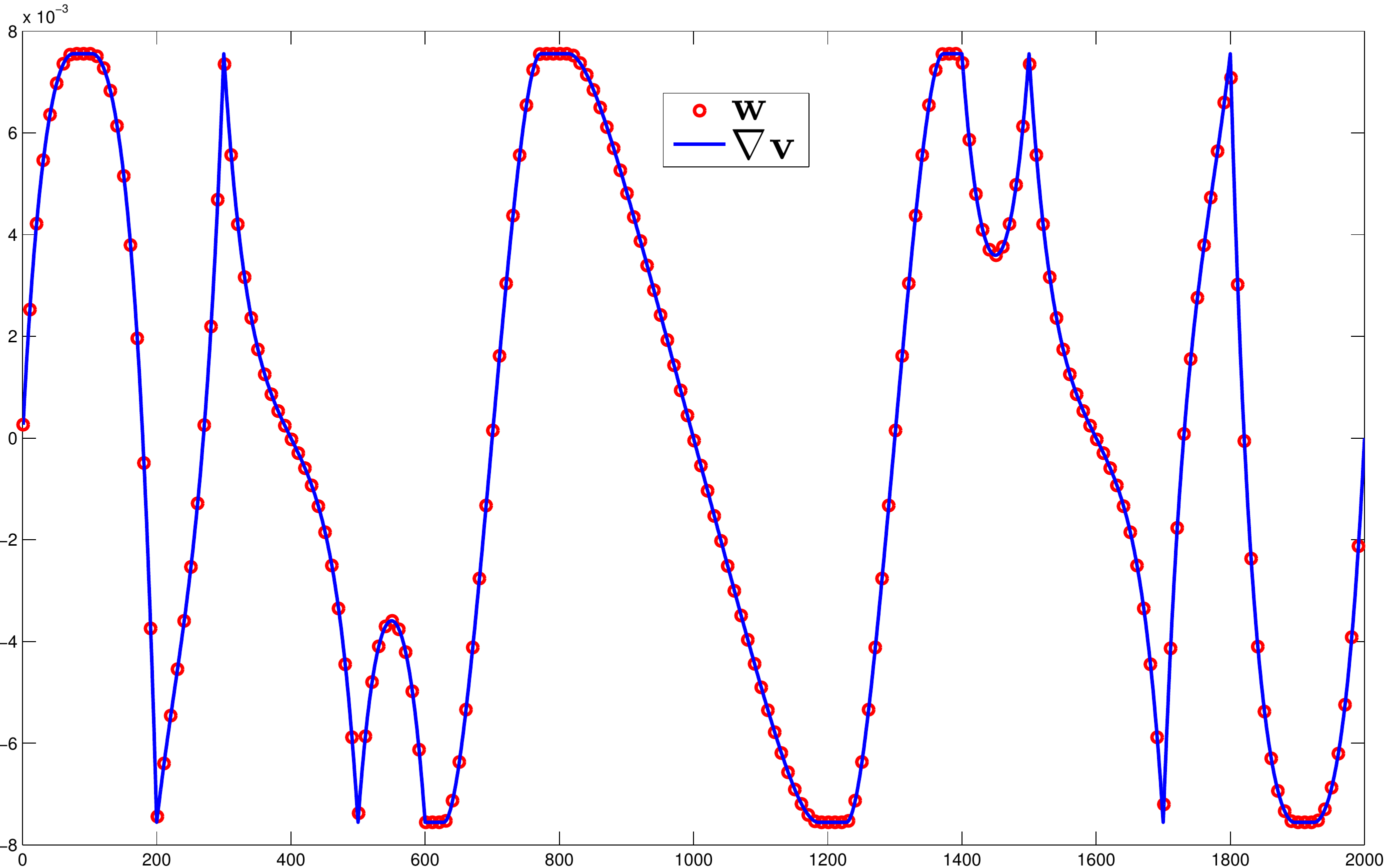}
                \caption{Equivalence of \eqref{tvlp_min} and \eqref{inf_conv}: $w=\nabla v$} 
                \label{iinfimal_convolution:d}
\end{subfigure}
\end{center}
\caption{Numerical results on the image decomposition approach \eqref{inf_conv} for $p=2$, see Section \ref{sec:infimal}.}
\label{infimal_convolution}
\end{figure}
\begin{figure}[h]
\begin{center}
\begin{subfigure}{6cm}
                \centering                                                  
                \includegraphics[width=6cm]{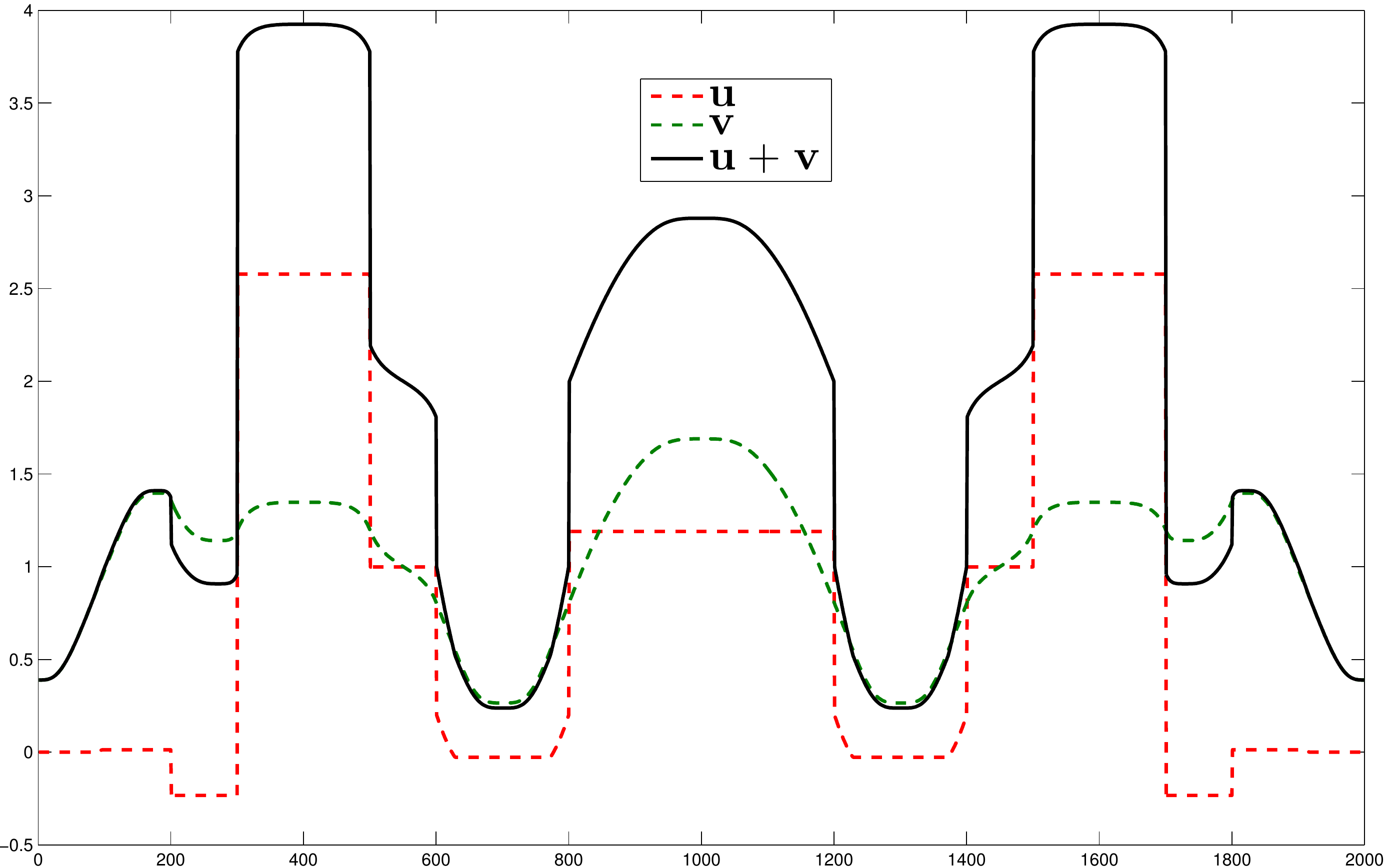}
                \caption{Decomposition of the data in Figure \ref{infimal_convolution:a} for $p=\frac{4}{3}$} 
                \label{inf_conv:a}
\end{subfigure}
\begin{subfigure}{6cm}
                \centering                                                  
                \includegraphics[width=6cm]{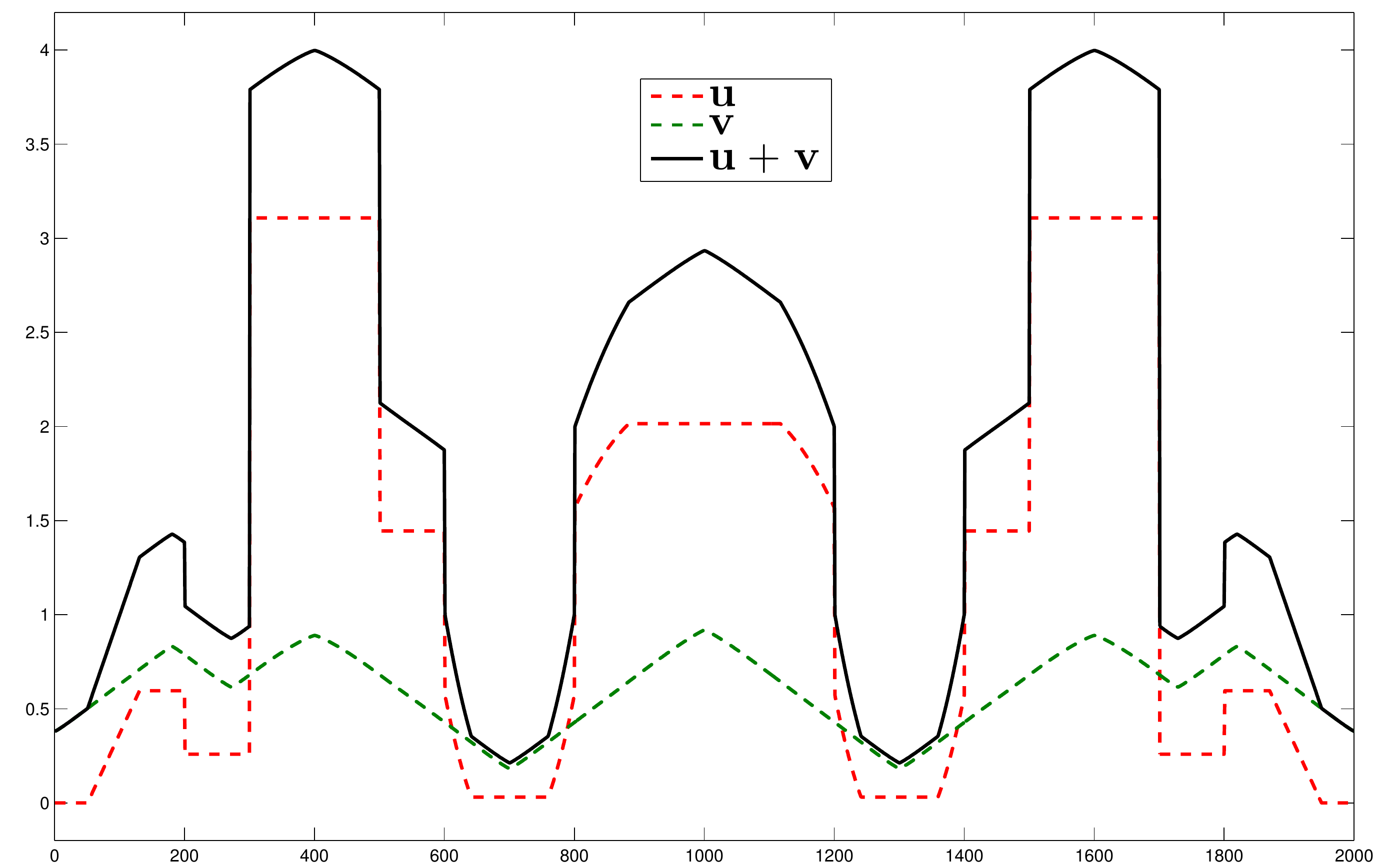}
                \caption{Decomposition of the data in Figure \ref{infimal_convolution:a} for $p=10$} 
                \label{inf_conv:b}
\end{subfigure}
\end{center}
\caption{Decomposition of the data in Figure \ref{infimal_convolution:a} into $u, v$ parts for $p=\frac{4}{3}$ and $p=10$. The value $p=\frac{4}{3}$ produces a $v$ component with flat structures  while $p=10$ produces a component with almost affine structures. In both cases we have $\norm{2}{f-u-v}=6.667$.} 
\label{infimal_convolution_1}
\end{figure}

\subsection{Two dimensional results}\label{sec:numerics:2d} 

In this section we consider the two dimensional case where $u\in\re^{n\times m}$, $w\in(\re^{n\times m})^2$ with $m>1$ and $\Omega$ denotes a rectangular/square image domain. We focus on image denoising tasks and on eliminating the staircasing effect for different values of $p$. 
We use here the split Bregman algorithm proposed in  Section \ref{sec:SB}.

We start with the image in Figure \ref{square_1}, i.e., a square with piecewise affine structures. The image size is $200\times 200$ pixels at a $[0,1]$ intensity range. The noisy image, Figure \ref{square_1:b}, is a corrupted version of the original image, Figure \ref{square_1:a}, with Gaussian noise of zero mean and variance $\sigma=0.01$.

\begin{figure}[h]
\begin{center}
\begin{subfigure}[t]{6cm}
                \centering                                                  
                \includegraphics[width=3.8cm]{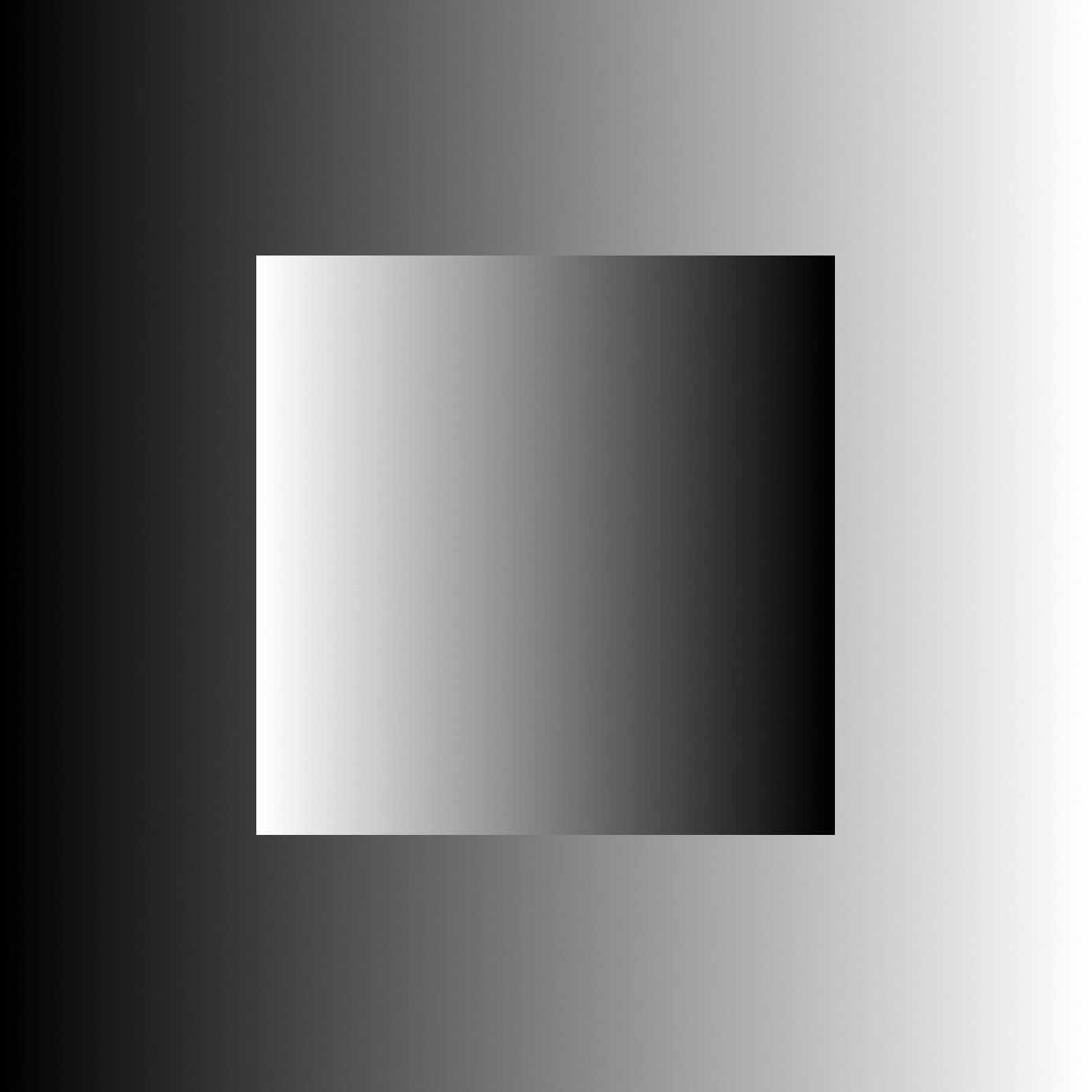}
                \caption{Square} 
                \label{square_1:a}
\end{subfigure}
\begin{subfigure}[t]{6cm}
                \centering                                                  
                \includegraphics[width=3.8cm]{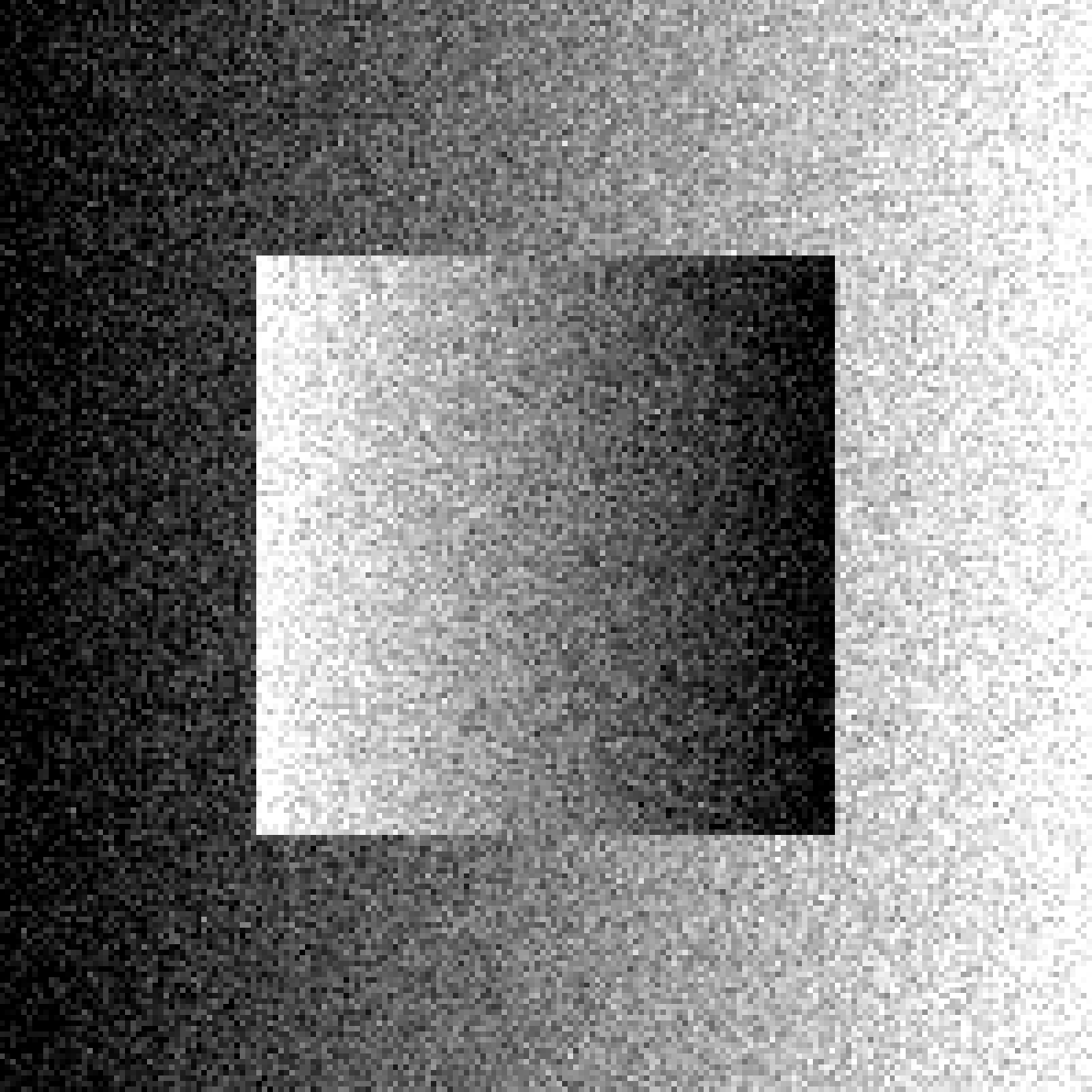}
                \caption{\centering Noisy square: PSNR=20.66 and SSIM=0.1791}
                \label{square_1:b}
\end{subfigure}
\end{center}
\caption{Square with piecewise affine structures and its noisy version with $\sigma=0.01$.}
\label{square_1}
\end{figure}
 
\begin{figure}[h]
\begin{center}
\begin{subfigure}[h]{5cm}
                \centering                                                  
                \includegraphics[width=5cm]{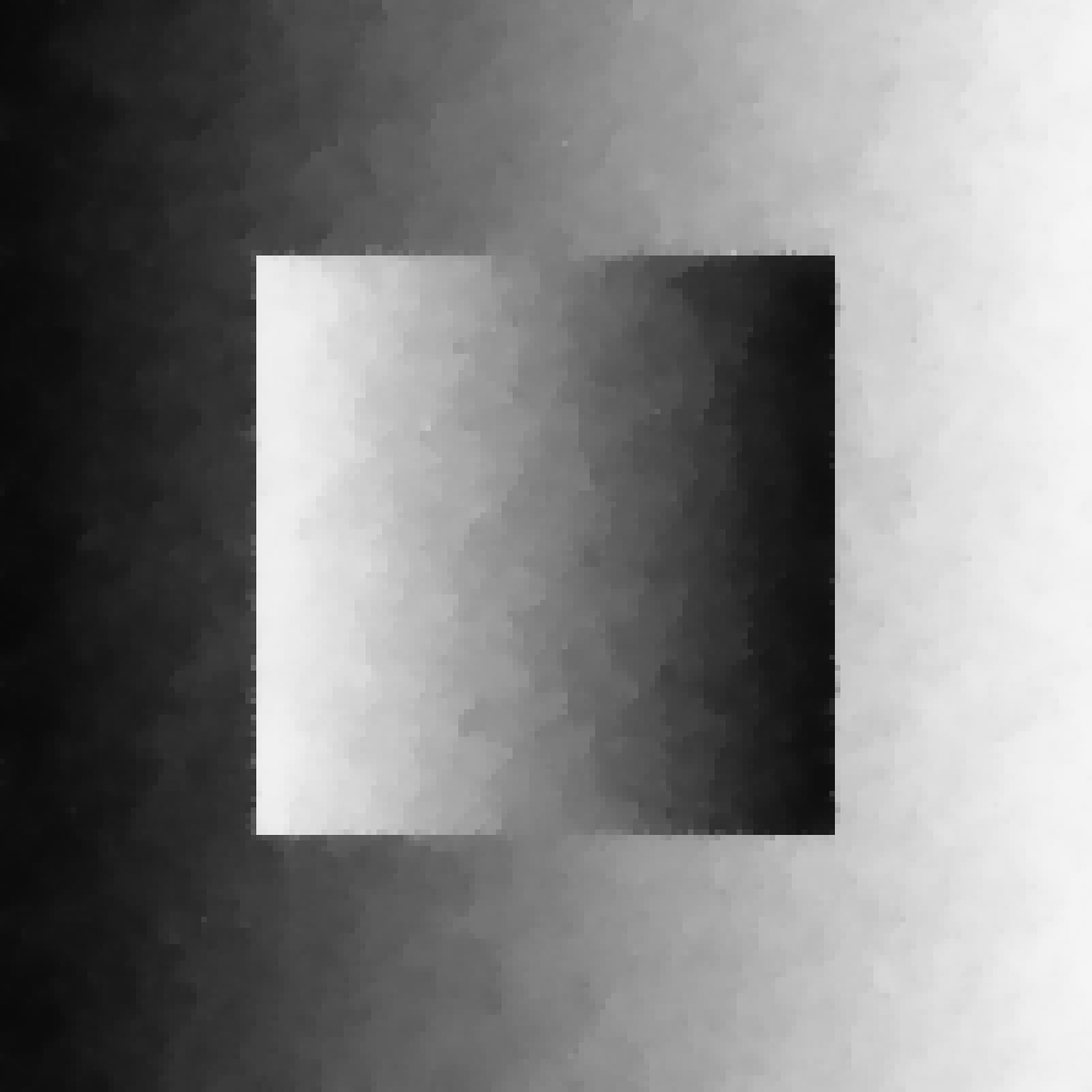}
                \caption{\centering$\mathrm{TVL}^{\frac{3}{2}}$: $\alpha=0.1$, $\beta=2.5$, PSNR=33.63 }
                \label{square_2:a}
\end{subfigure}
\begin{subfigure}[h]{5cm}
                \centering                                                  
                \includegraphics[width=5cm]{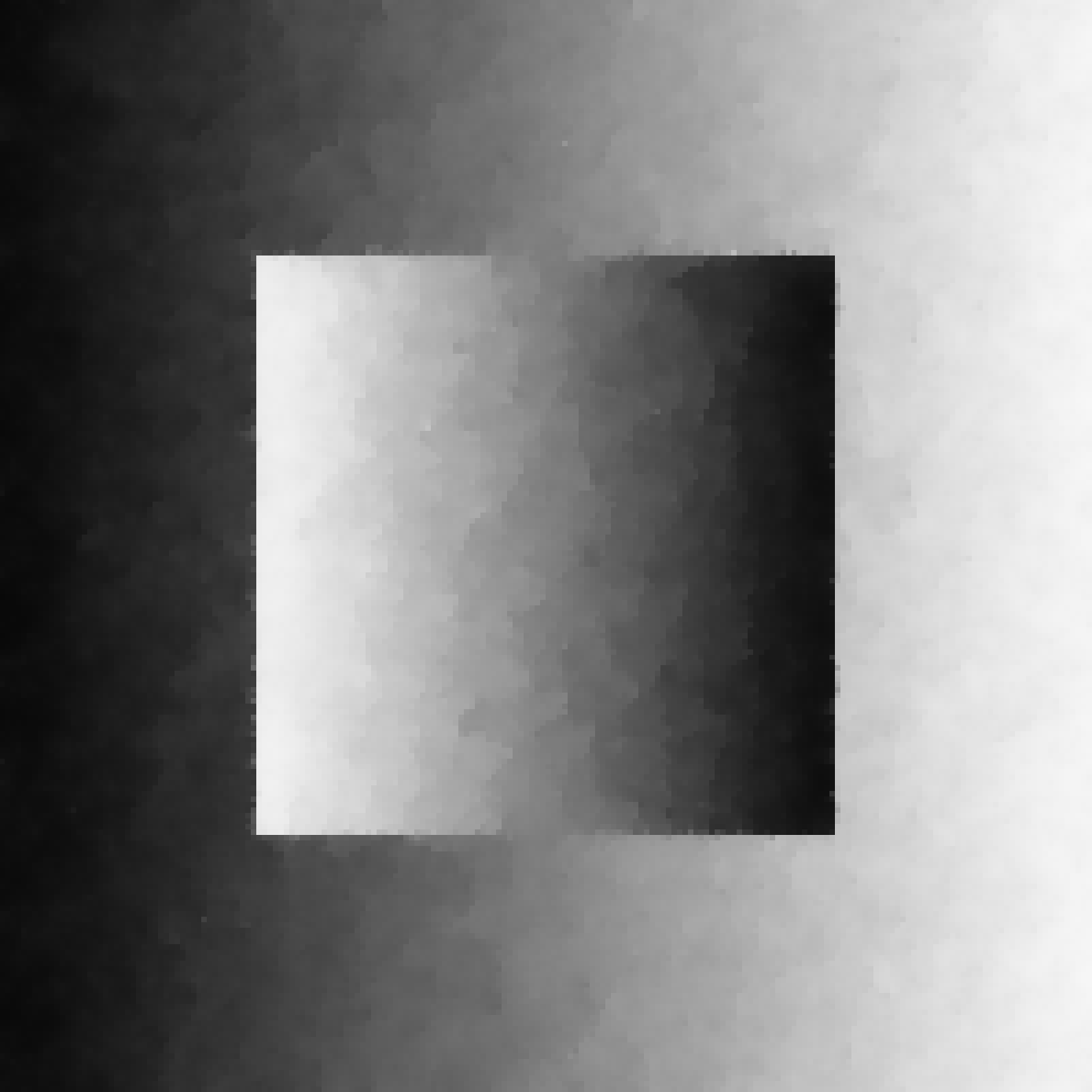}
                 \caption{\centering$\mathrm{TVL}^{2}$: $\alpha=0.1$, $\beta=13.5$, PSNR=33.68} 
                 \label{square_2:b}
\end{subfigure}
\begin{subfigure}[h]{5cm}
                \centering                                                  
                \includegraphics[width=5cm,height=5cm]{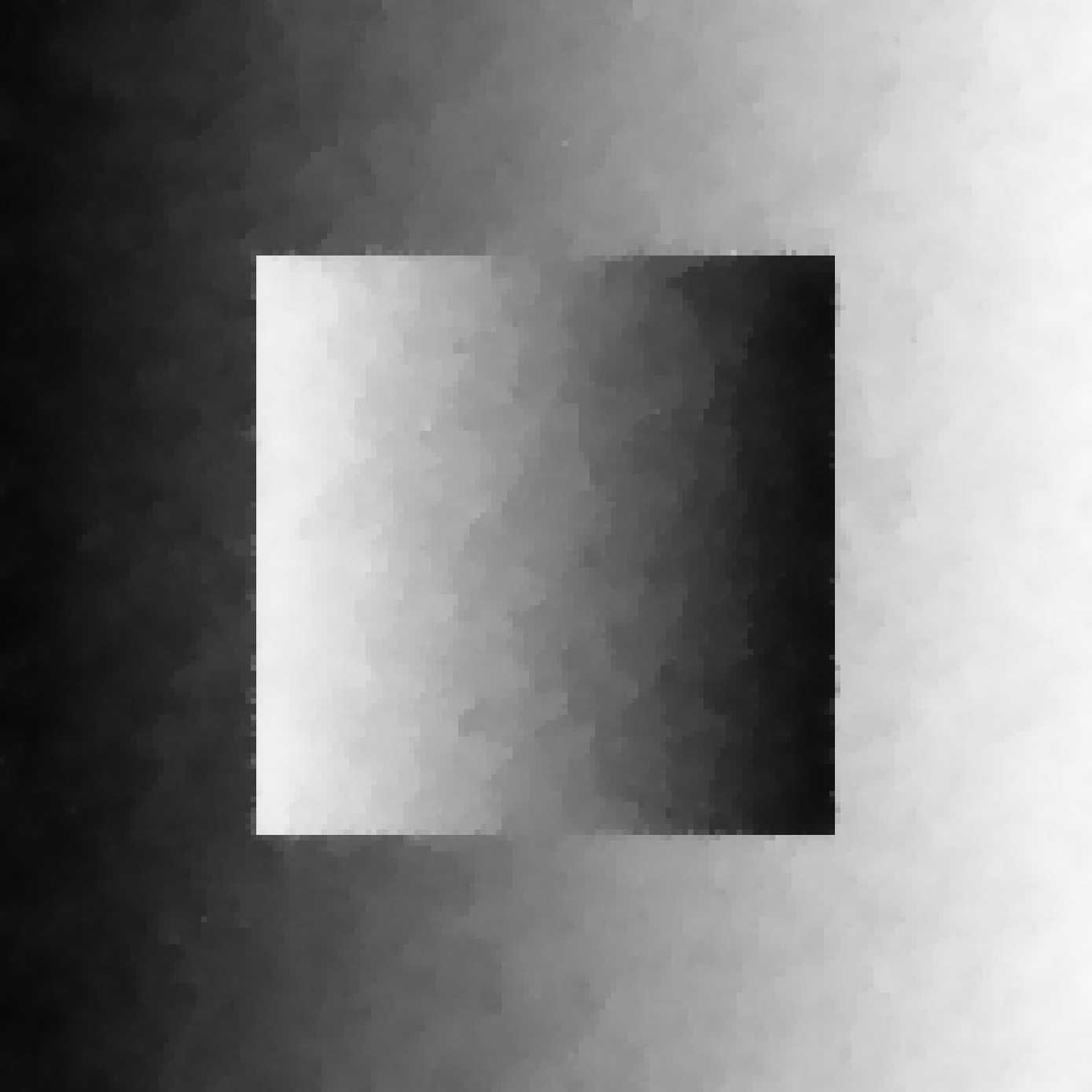}
                 \caption{\centering$\mathrm{TVL}^{3}$: $\alpha=0.1$, $\beta=76$, PSNR=33.70} 
                 \label{square_2:c}
\end{subfigure}\\
\begin{subfigure}[h]{5cm}
                \centering                                                  
                \includegraphics[width=5cm]{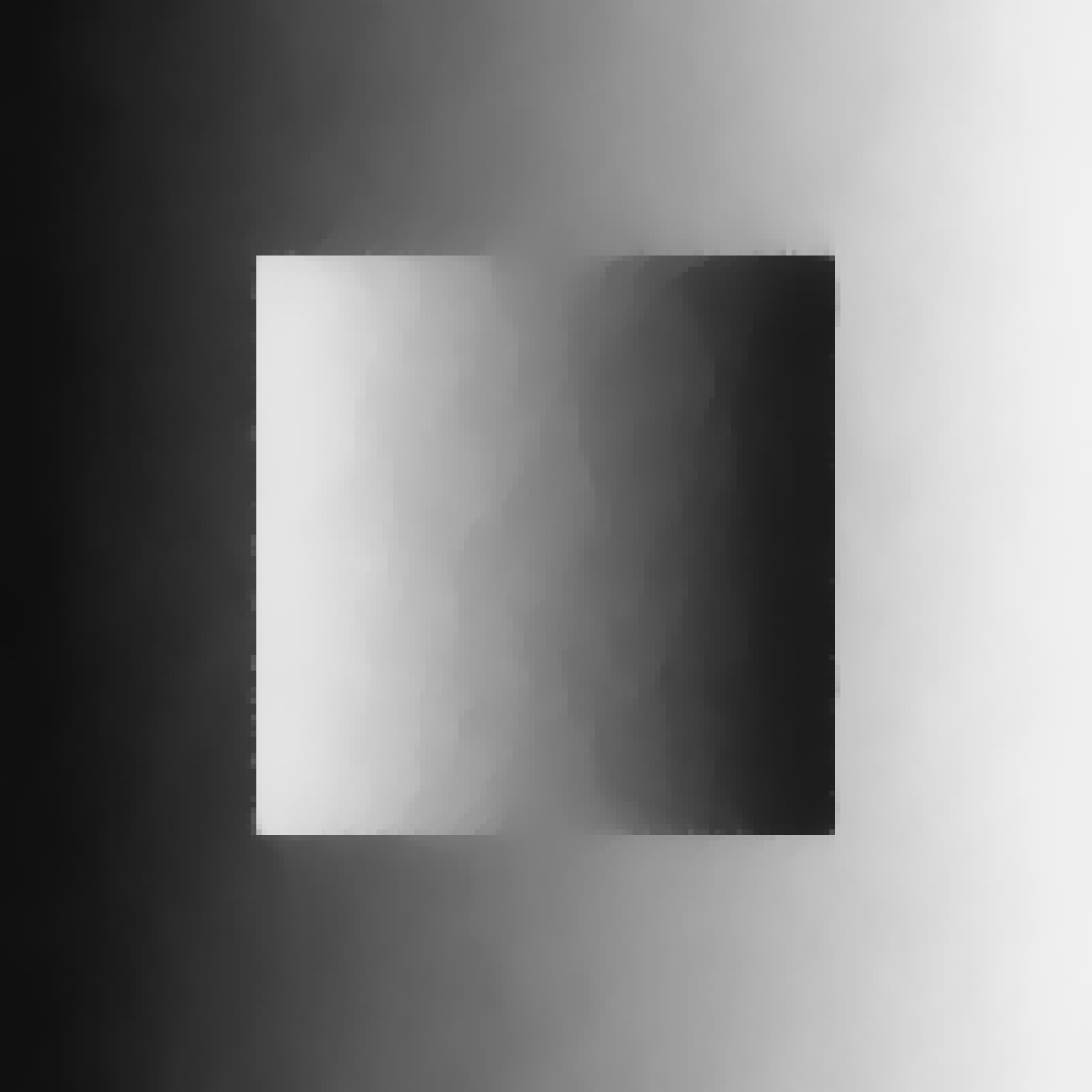}
                 \caption{\centering$\mathrm{TVL}^{\frac{3}{2}}$: $\alpha=0.3$, $\beta=7.7$, SSIM=0.9669} 
                 \label{square_2:d}
\end{subfigure}
\begin{subfigure}[h]{5cm}
                \centering                                                  
                \includegraphics[width=5cm]{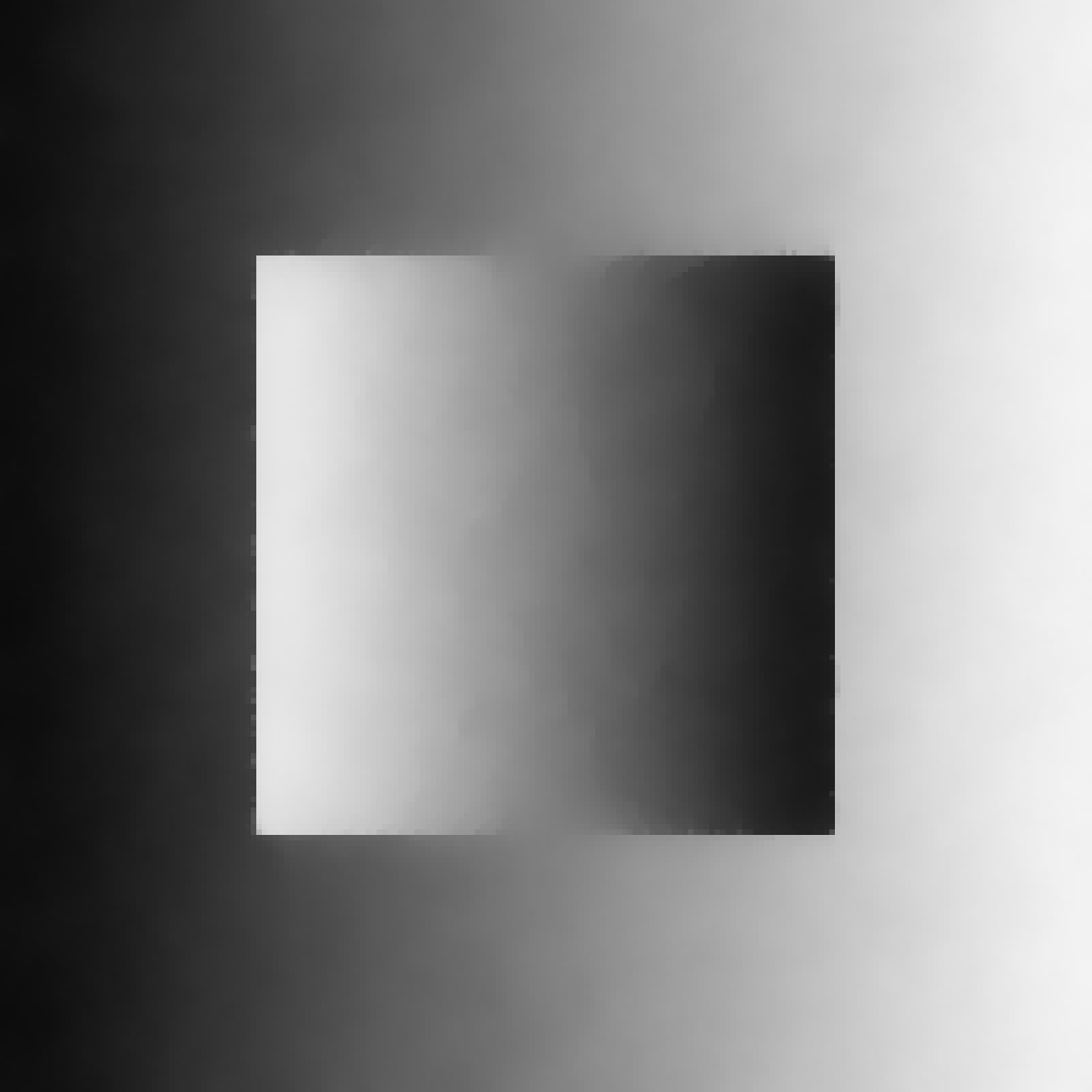}
                 \caption{\centering$\mathrm{TVL}^{2}$: $\alpha=0.3$, $\beta=34$, SSIM=0.9706} 
                 \label{square_2:e}
\end{subfigure}
\begin{subfigure}[h]{5cm}
                \centering                                                  
                \includegraphics[width=5cm]{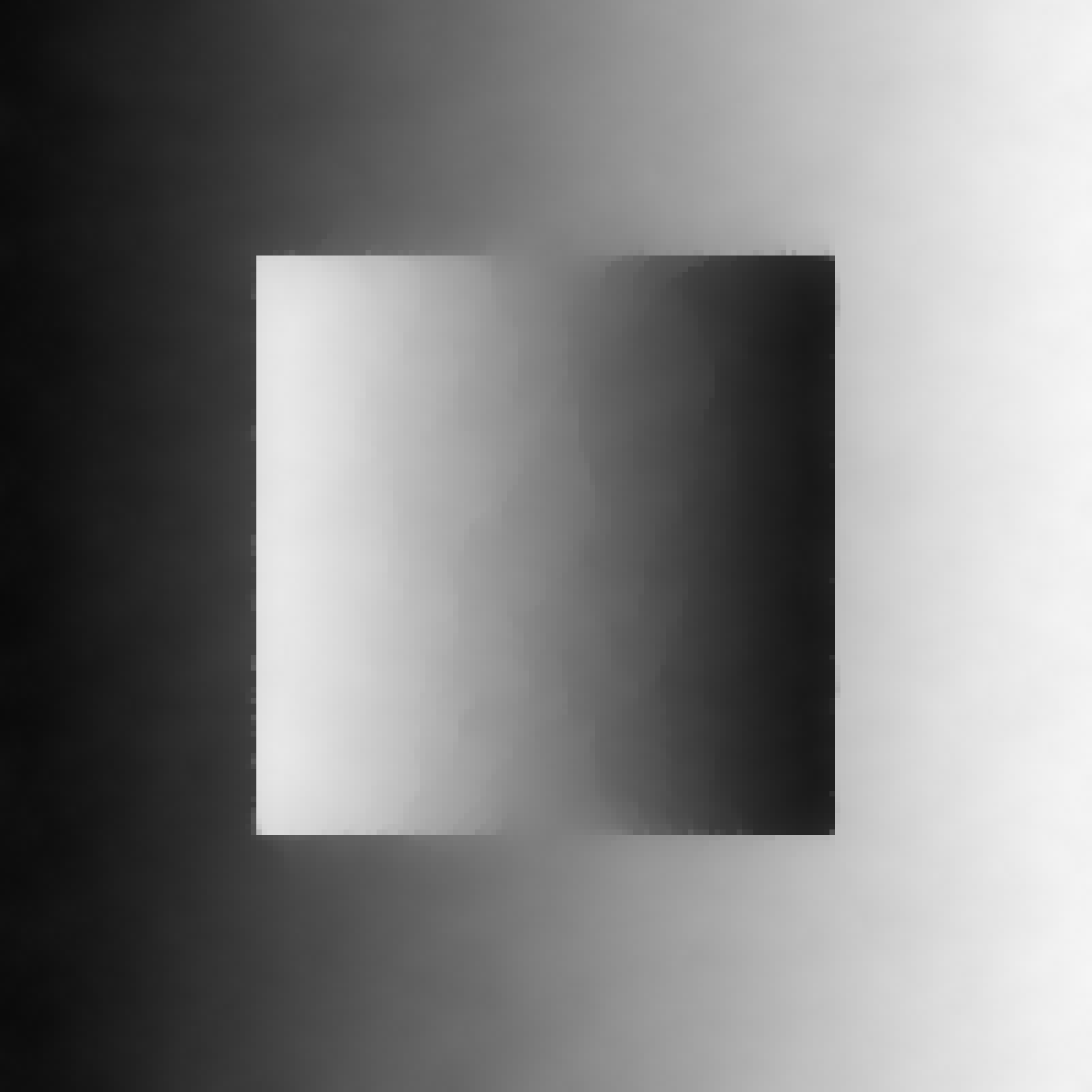}
                 \caption{\centering$\mathrm{TVL}^{3}$: $\alpha=0.3$, $\beta=182$, SSIM=0.9709} 
                 \label{square_2:f}
\end{subfigure}
\end{center}
\caption{Best reconstructions in terms of PSNR and SSIM for $p=\frac{3}{2}, 2, 3$.}
\label{square_2}
\end{figure}

In Figure \ref{square_2}, we present the best reconstructions results in terms of two quality measures, the \emph{Peak Signal to Noise Ratio} (PSNR) and the \emph{Structural Similarity Index} (SSIM), see \cite{Wang} for the definition of the latter. In each case, the  values of $\alpha$ and $\beta$ are selected appropriately for optimal PSNR and SSIM. 
Our stopping criterion is the relative residual error becoming less than $10^{-6}$ i.e.,
\begin{equation}
\frac{\norm{2}{u^{k+1}-u^{k}}}{\norm{2}{u^{k+1}}}\leq 10^{-6}.
\end{equation}
\begin{figure}[b]
\begin{center}
\begin{subfigure}[t]{3.8cm}
                \centering                                                  
                \includegraphics[width=3.8cm]{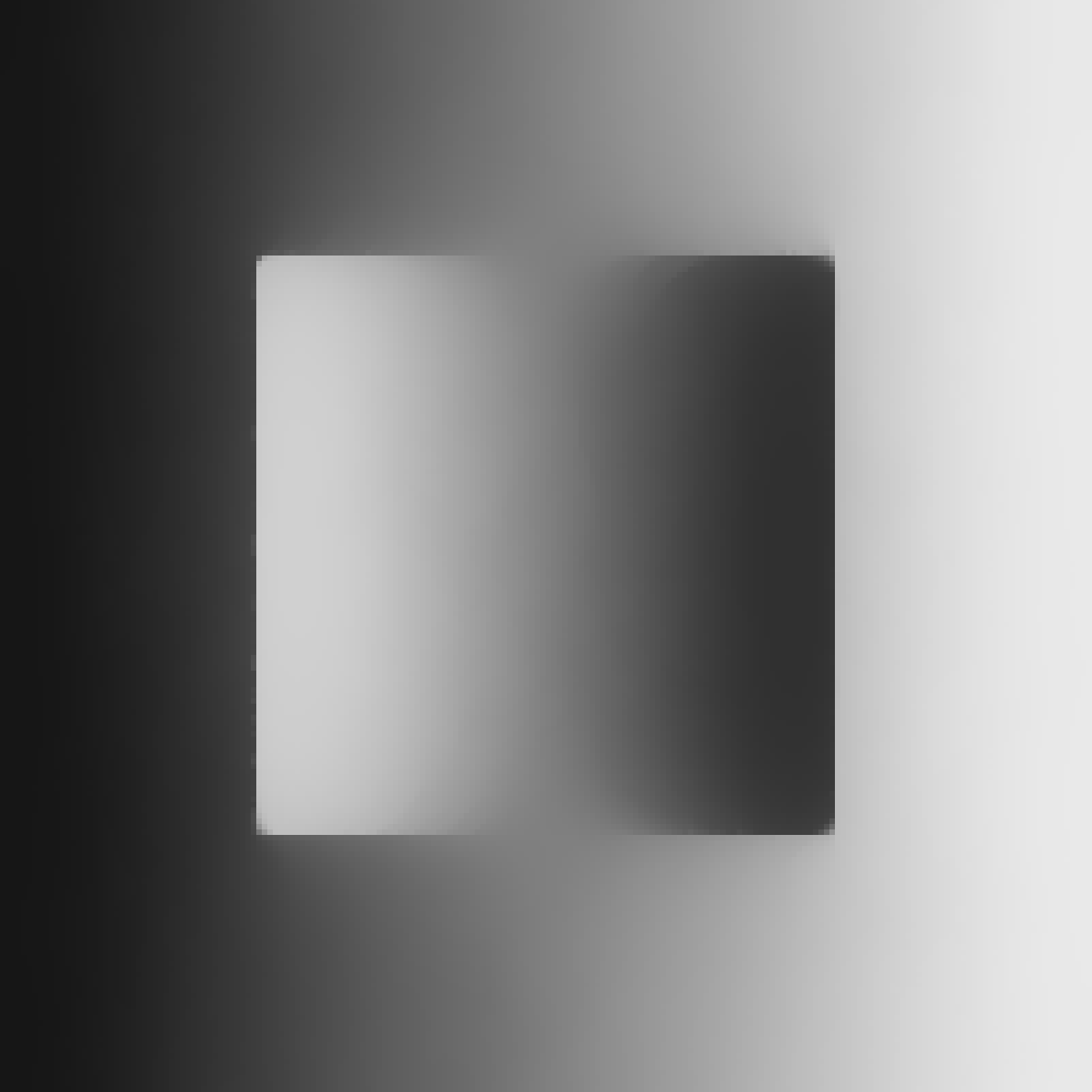}
\end{subfigure}
\begin{subfigure}[t]{3.8cm}
                \centering                                                  
                \includegraphics[width=3.8cm]{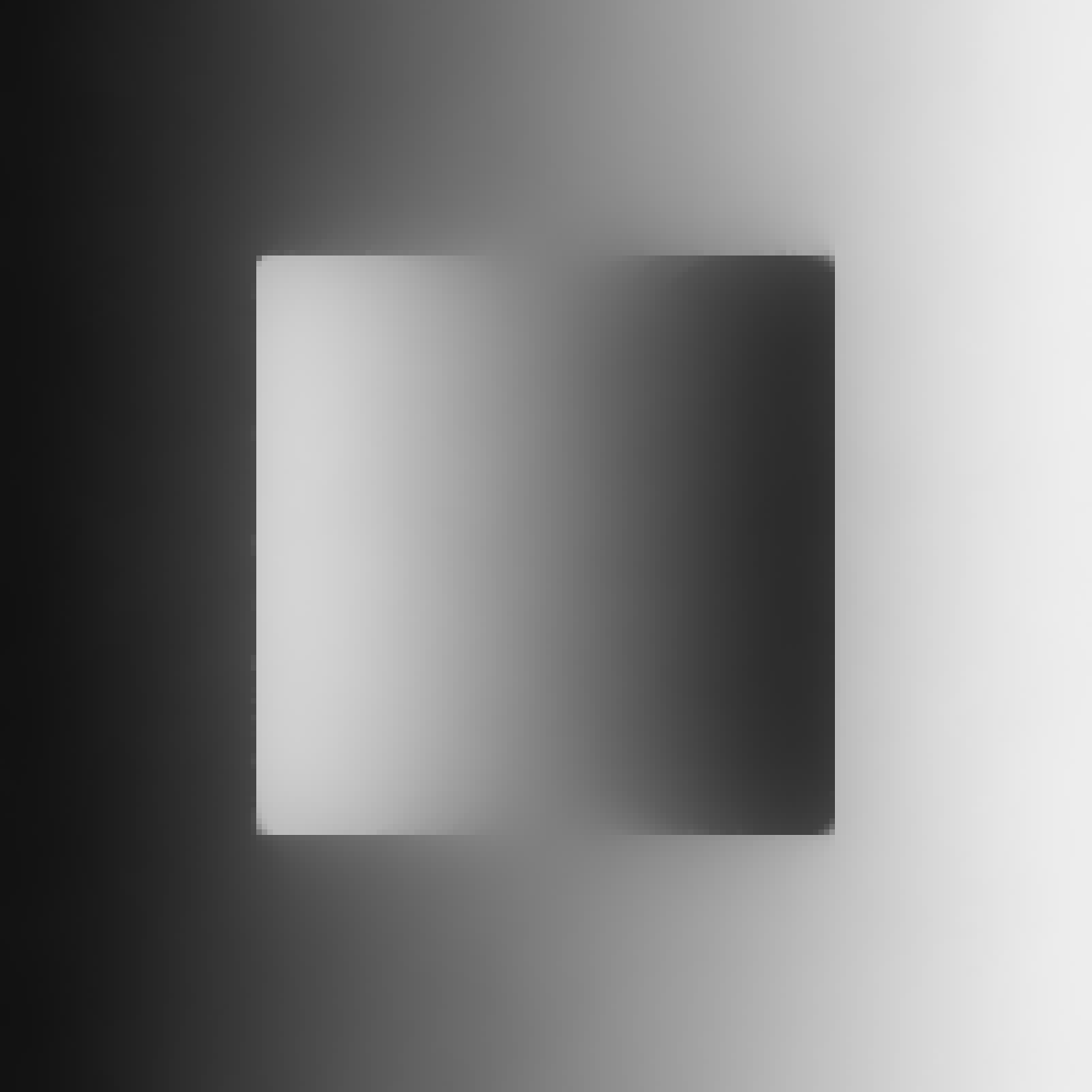}
\end{subfigure}
\begin{subfigure}[t]{3.8cm}
                \centering                                                  
                 \includegraphics[width=3.8cm]{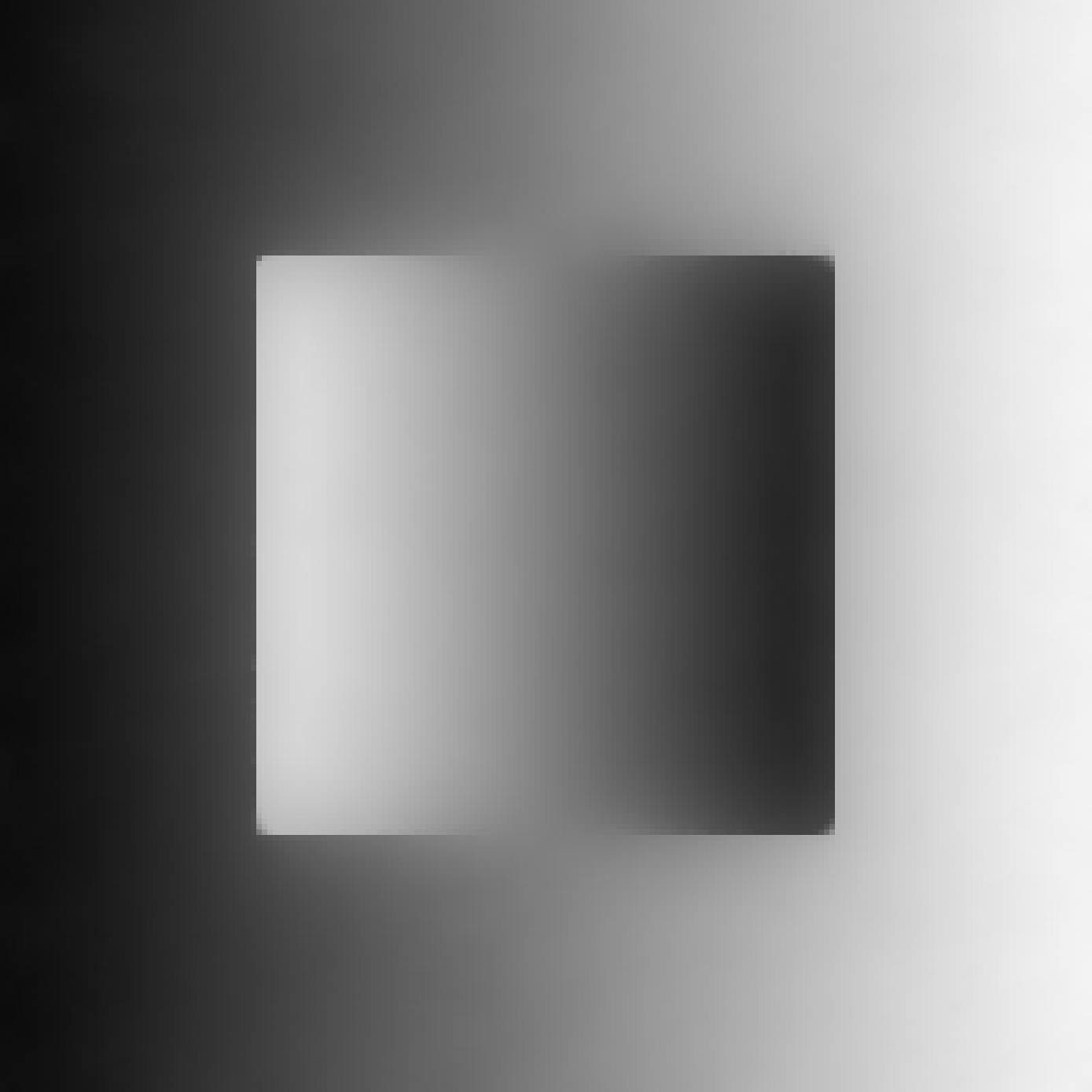}
\end{subfigure}
\begin{subfigure}[t]{3.8cm}
                \centering                                                  
                \includegraphics[width=3.8cm]{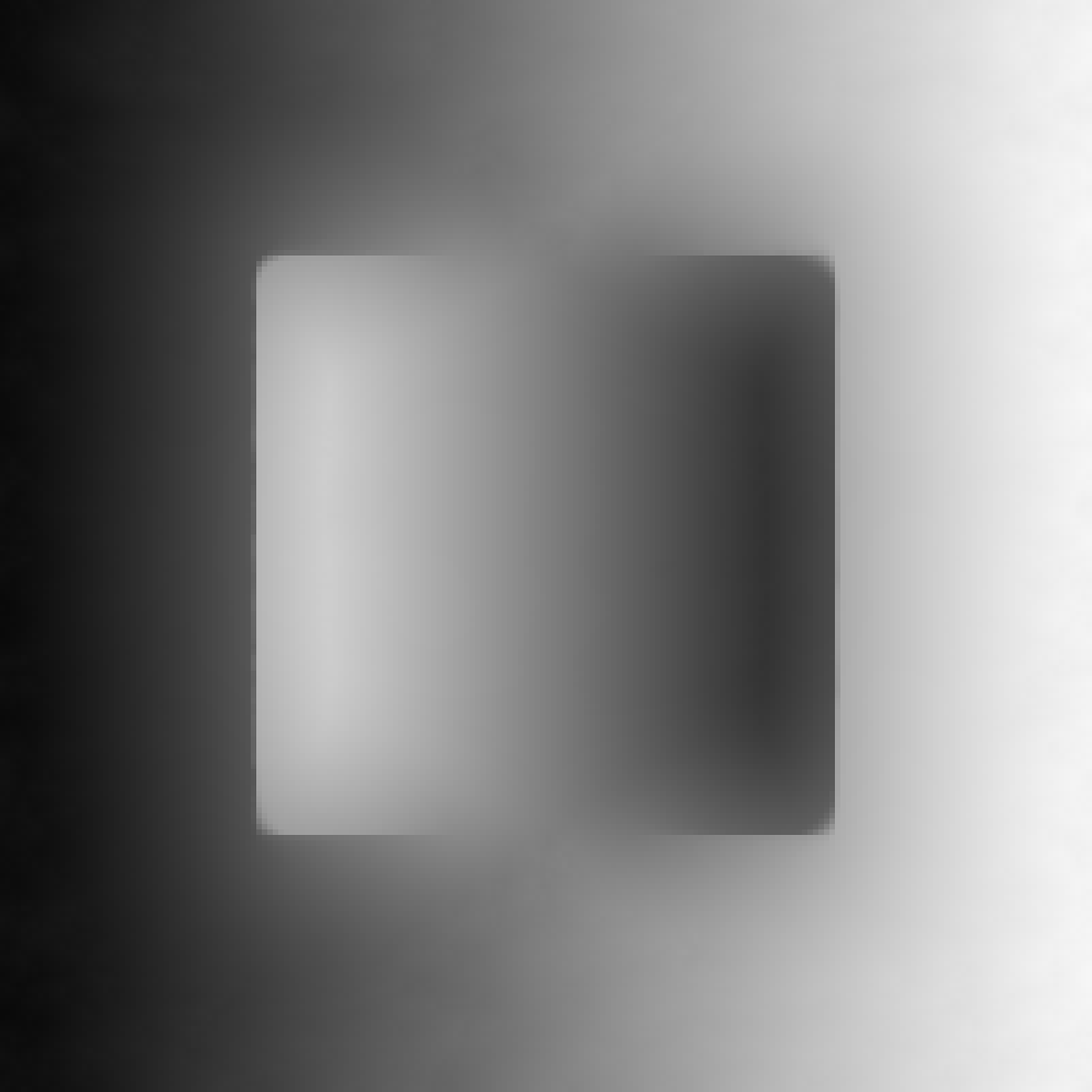} 
\end{subfigure}\\
\begin{subfigure}[t]{3.8cm}
                \centering                                                  
                \includegraphics[width=4cm]{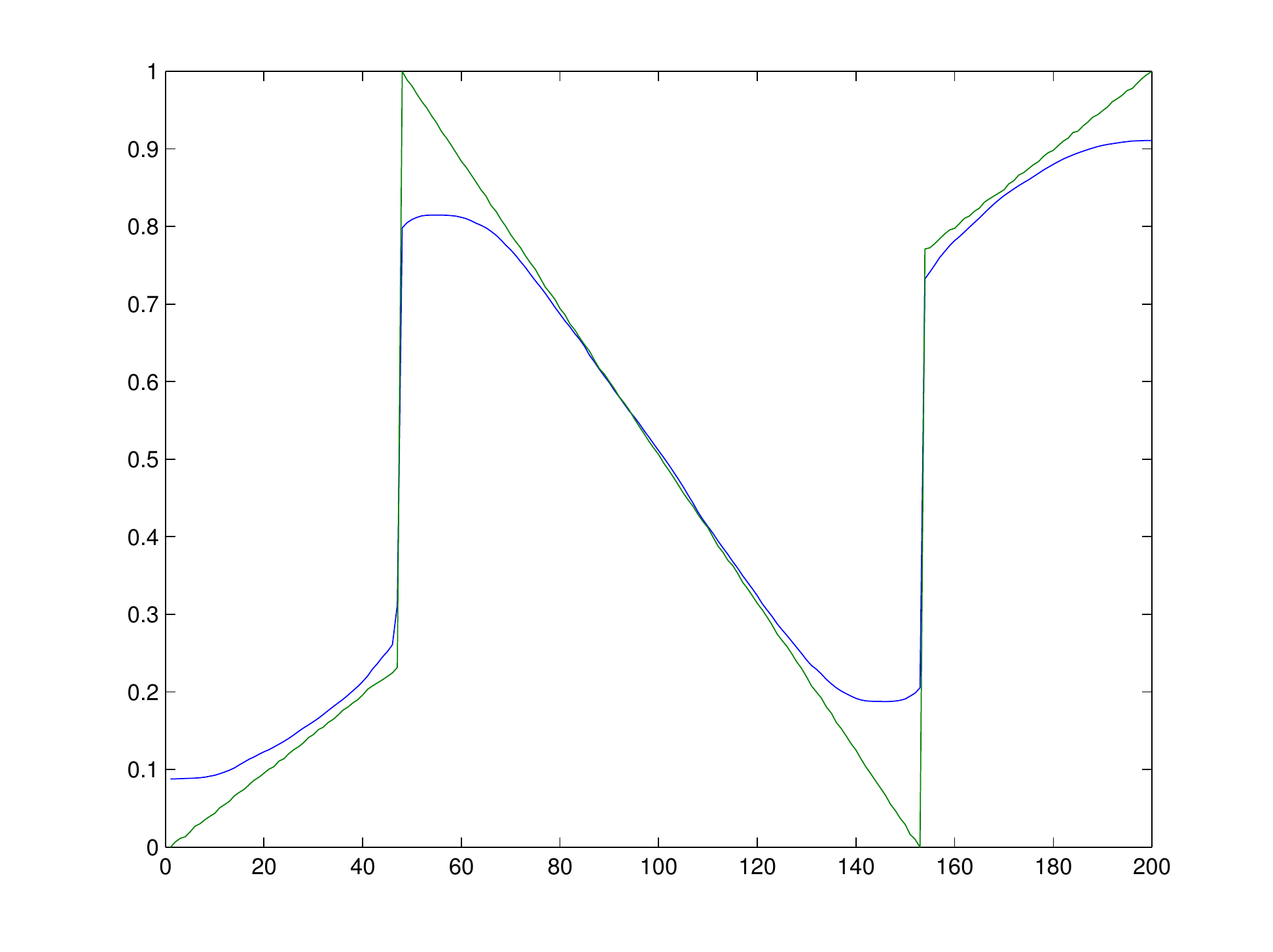}
                 \caption{\centering$\mathrm{TVL}^{\frac{3}{2}}$: $\alpha=1$, $\beta=25$, SSIM=0.9391} 
\end{subfigure}
\begin{subfigure}[t]{3.8cm}
                \centering                                                  
                \includegraphics[width=4cm]{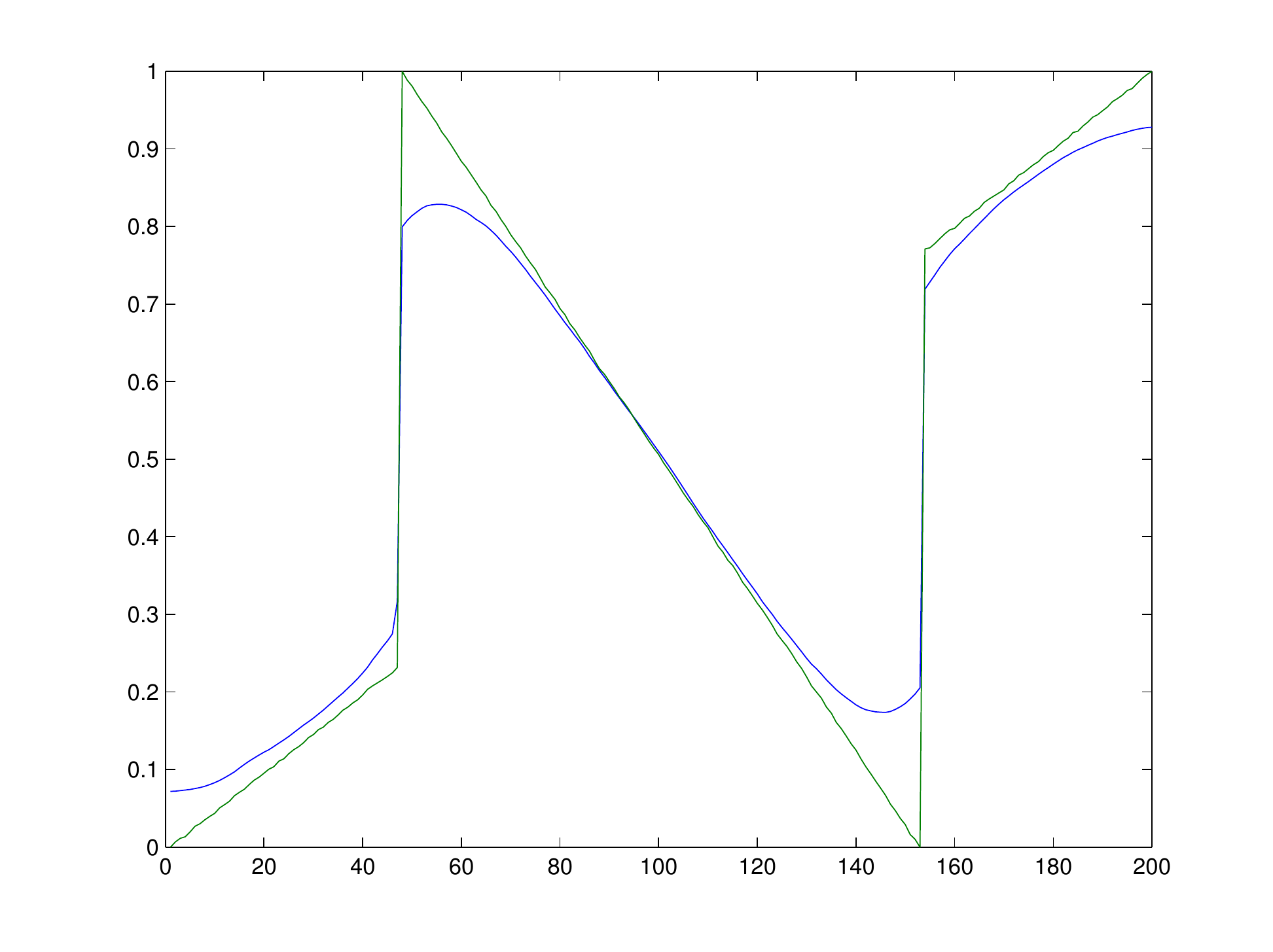}
                 \caption{\centering$\mathrm{TVL}^{2}$: $\alpha=1$, $\beta=116$, SSIM=0.9433} 
\end{subfigure}
\begin{subfigure}[t]{3.8cm}
                \centering                                                  
                \includegraphics[width=4cm]{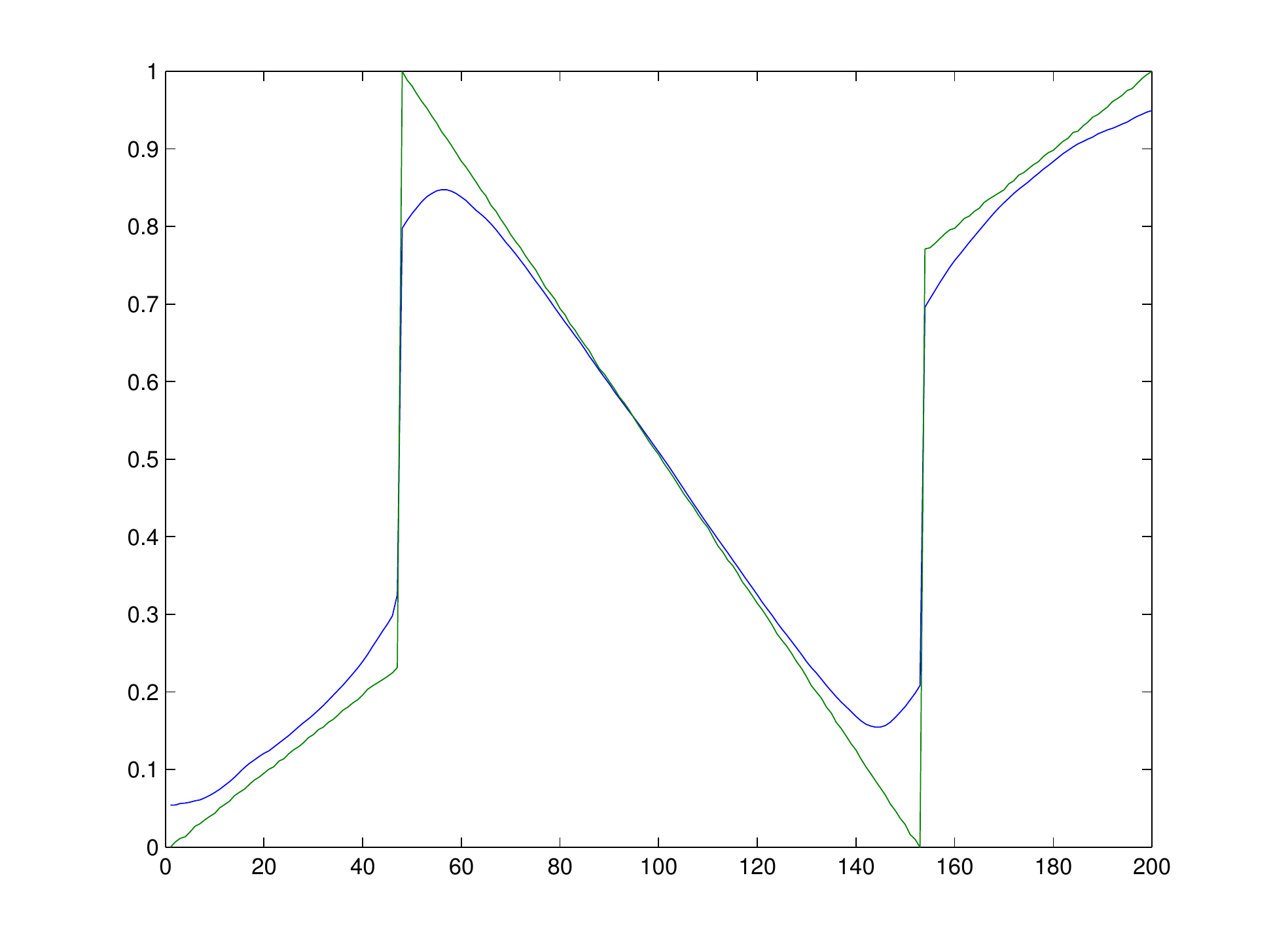}
                 \caption{\centering$\mathrm{TVL}^{3}$: $\alpha=1$, $\beta=438$, SSIM=0.9430} 
\end{subfigure}
\begin{subfigure}[t]{3.8cm}
                \centering                                                  
                \includegraphics[width=4cm]{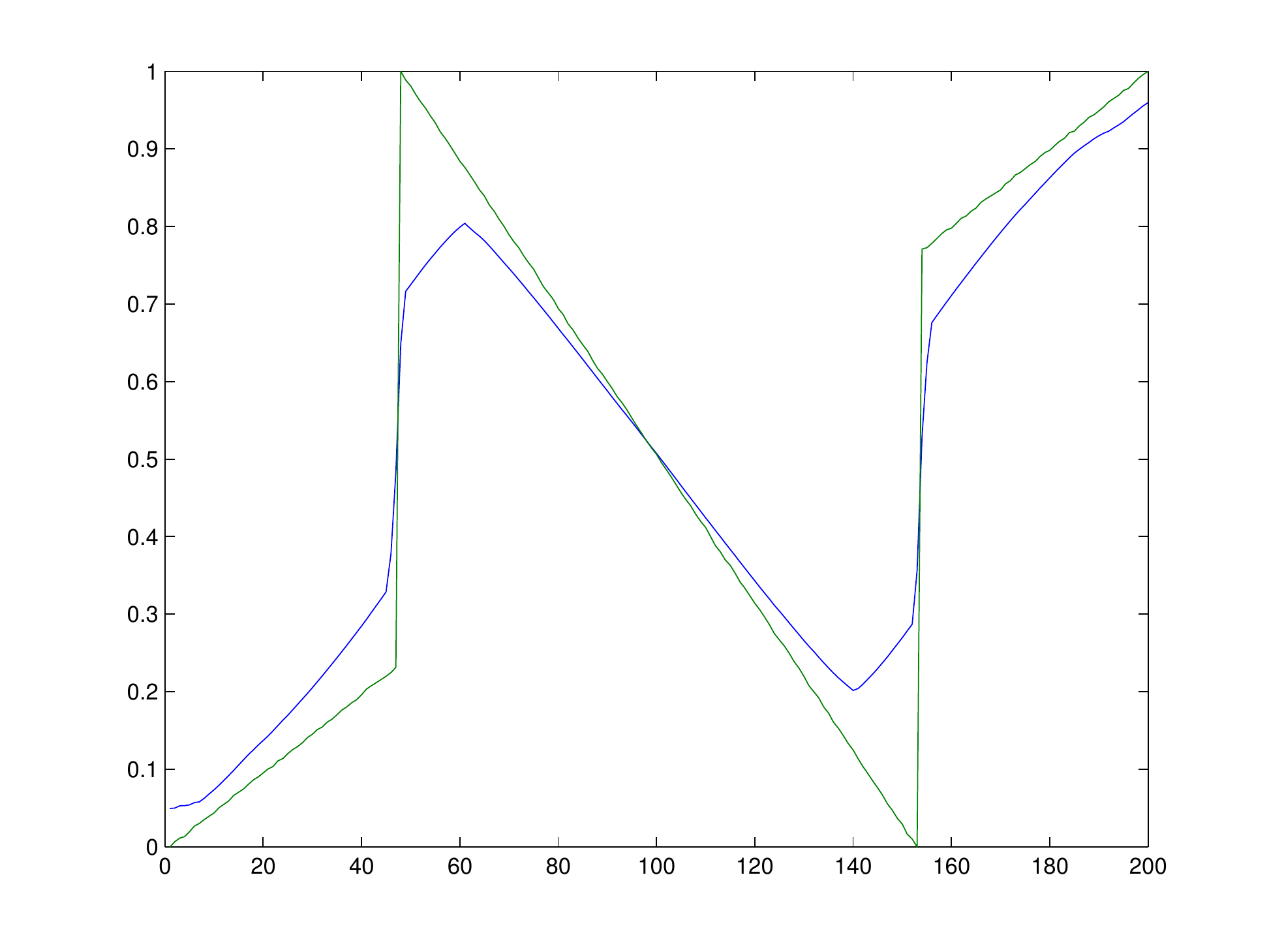}
                 \caption{\centering$\mathrm{TVL}^{7}$: $\alpha=2$, $\beta=5000$, SSIM=0.9001} 
\end{subfigure}
\end{center}
\caption{Staircasing elimination for $p=\frac{3}{2}, 2, 3$ and $7$. High values of $p$ promotes almost affine structures.}
\label{square_3}
\end{figure}
Finally, for computational efficiency, we fix $\lambda=10\alpha$ when $1<p<4$ and $\lambda=1000\alpha$ when $p\geq4$ (empirical rule). We observe that the best reconstructions in terms of the PSNR have no visual difference among $p=\frac{3}{2}, 2$ and 3 and  staircasing is present, Figures \ref{square_2:a}, \ref{square_2:b} and \ref{square_2:c}. This is one more indication that the PSNR -- which is based on the squares of the difference between the ground truth and the reconstruction -- does not correspond to the optimal visual results. However, the best reconstructions in terms of SSIM are visually better. They exhibit significantly reduced staircasing for $p=\frac{3}{2}$ and $p=3$ and is essentially absent in the case of $p=2$, see  Figures \ref{square_2:d}, \ref{square_2:e} and \ref{square_2:f}.

 We can also get a total staircasing elimination by setting higher values for the parameters $\alpha$ and $\beta$, as we show in Figure \ref{square_3}. There, one observes that on one hand as we increase $p$, almost affine structures are promoted -- see the middle row profiles in Figure \ref{square_3} -- and on the other hand these choices of $\alpha, \beta$ produce a serious loss of contrast that however can  be easily treated via the \emph{Bregman iteration}.

Contrast enhancement via Bregman iteration was introduced in \cite{Osher1}, see also \cite{TGVbregman} for an application to higher-order models. It involves solving a modified version of the minimisation problem. Setting $u^{0}=f$, for $k=1,2,\ldots$, we solve 
\begin{equation}
\begin{aligned}
u^{k+1}&=\underset{\substack{u\in\re^{n\times m}\\ w\in(\re^{n\times m})^2}}{\operatorname{argmin}}\; \frac{1}{2}\norm{2}{f+\tilde{v}^{k}-u}^{2}+ \alpha\norm{1}{\nabla u - w} + \beta\norm{p}{w},\\
\tilde{v}^{k+1}&=\tilde{v}^{k}+f-u^{k+1}.
\end{aligned}
\label{breg1}
\end{equation}

\begin{figure}[h]
\begin{center}
\begin{subfigure}[t]{5cm}
                \centering                                                  
                \includegraphics[width=5cm]{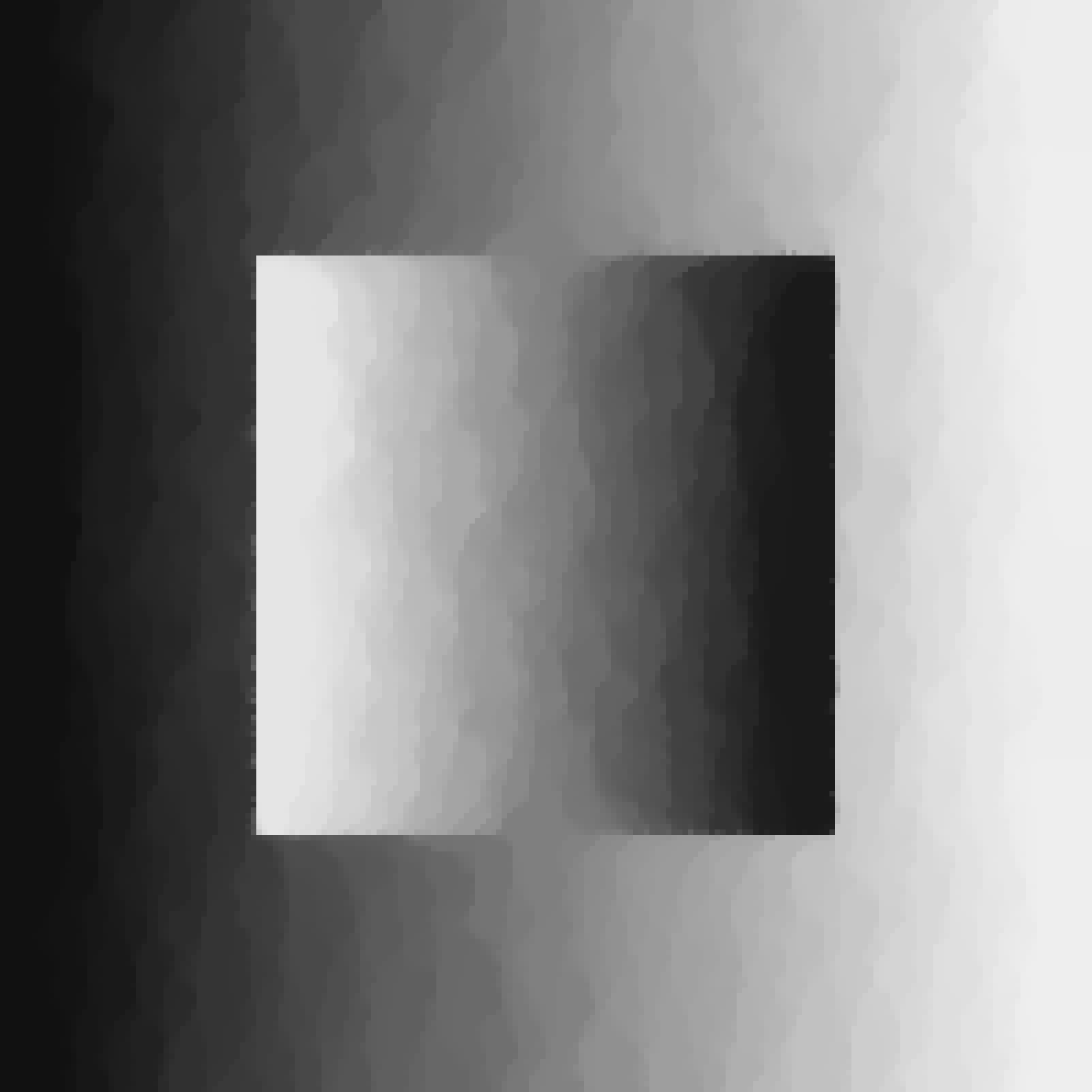}
                \caption{\centering$\mathrm{TV}$: $\alpha=0.2$, SSIM=0.9387} 
\end{subfigure}
\begin{subfigure}[t]{5cm}
                \centering                                                  
                \includegraphics[width=5cm]{no_stair_tvl_2_a_1_b_100-eps-converted-to.pdf}
                \caption{\centering$\mathrm{TVL}^{2}$: $\alpha=1$, $\beta=116$, SSIM=0.9433} 
\end{subfigure}
\begin{subfigure}[t]{5cm}
                \centering                                                  
                 \includegraphics[width=5cm]{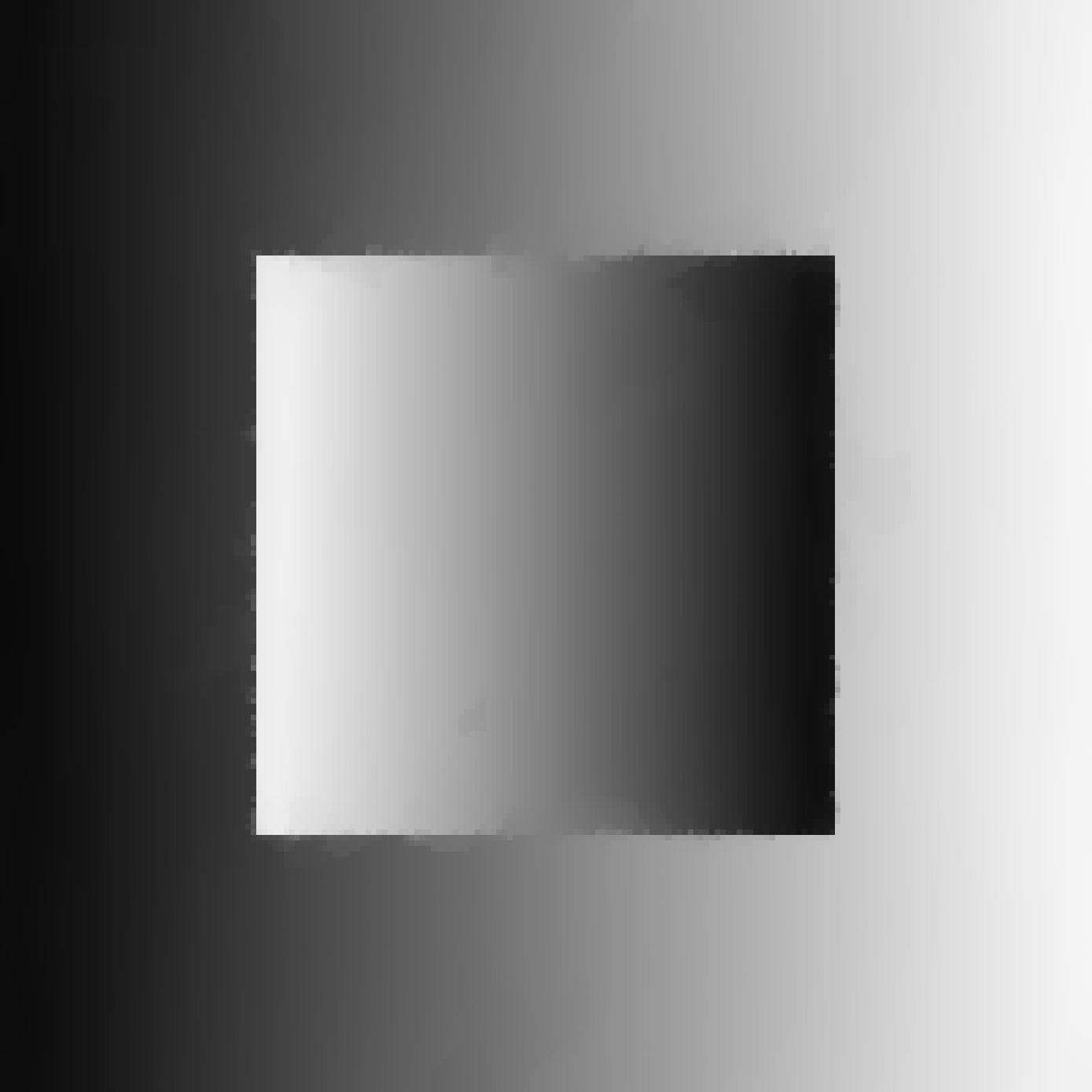}
                 \caption{\centering$\mathrm{TGV}^{2}$: $\alpha=0.12$, $\beta=0.55$, SSIM=0.9861} 
\end{subfigure}
\begin{subfigure}[t]{5cm}
                \centering                                                  
                \includegraphics[width=5cm]{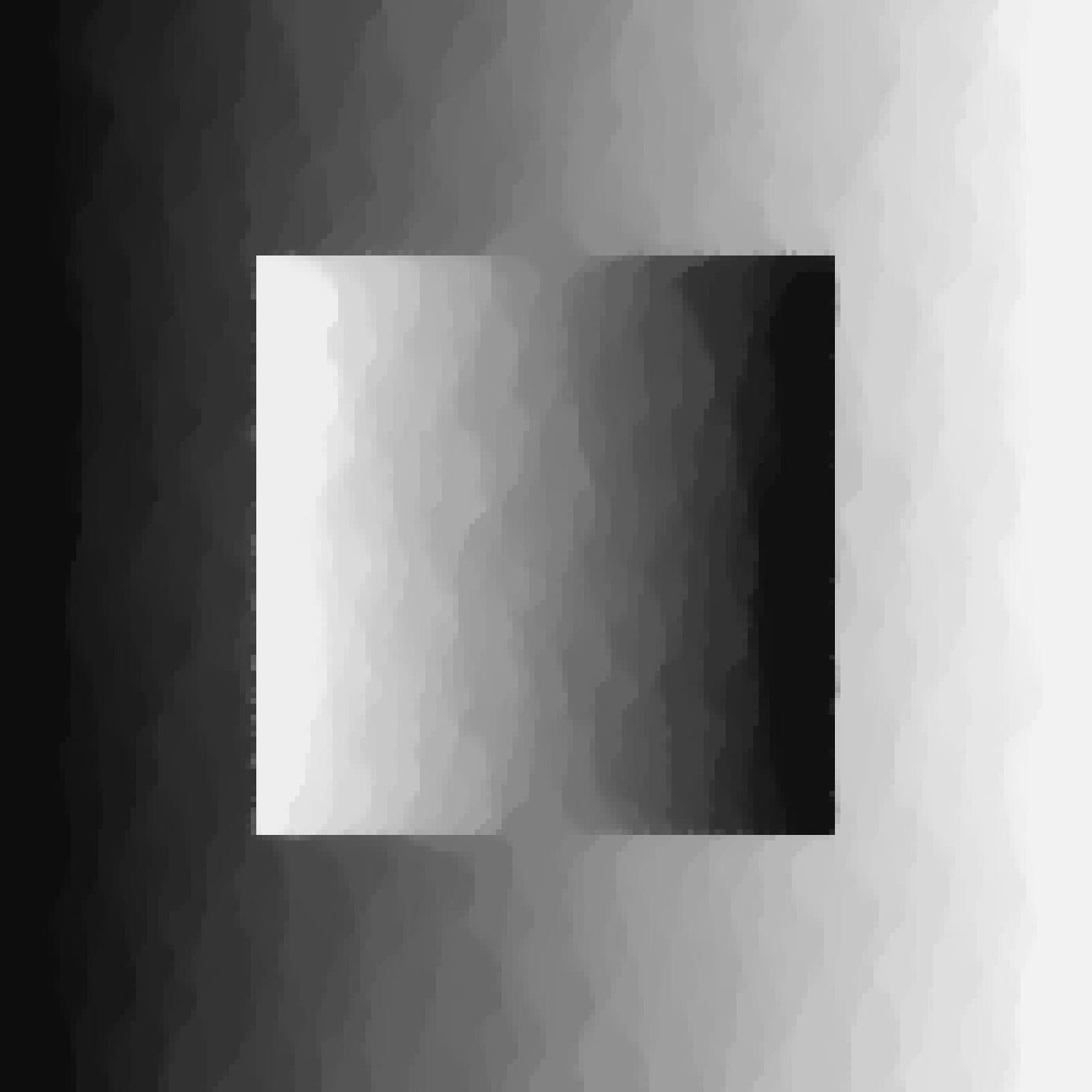}
                \caption{\centering Bregmanised $\mathrm{TV}$: $\alpha=1$, SSIM=0.9401, 4th iteration} 
\end{subfigure}
\begin{subfigure}[t]{5cm}
                \centering                                                  
                \includegraphics[width=5cm]{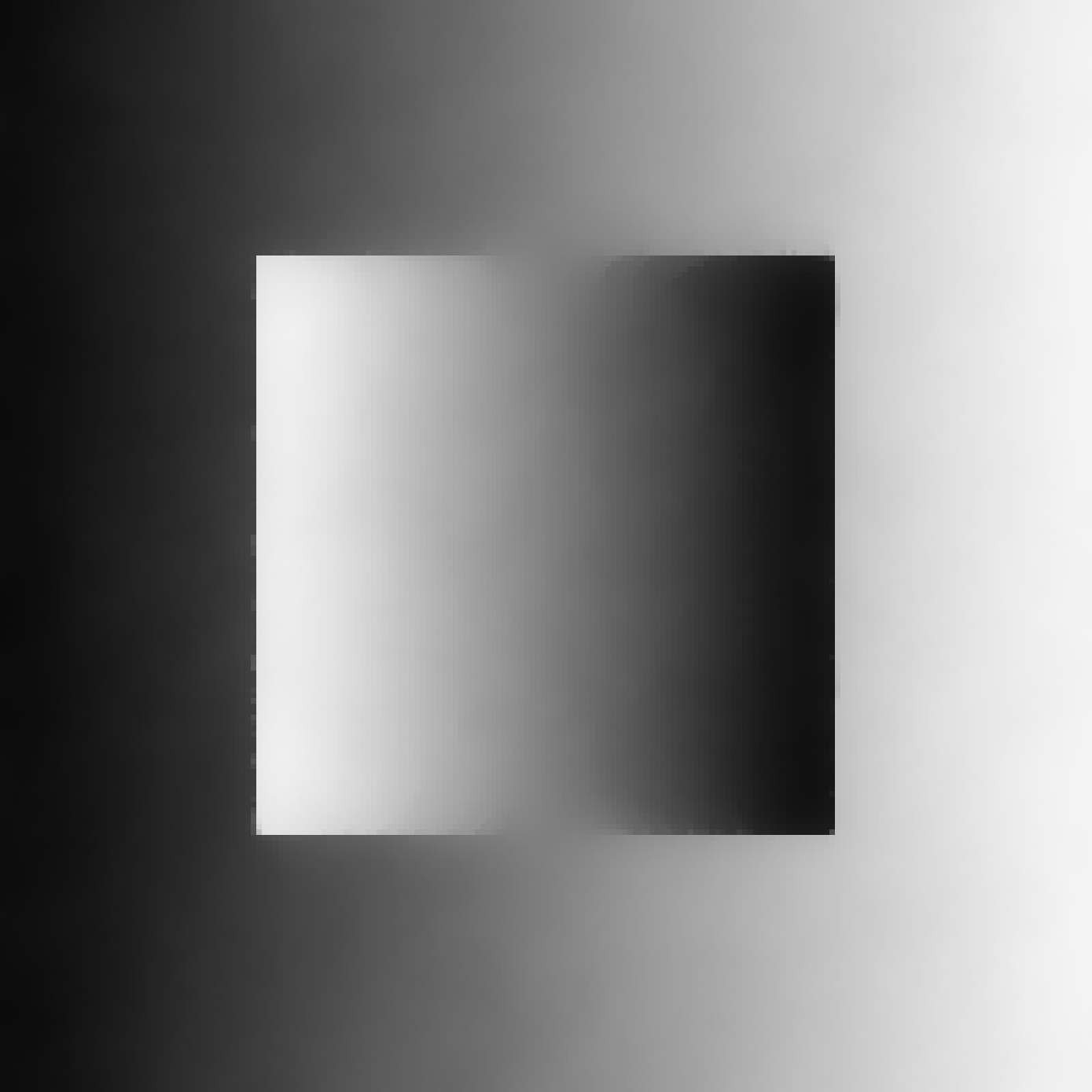}
                \caption{\centering Bregmanised $\mathrm{TVL}^{2}$: $\alpha=2$, $\beta=220$, SSIM=0.9778, 4th iteration} 
                \label{square_4:e}
\end{subfigure}
\begin{subfigure}[t]{5cm}
                \centering                                                  
                \includegraphics[width=5cm]{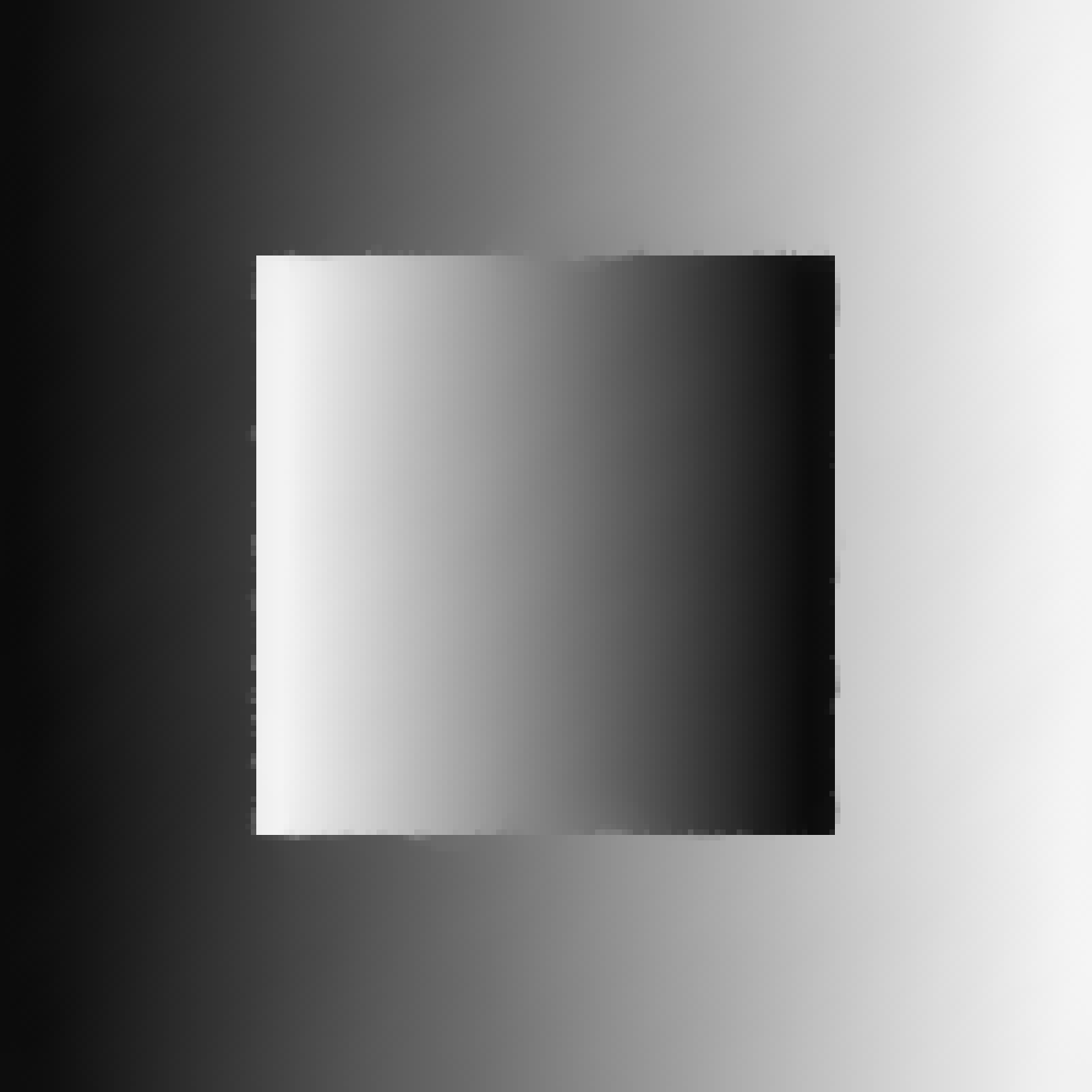}
                 \caption{\centering Bregmanised $\mathrm{TGV}^{2}$: $\alpha=2$, $\beta=10$, SSIM=0.9889, 8th iteration} 
                 \label{square_4:f}
\end{subfigure}
\end{center}
\caption{First Row: Best reconstruction in terms of SSIM for $\mathrm{TV}$, $\mathrm{TVL}^{2}$ and $\mathrm{TGV}^{2}$. Second Row: Best reconstruction in terms of SSIM for Bregmanised $\mathrm{TV}$, $\mathrm{TVL}^{2}$ and $\mathrm{TGV}^{2}$.}
\label{square_4}
\end{figure}
Instead of solving \eqref{discrete_tvlp} once for fixed $\alpha$ and $\beta$, we solve a sequence of similar problems adding back a noisy residual in each iteration which results to a contrast improvement. For stopping criteria regarding the Bregman iteration we refer to \cite{Osher1}. In Figure \ref{square_4} we present our best Bregmanised results in terms of SSIM. There, we notice that Bregman iteration  leads to a significant contrast improvement, in comparison to the results of Figure \ref{square_3}. In fact, we observe that the Bregmanised $\mathrm{TVL}^{2}$ (first-order regularisation), can achieve reconstructions that are visually close to the second-order Bregmanised $\mathrm{TGV}^{2}$, compare Figures \ref{square_4:e} and \ref{square_4:f}. The second-order $\mathrm{TGV}^{2}$ and Bregmanised $\mathrm{TGV}^{2}$ are solved using the Chambolle--Pock primal-dual method \cite{ChambollePock}.

We continue our experimental analysis with a radially symmetric image, see Figure \ref{radial_1}. In Figure \ref{radial_2}, we demonstrate that we can achieve staircasing-free reconstructions for $p=\frac{3}{2}, 2, 3$ and $7$. In fact, as we increase $p$, we obtain results that  preserve the spike in the centre of the circle, see Figure \ref{radial_2:d}. This provides us with another motivation to examine the $p=\infty$ case in \cite{partII}. The loss of contrast can be  again treated using the Bregman iteration \eqref{breg1}. The best results of the latter in terms of SSIM are presented in Figure \ref{radial_3}, for $p=2, 4$ and $7$ and they are also compared with the corresponding  Bregmanised $\mathrm{TGV}^{2}$.  We observe that we can obtain reconstructions that are visually close to the $\mathrm{TGV}^{2}$ ones and in fact notice that for $p=7$, the spike on the centre of the circle is better reconstructed compared to $\mathrm{TGV}^{2}$, see also the surface plots in Figure \ref{radial_3_surf}. 

\begin{figure}[h]
\begin{center}
\begin{subfigure}[h]{6cm}
                \centering                                                  
                \includegraphics[width=3.8cm]{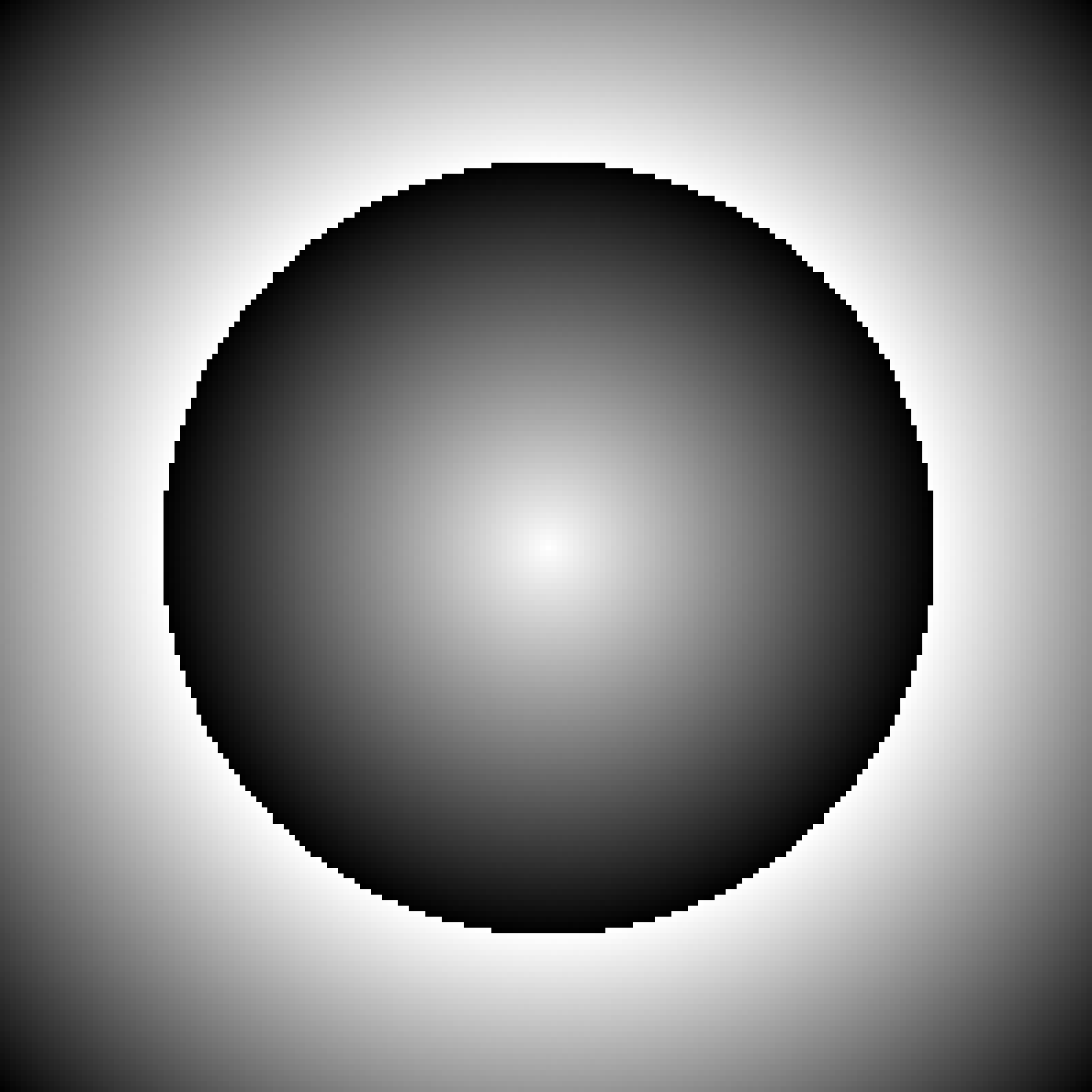}
                \caption{Circle} 
\end{subfigure}
\begin{subfigure}[h]{6cm}
                \centering                                                  
                \includegraphics[width=3.8cm]{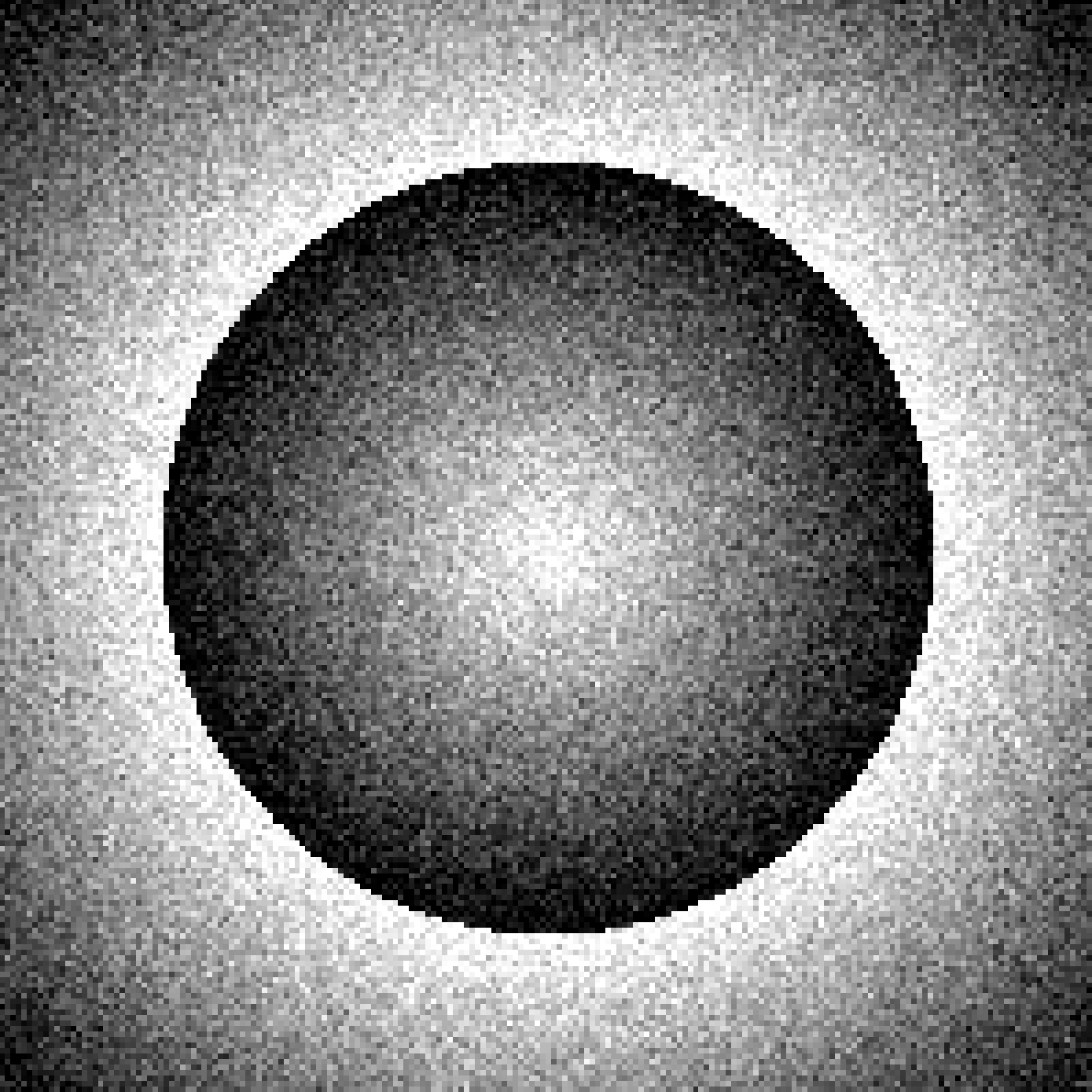}
                \caption{Noisy circle: SSIM=0.2457} 
\end{subfigure}
\end{center}
\caption{Image with symmetric radial structures and its noisy version with $\sigma=0.01$.}
\label{radial_1}
\end{figure}

\begin{figure}[h]
\begin{center}
\begin{subfigure}[t]{3.8cm}
                \centering                                                  
                \includegraphics[width=3.8cm]{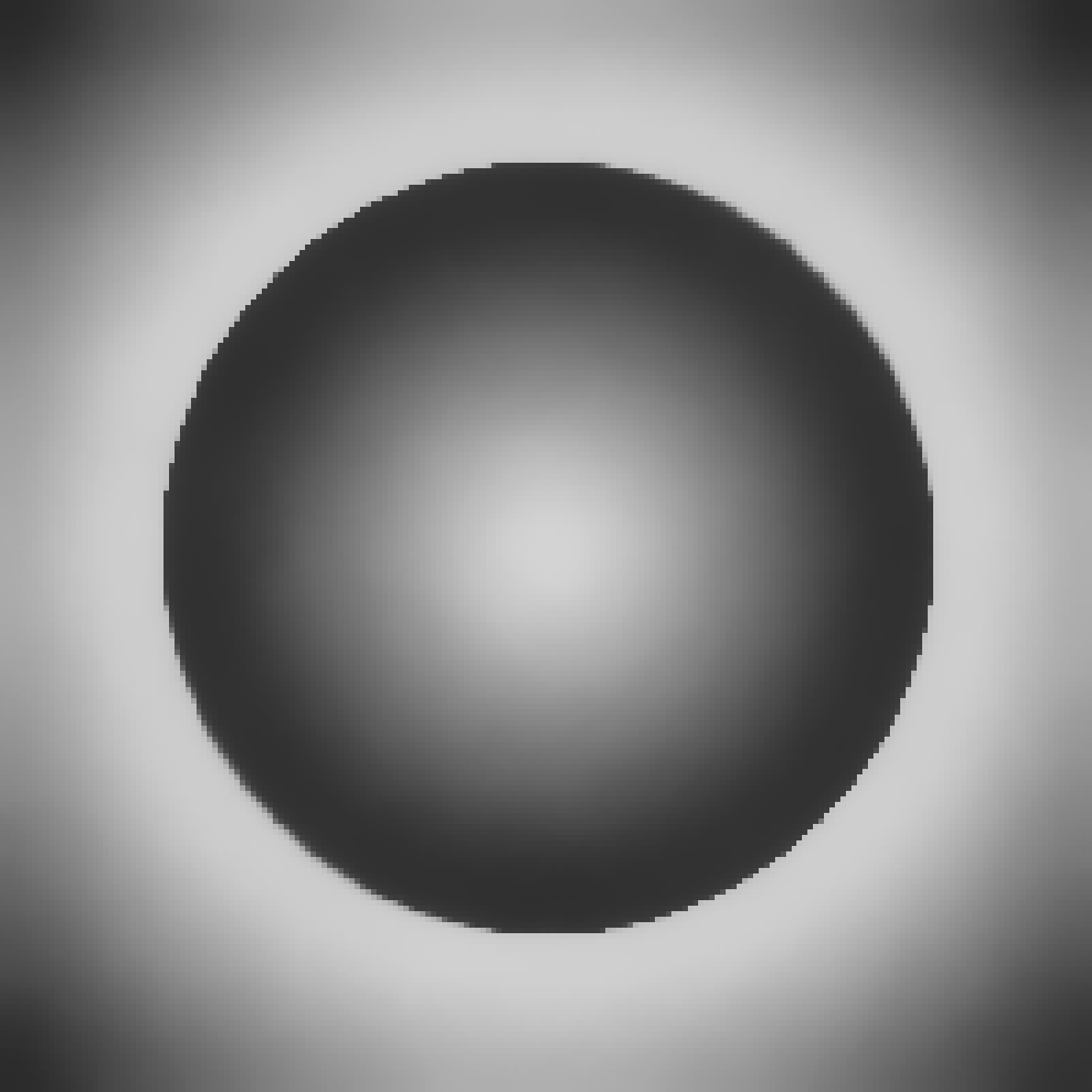}
\end{subfigure}
\begin{subfigure}[t]{3.8cm}
                \centering                                                  
                \includegraphics[width=3.8cm]{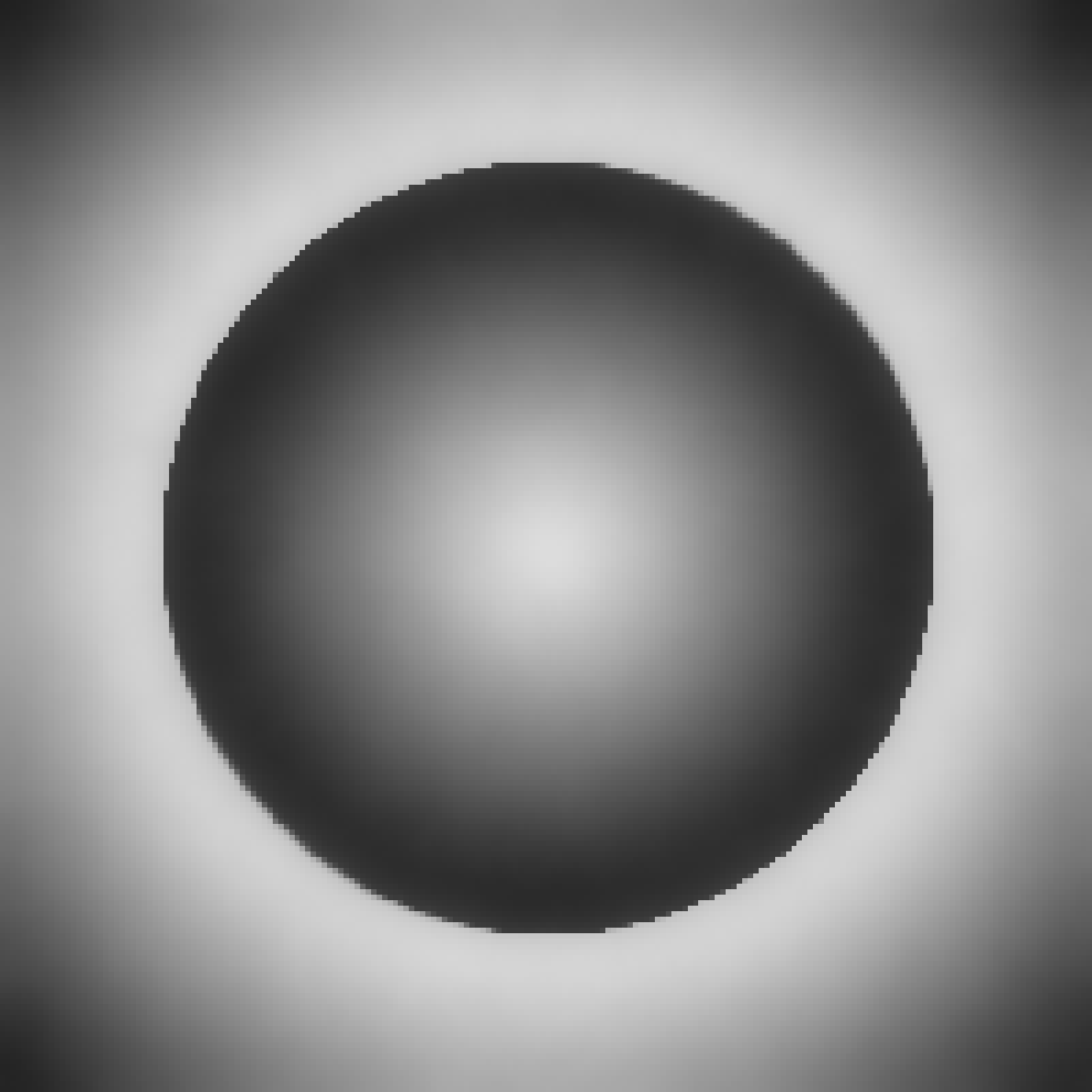}
\end{subfigure}
\begin{subfigure}[t]{3.8cm}
                \centering                                                  
                 \includegraphics[width=3.8cm]{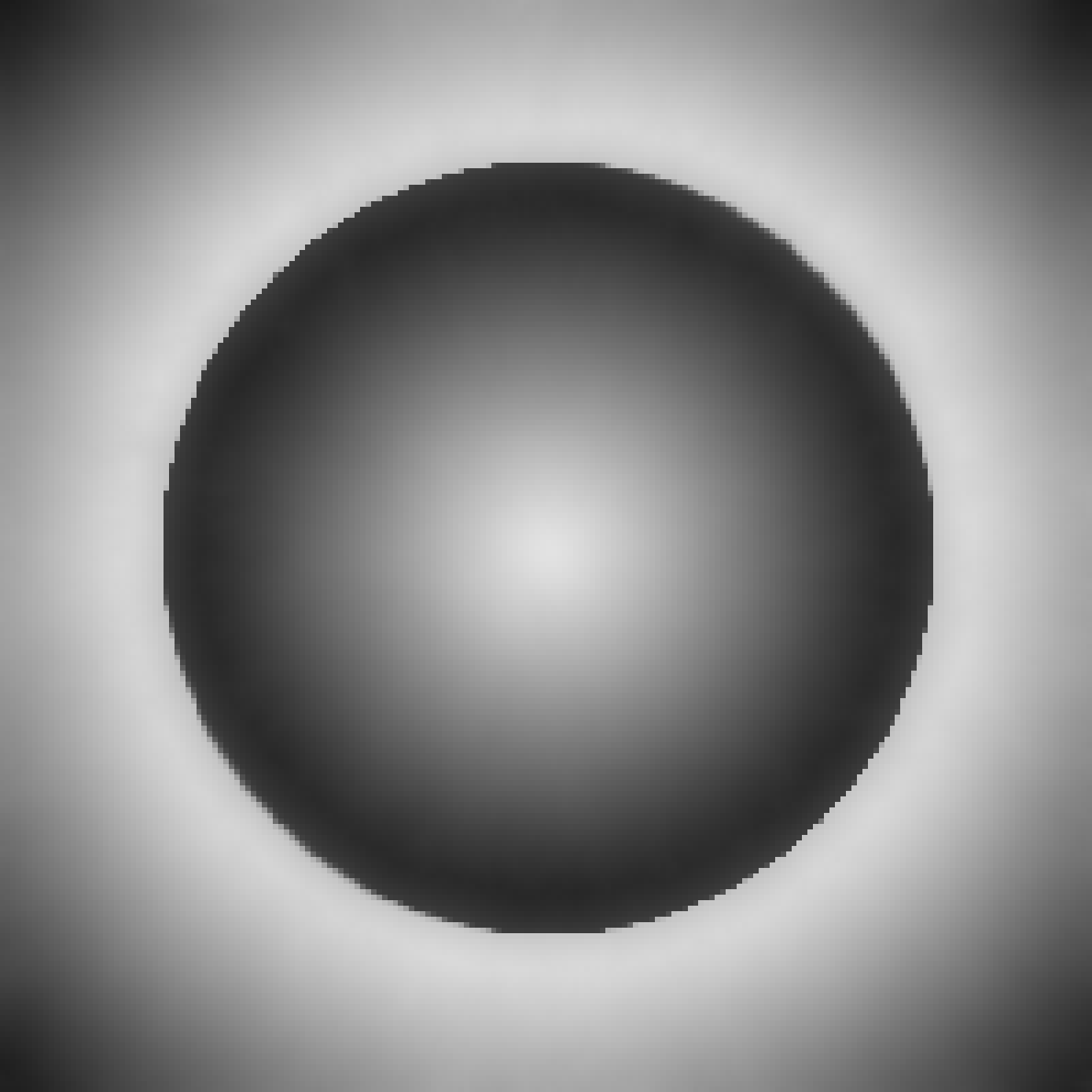}
\end{subfigure}
\begin{subfigure}[t]{3.8cm}
                \centering                                                  
                \includegraphics[width=3.8cm]{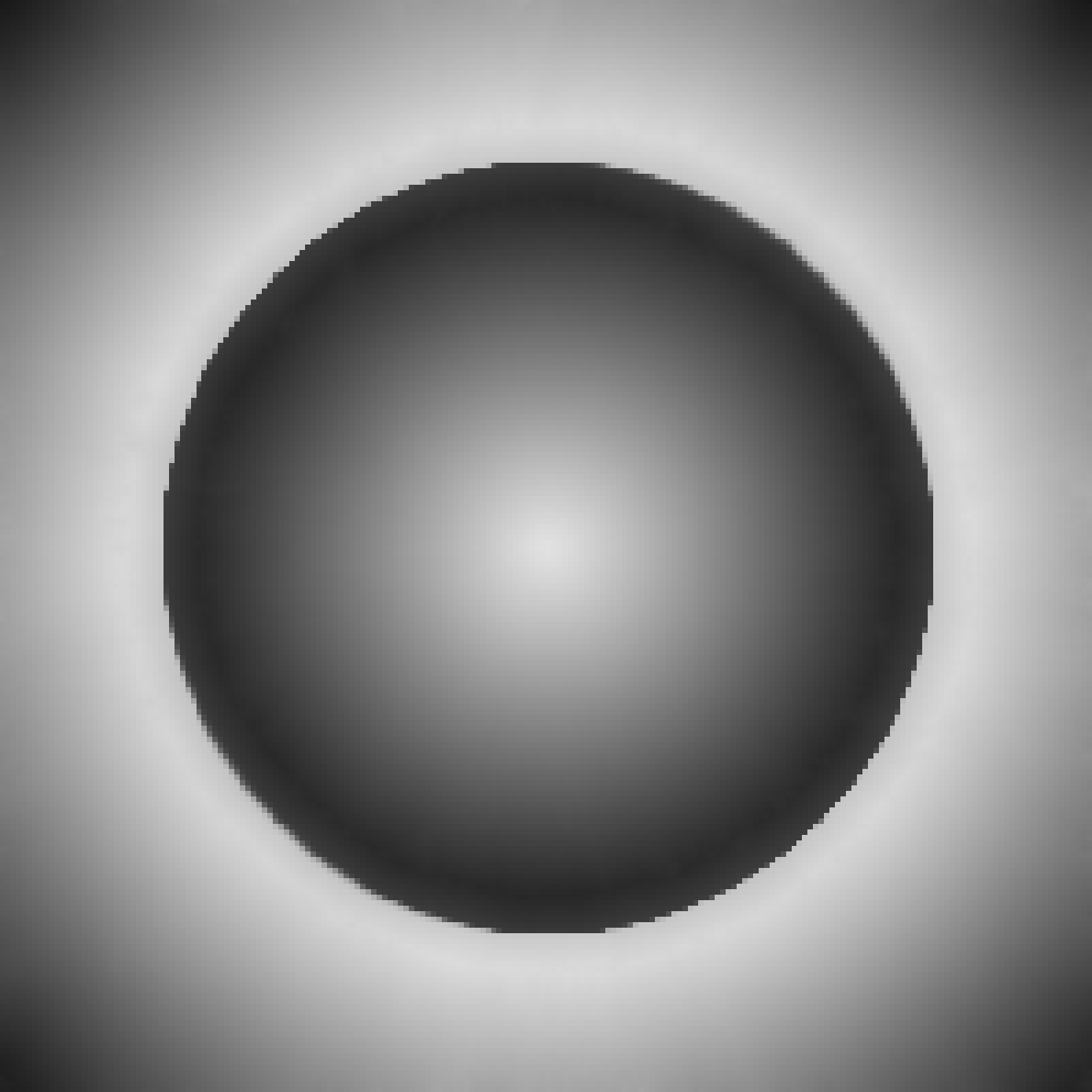} 
\end{subfigure}\\
\begin{subfigure}[t]{3.8cm}
                \centering                                                  
                \includegraphics[width=3.8cm]{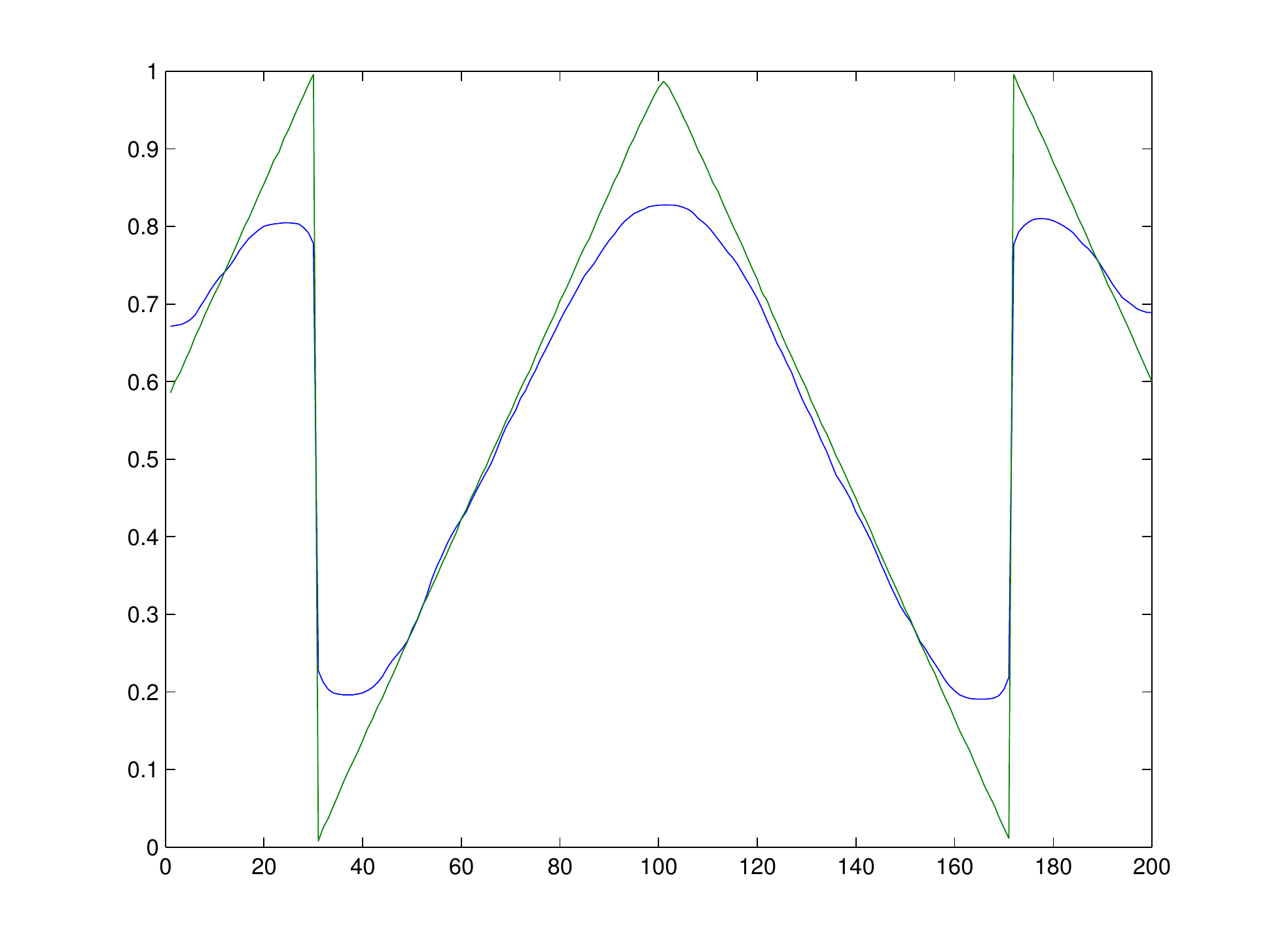}
                 \caption{\centering$\mathrm{TVL}^{\frac{3}{2}}$: $\alpha=0.8$, $\beta=17$, SSIM=0.8909} 
\end{subfigure}
\begin{subfigure}[t]{3.8cm}
                \centering                                                  
                \includegraphics[width=3.8cm]{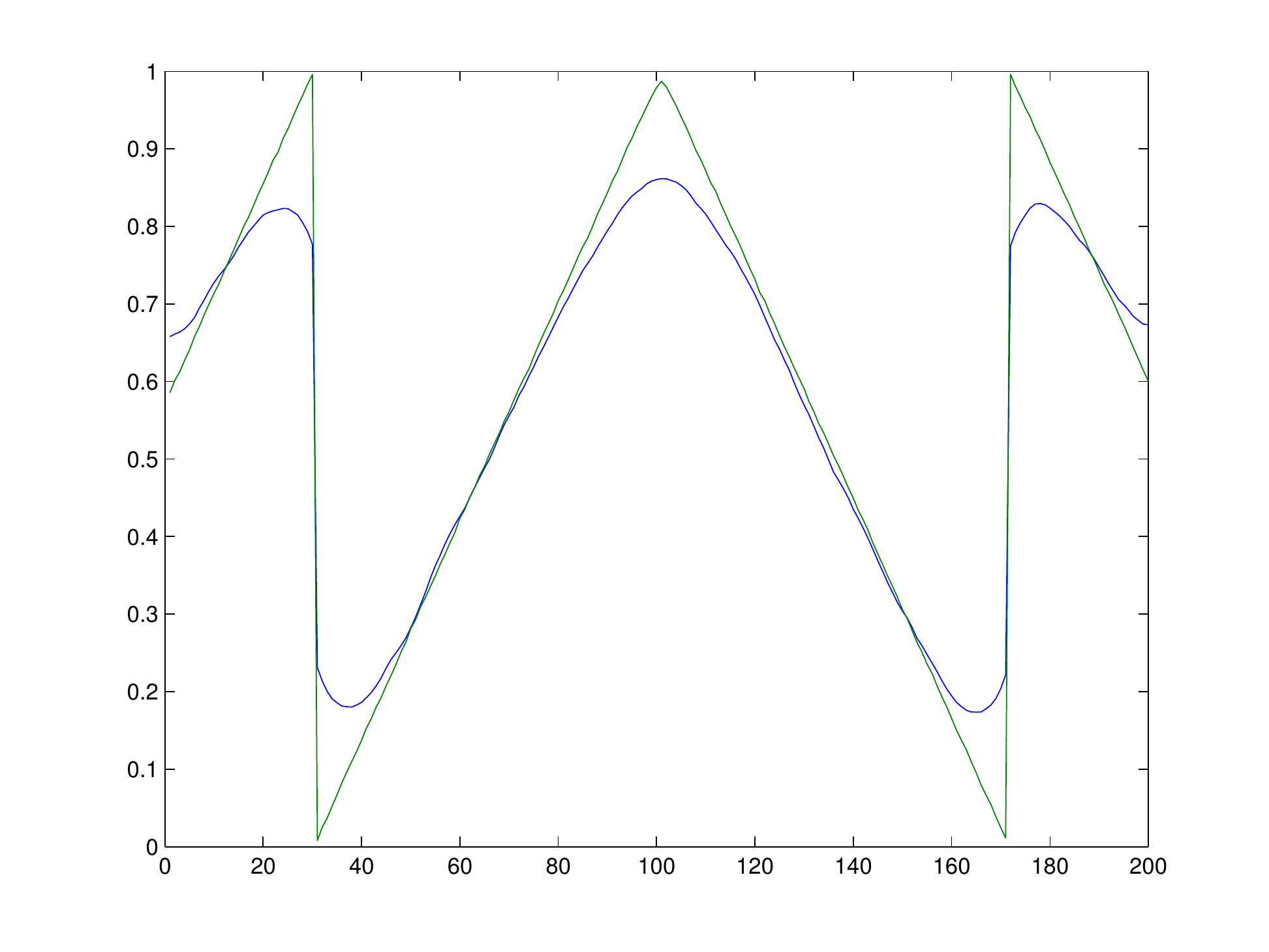}
                 \caption{\centering$\mathrm{TVL}^{2}$: $\alpha=0.8$, $\beta=79$, SSIM=0.8998} 
\end{subfigure}
\begin{subfigure}[t]{3.8cm}
                \centering                                                  
                \includegraphics[width=3.8cm]{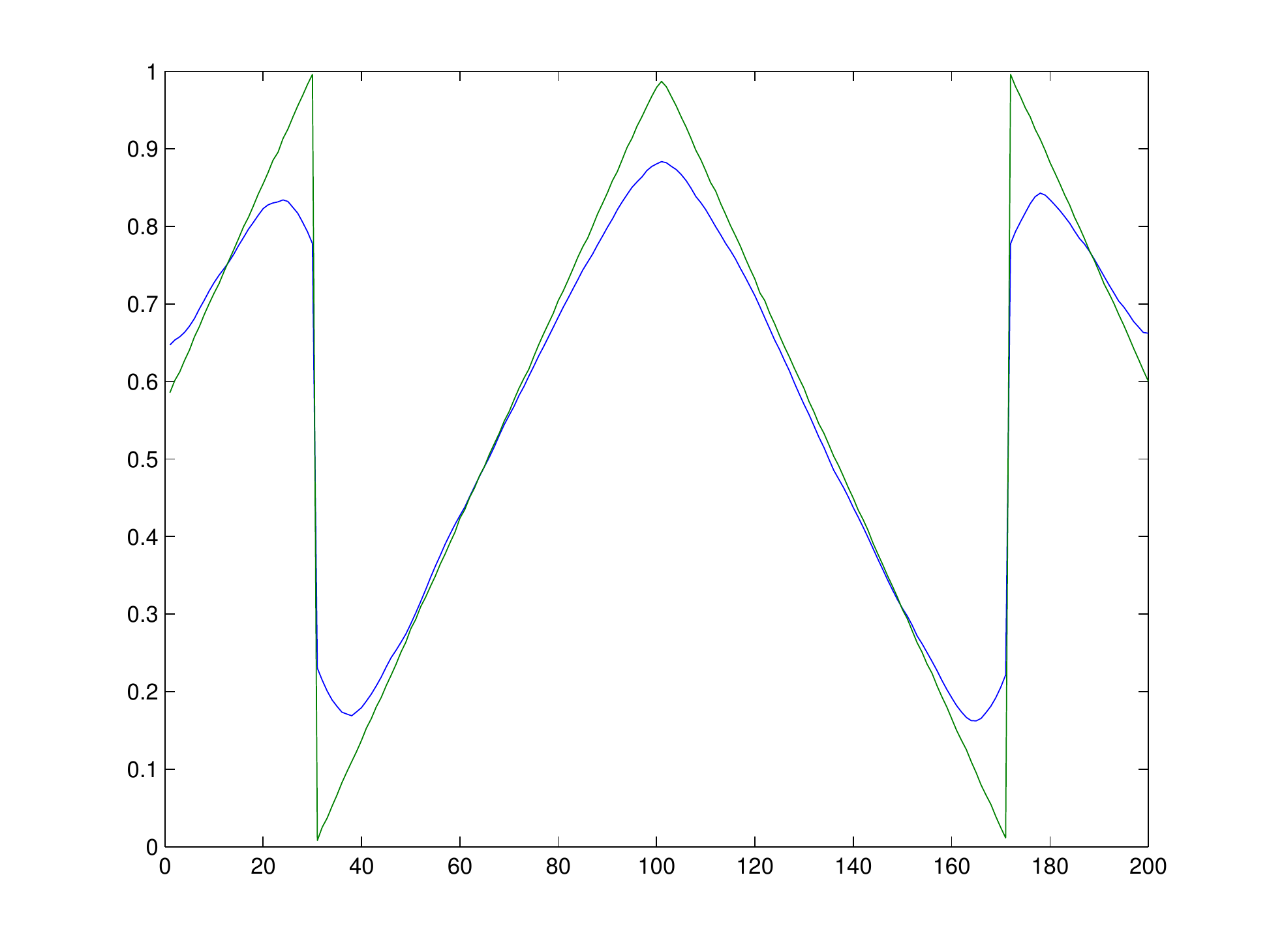}
                 \caption{\centering$\mathrm{TVL}^{3}$: $\alpha=0.8$, $\beta=405$, SSIM=0.9019} 
\end{subfigure}
\begin{subfigure}[t]{3.8cm}
                \centering                                                  
                \includegraphics[width=3.8cm]{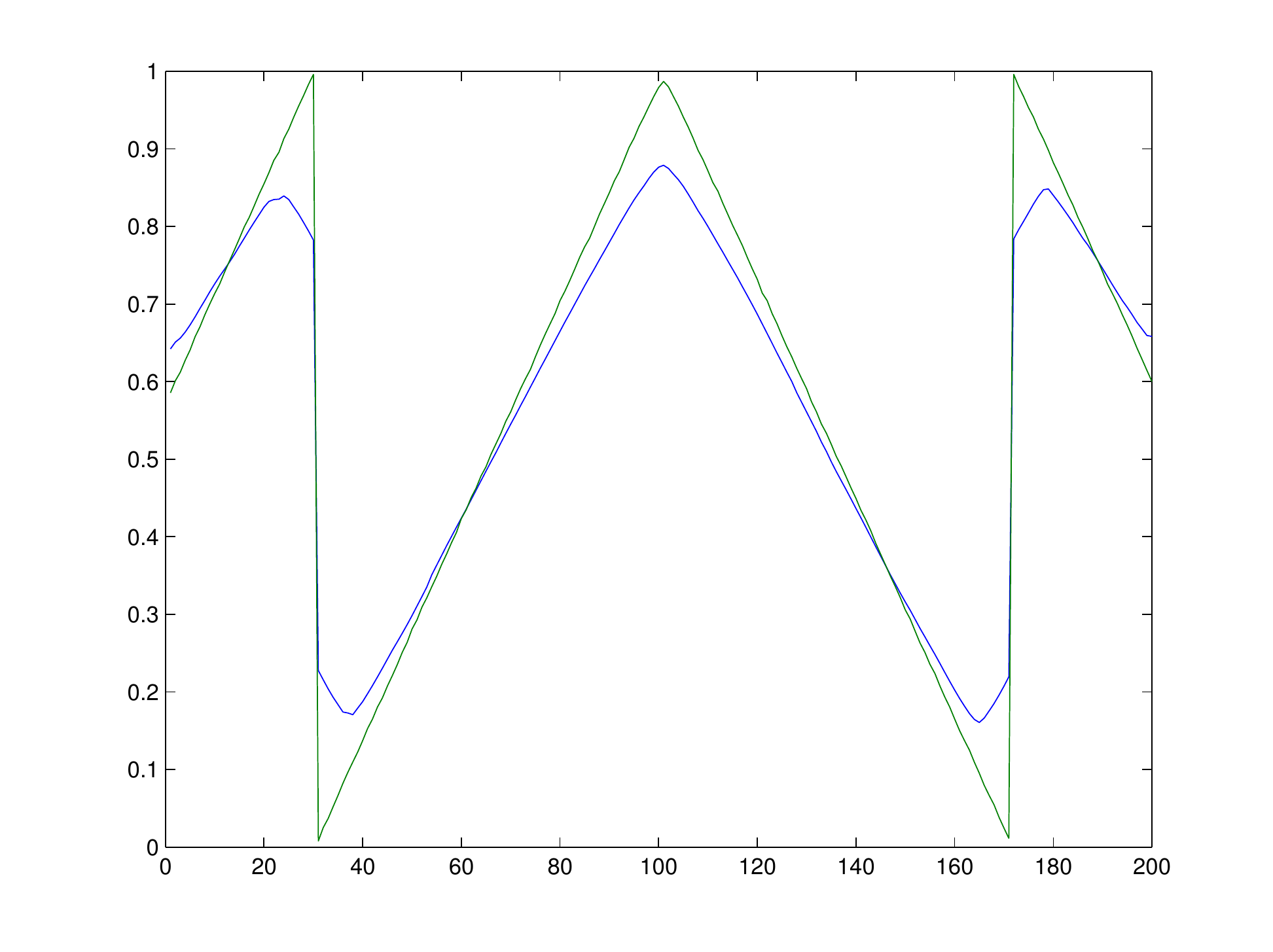}
                 \caption{\centering$\mathrm{TVL}^{7}$: $\alpha=0.8$, $\beta=3700$, SSIM=0.9024} 
                 \label{radial_2:d}
\end{subfigure}
\end{center}
\caption{Better preservation of spike-like structures for large values of $p$.}
\label{radial_2}
\end{figure}

\begin{figure}[h]
\begin{center}
\begin{subfigure}[t]{5cm}
                \centering                                                  
                \includegraphics[width=5cm]{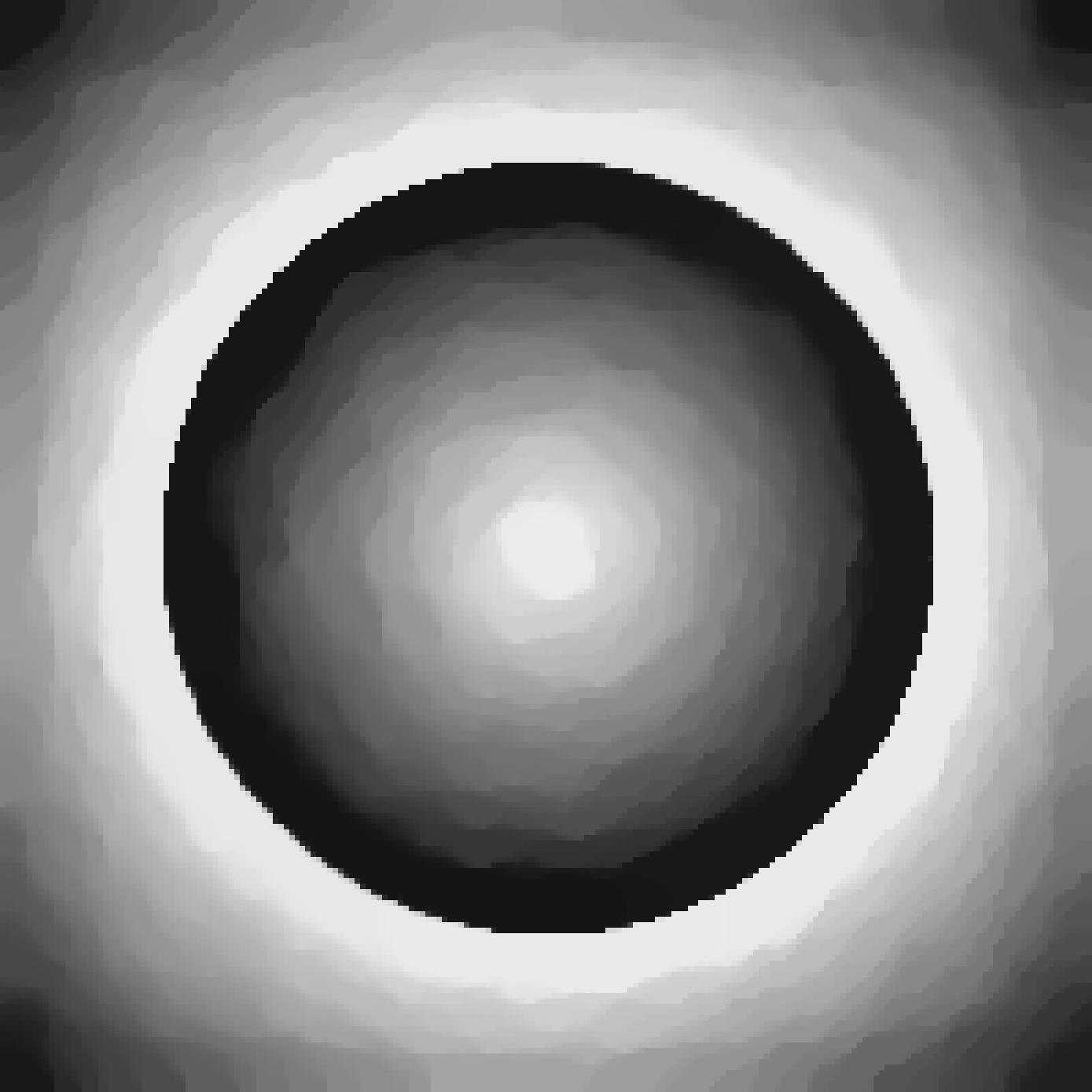}
                 \caption{\centering Bregmanised $\mathrm{TV}$: $\alpha=2$, SSIM=0.8912, 6th iteration}
\end{subfigure}
\begin{subfigure}[t]{5cm}
                \centering
                 \includegraphics[width=5cm]{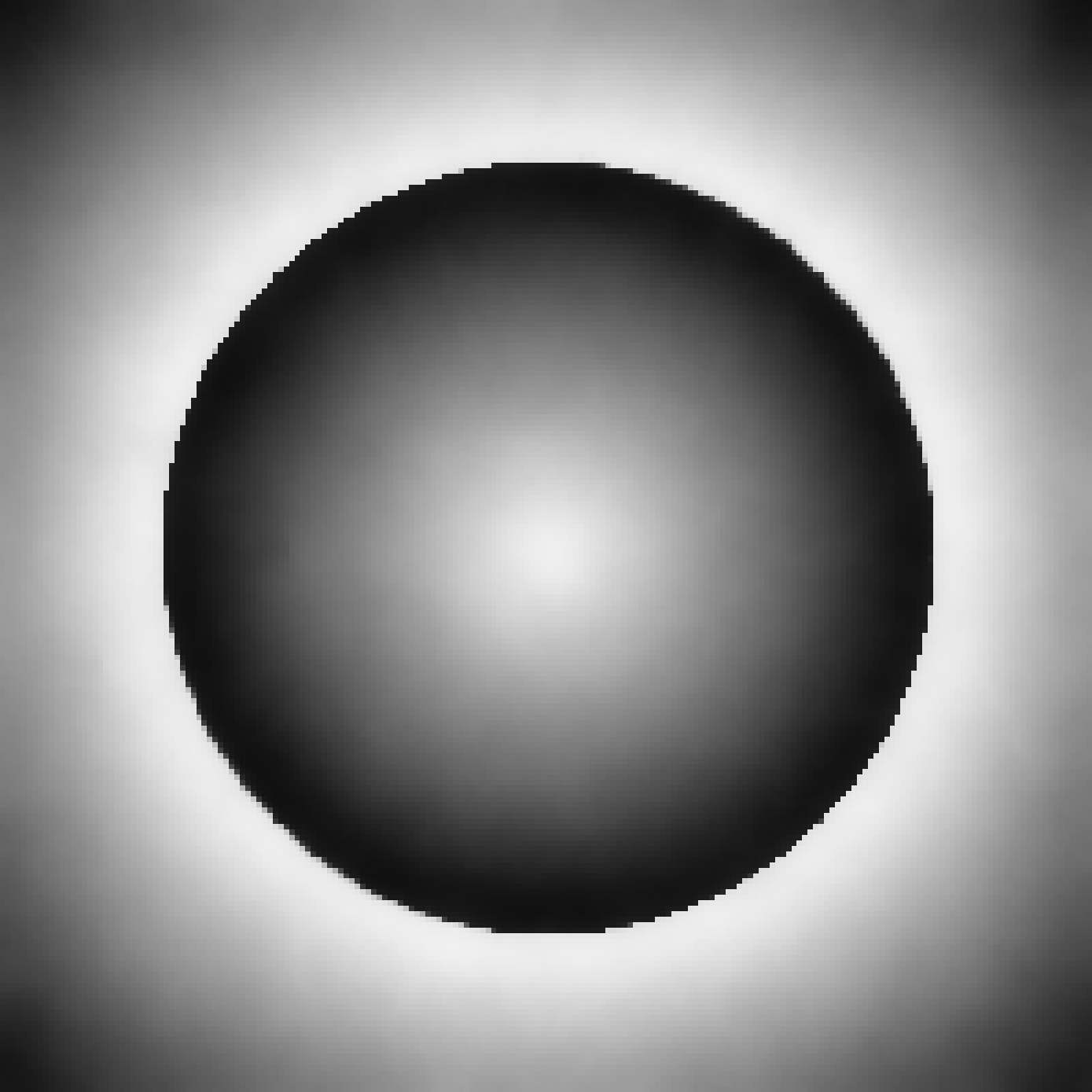}
               \caption{\centering Bregmanised $\mathrm{TVL^{2}}$: $\alpha=5$, $\beta=625$, SSIM=0.9718, 12th iteration}                                                  
\end{subfigure}
\begin{subfigure}[t]{5cm}
                \centering                                                  
                   \includegraphics[width=5cm]{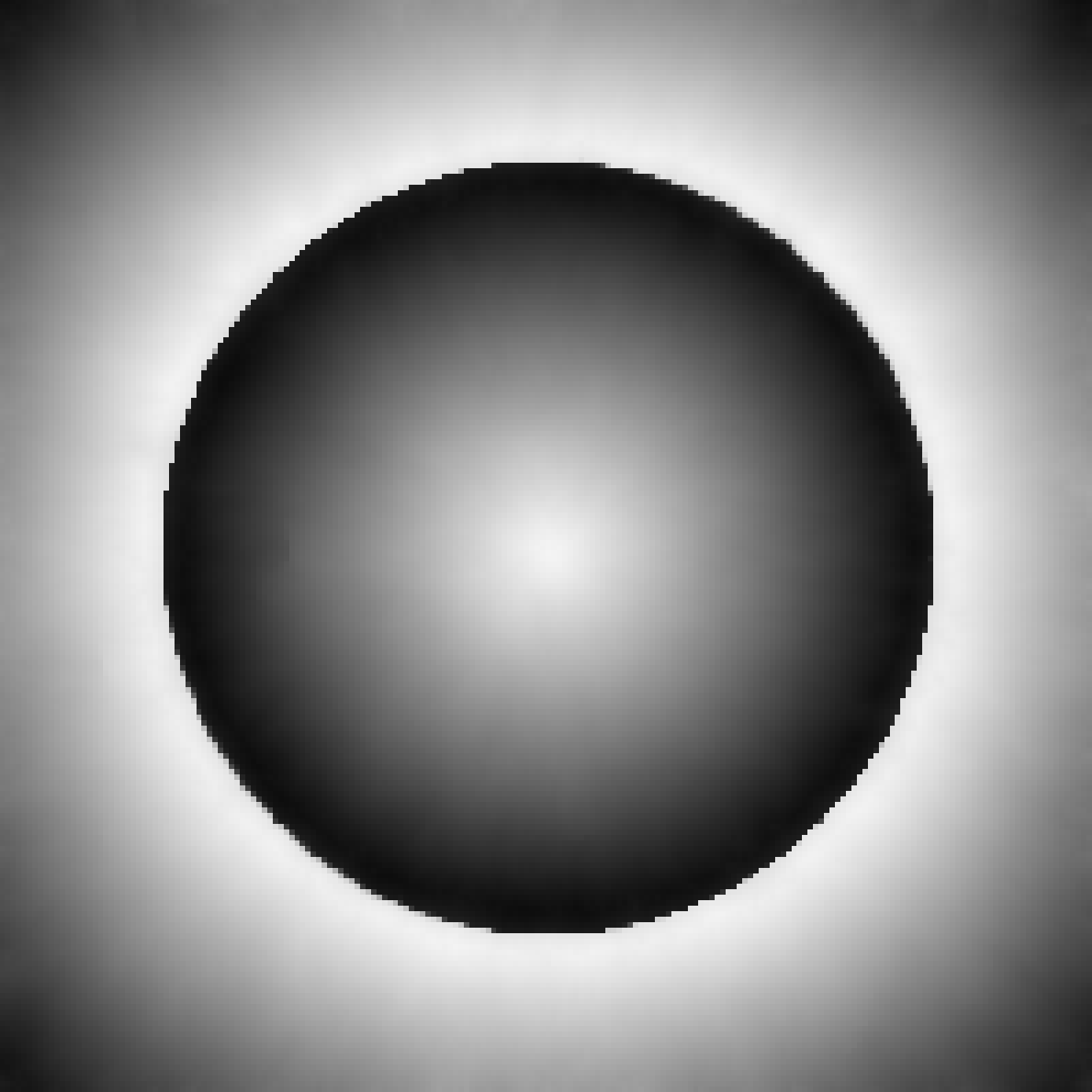}
                 \caption{\centering Bregmanised $\mathrm{TVL^{4}}$: $\alpha=5$, $\beta=8000$, SSIM=0.9802, 13th iteration} 
\end{subfigure}\\
\begin{subfigure}[t]{5cm}
                \centering                                                  
                   \includegraphics[width=5cm]{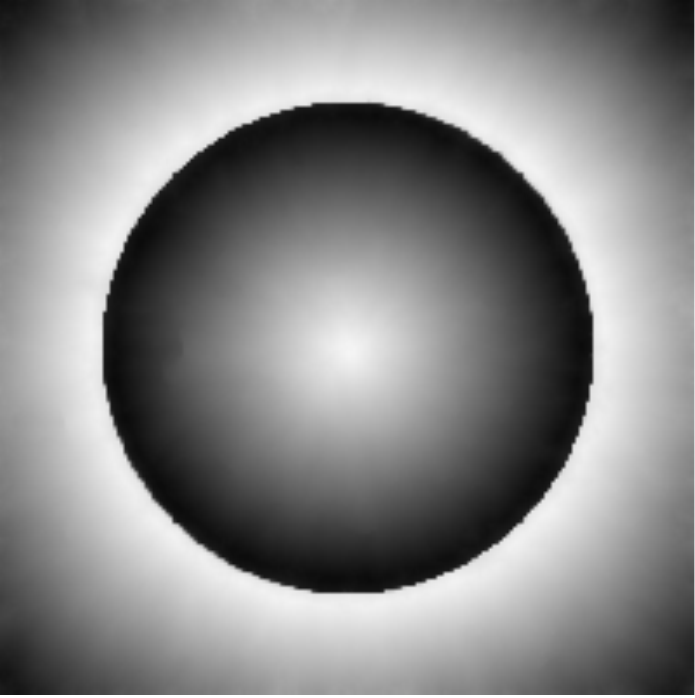}
                 \caption{\centering Bregmanised $\mathrm{TVL^{7}}$: $\alpha=3$, $\beta=15000$, SSIM=0.9807, 9th iteration} 
\end{subfigure}
\begin{subfigure}[t]{5cm}
                \centering                                                  
                \includegraphics[width=5cm]{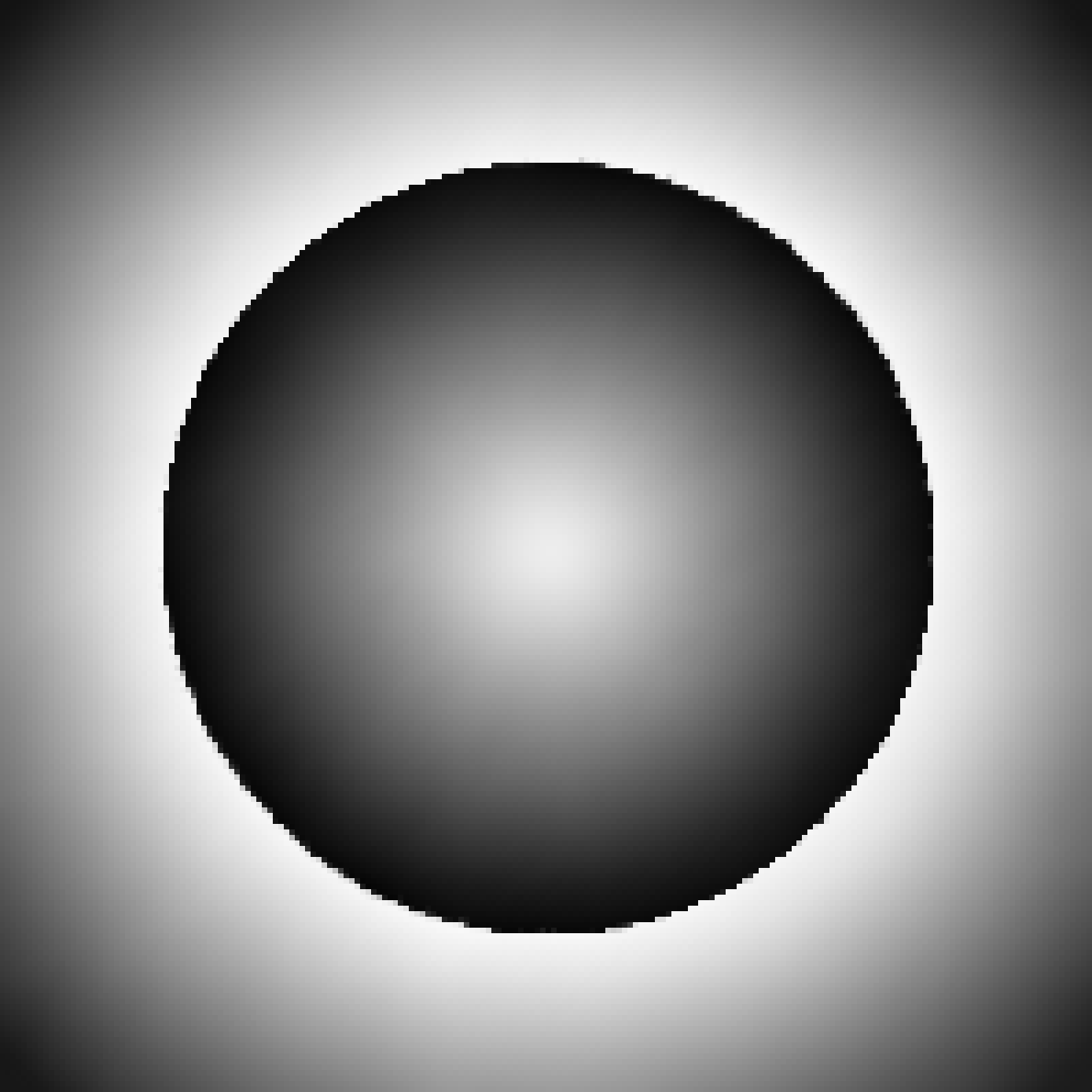}
                 \caption{\centering Bregmanised $\mathrm{TGV^{2}}$: $\alpha=2$, $\beta=10$,  SSIM=0.9913, 8th iteration} 
\end{subfigure}
\end{center}
\caption{Best reconstruction in terms of SSIM for Bregmanised $\mathrm{TVL}^{2}$, $\mathrm{TVL}^{4}$ and $\mathrm{TGV}^{2}$.}
\label{radial_3}
\end{figure}

\begin{figure}[h]
\begin{center}
\begin{subfigure}[t]{3.8cm}
                \centering                                                  
                \includegraphics[width=3.8cm]{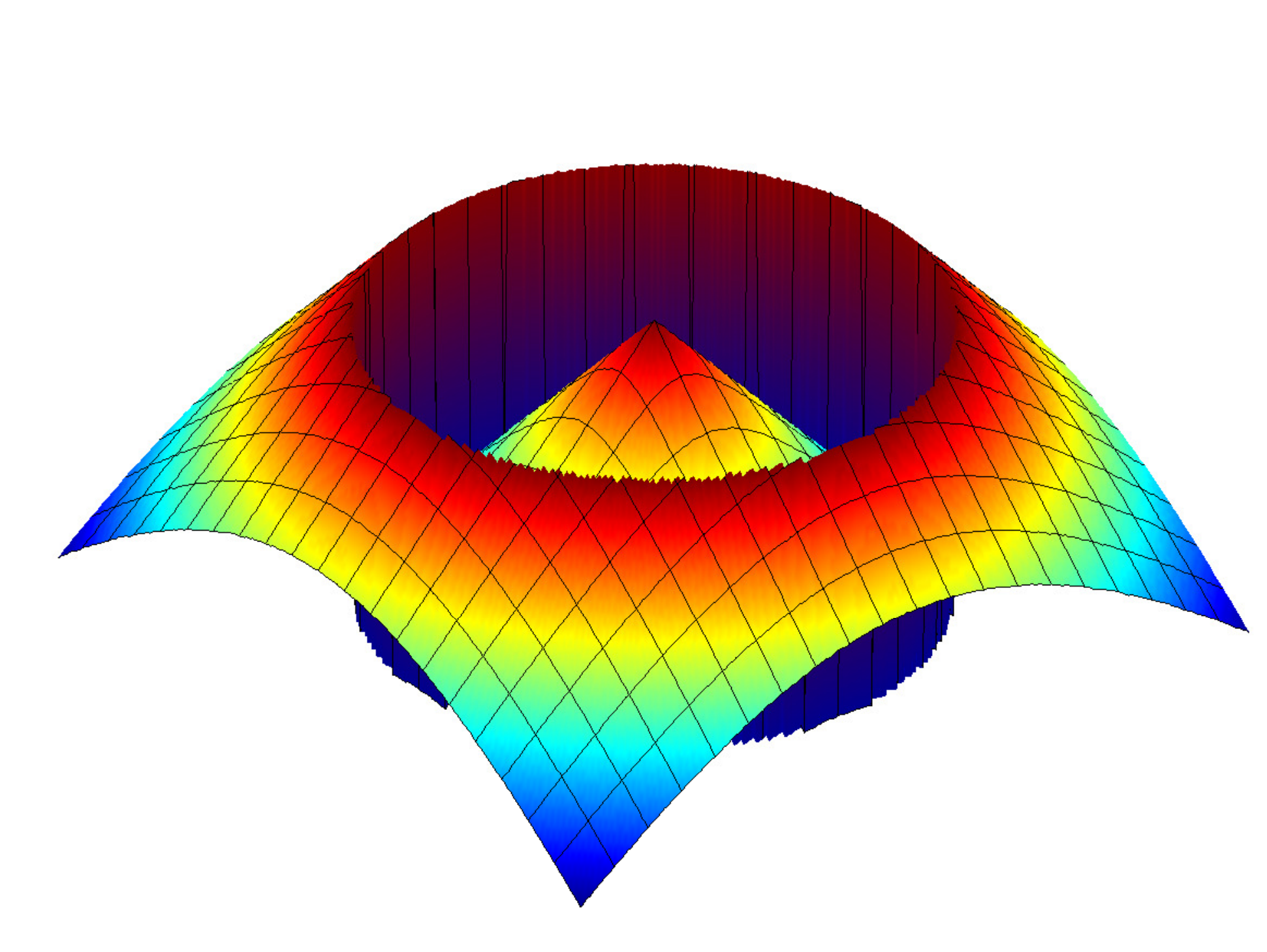}
                 \caption{\centering Original}
\end{subfigure}
\begin{subfigure}[t]{3.8cm}
                \centering                                                  
                \includegraphics[width=3.8cm]{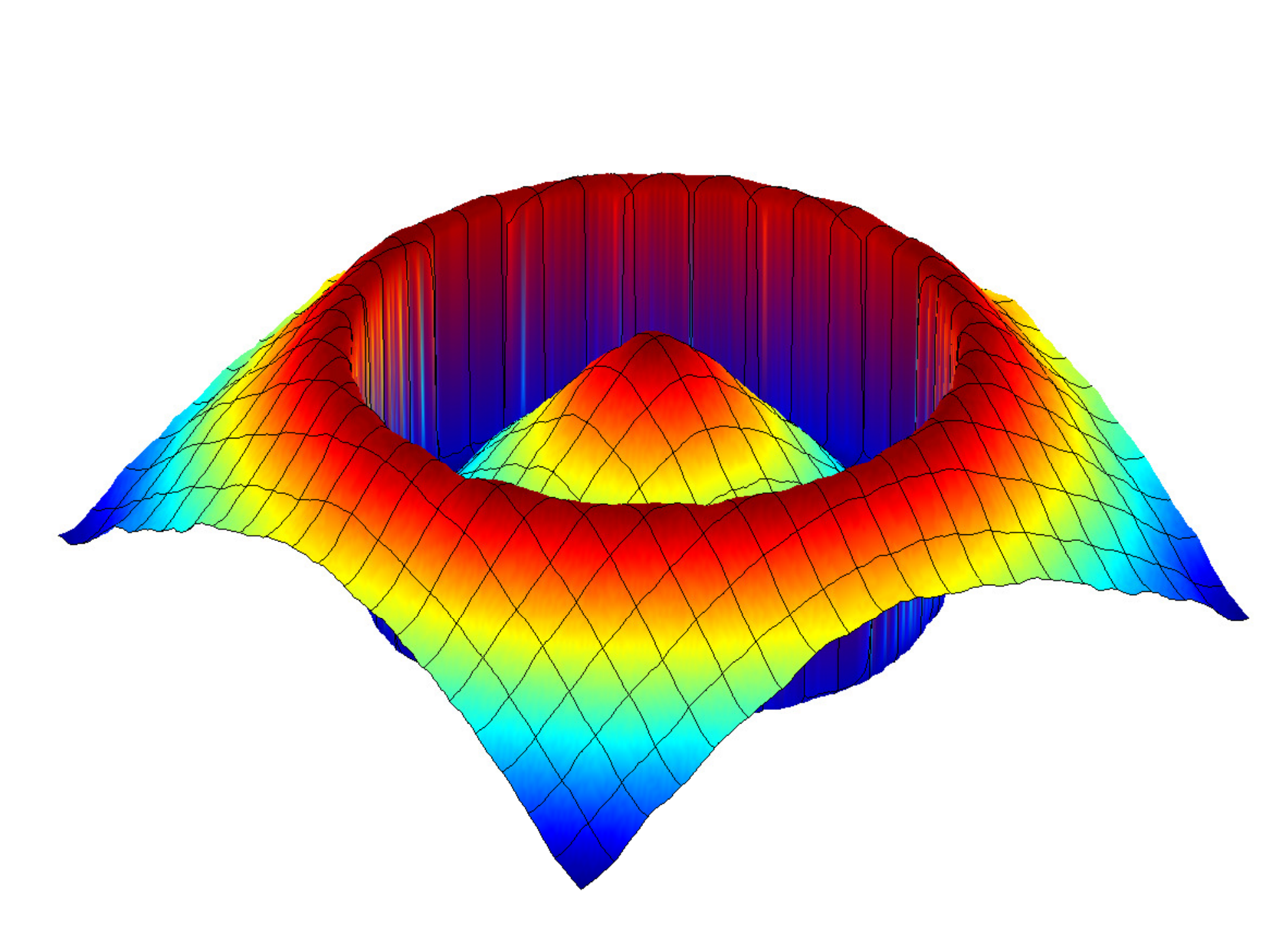}
                 \caption{\centering Bregmanised $\mathrm{TVL}^{2}$}
\end{subfigure}
\begin{subfigure}[t]{3.8cm}
                \centering                                                  
                \includegraphics[width=3.8cm]{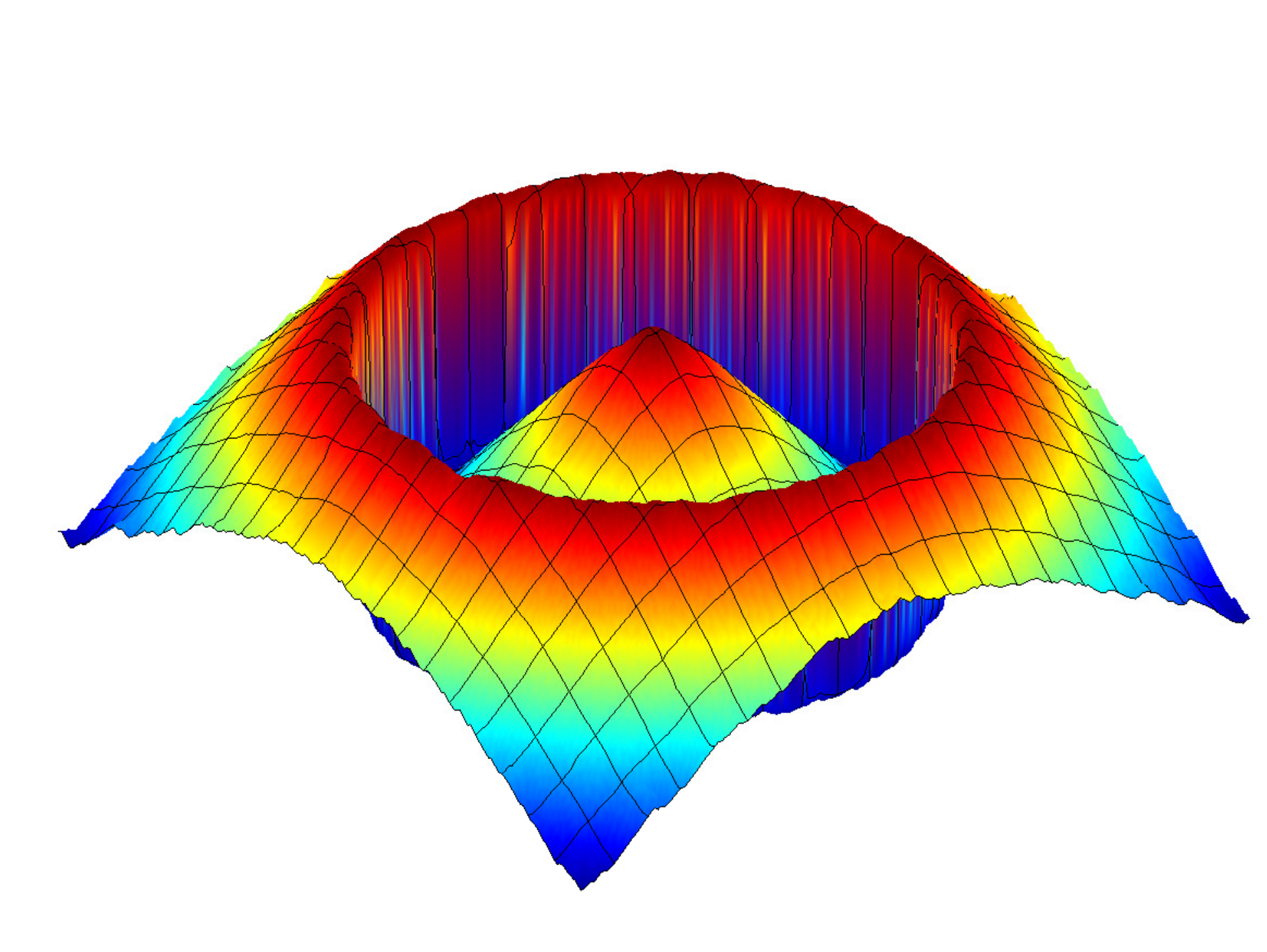}
                 \caption{\centering Bregmanised $\mathrm{TVL}^{7}$}
\end{subfigure}
\begin{subfigure}[t]{3.8cm}
                \centering                                                  
                \includegraphics[width=3.8cm]{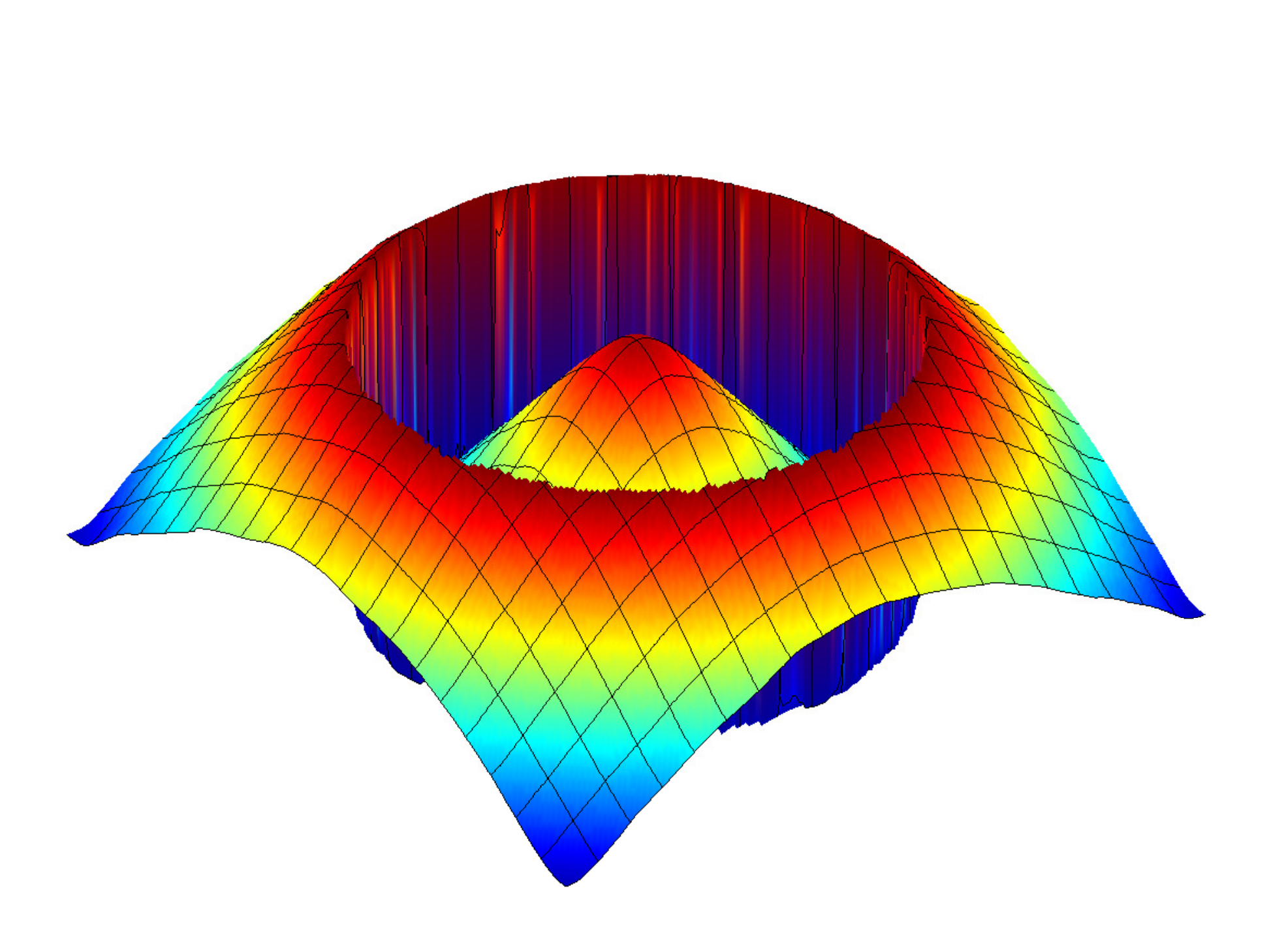}
                 \caption{\centering Bregmanised $\mathrm{TGV}$}
\end{subfigure}
\\
\begin{subfigure}[t]{3.8cm}
                \centering                                                  
                \includegraphics[width=3.8cm]{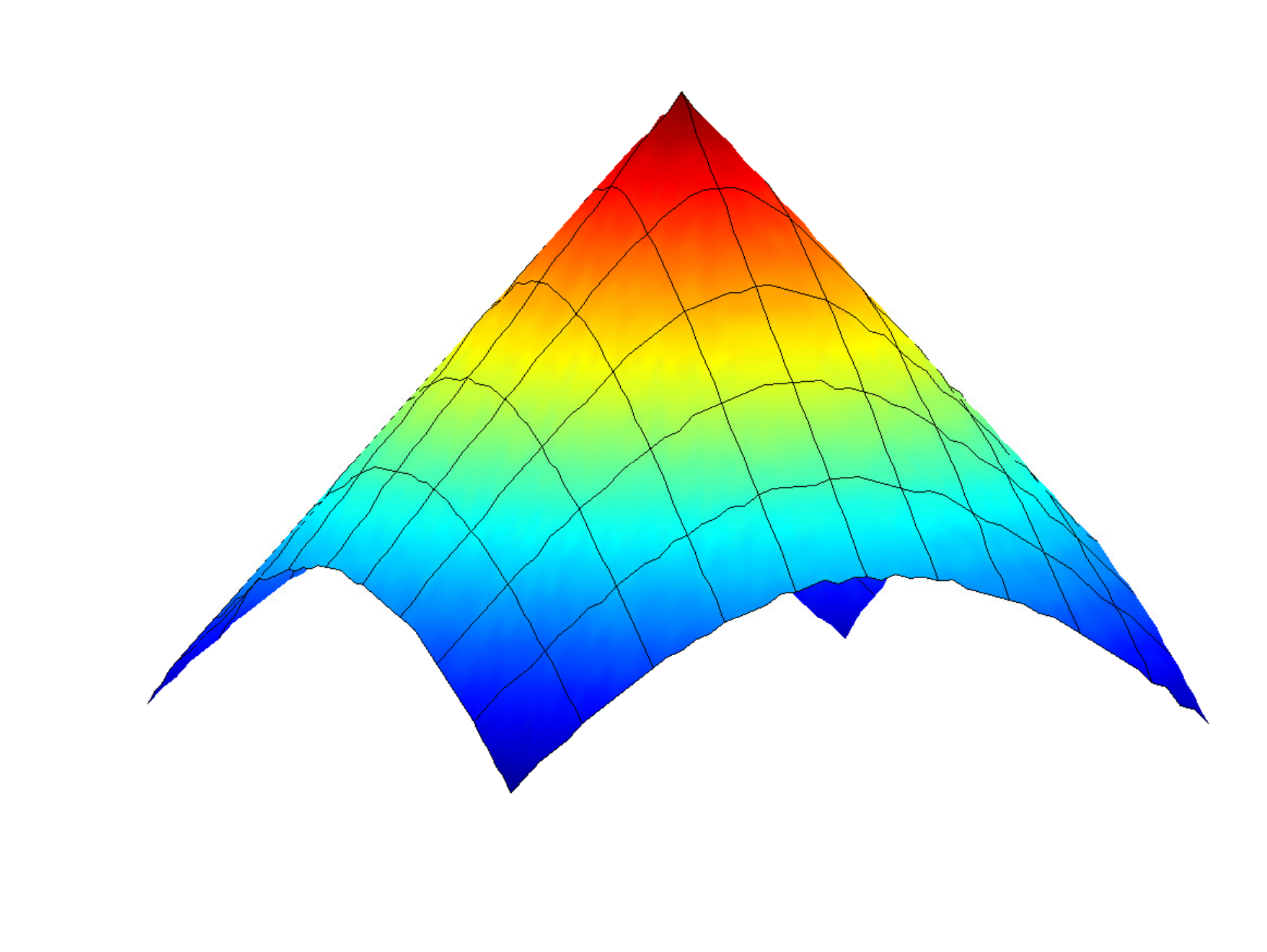}
                 \caption{\centering Original: \newline central part zoom}
\end{subfigure}
\begin{subfigure}[t]{3.8cm}
                \centering                                                  
                \includegraphics[width=3.8cm]{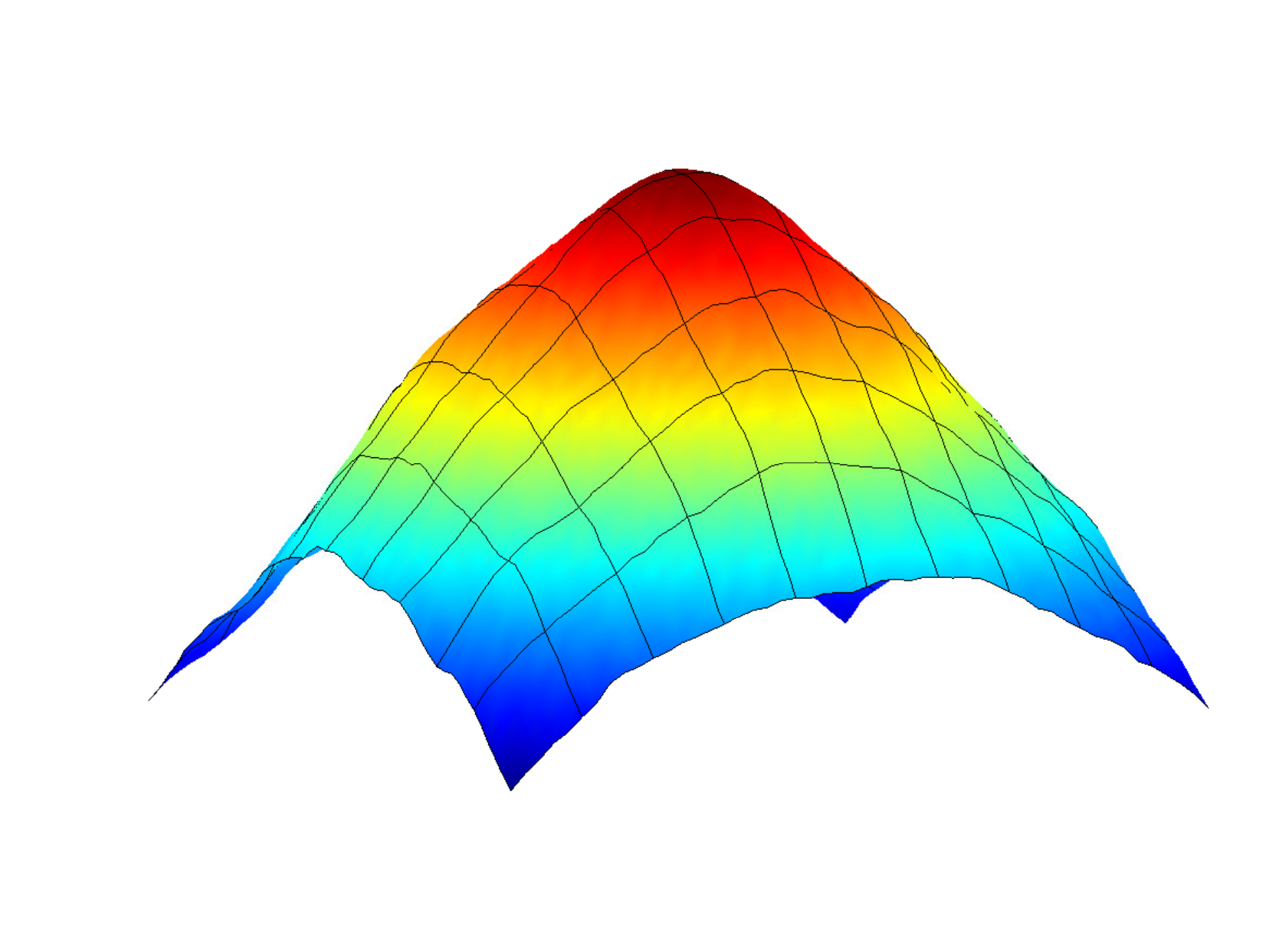}
                 \caption{\centering Bregmanised $\mathrm{TVL}^{2}$: central part zoom}
\end{subfigure}
\begin{subfigure}[t]{3.8cm}
                \centering                                                  
                \includegraphics[width=3.8cm]{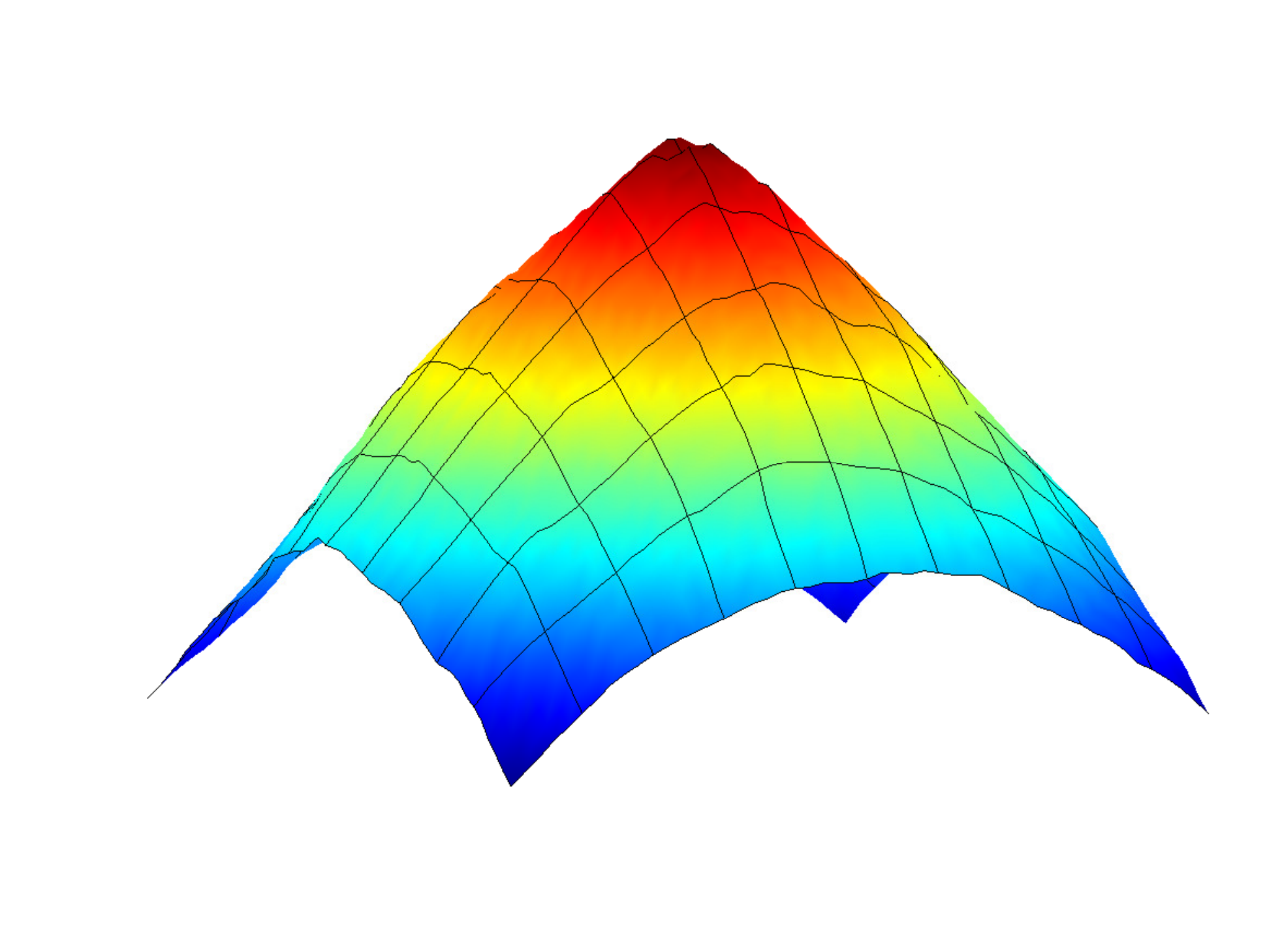}
                 \caption{\centering Bregmanised $\mathrm{TVL}^{7}$: central part zoom}
\end{subfigure}
\begin{subfigure}[t]{3.8cm}
                \centering                                                  
                \includegraphics[width=3.8cm]{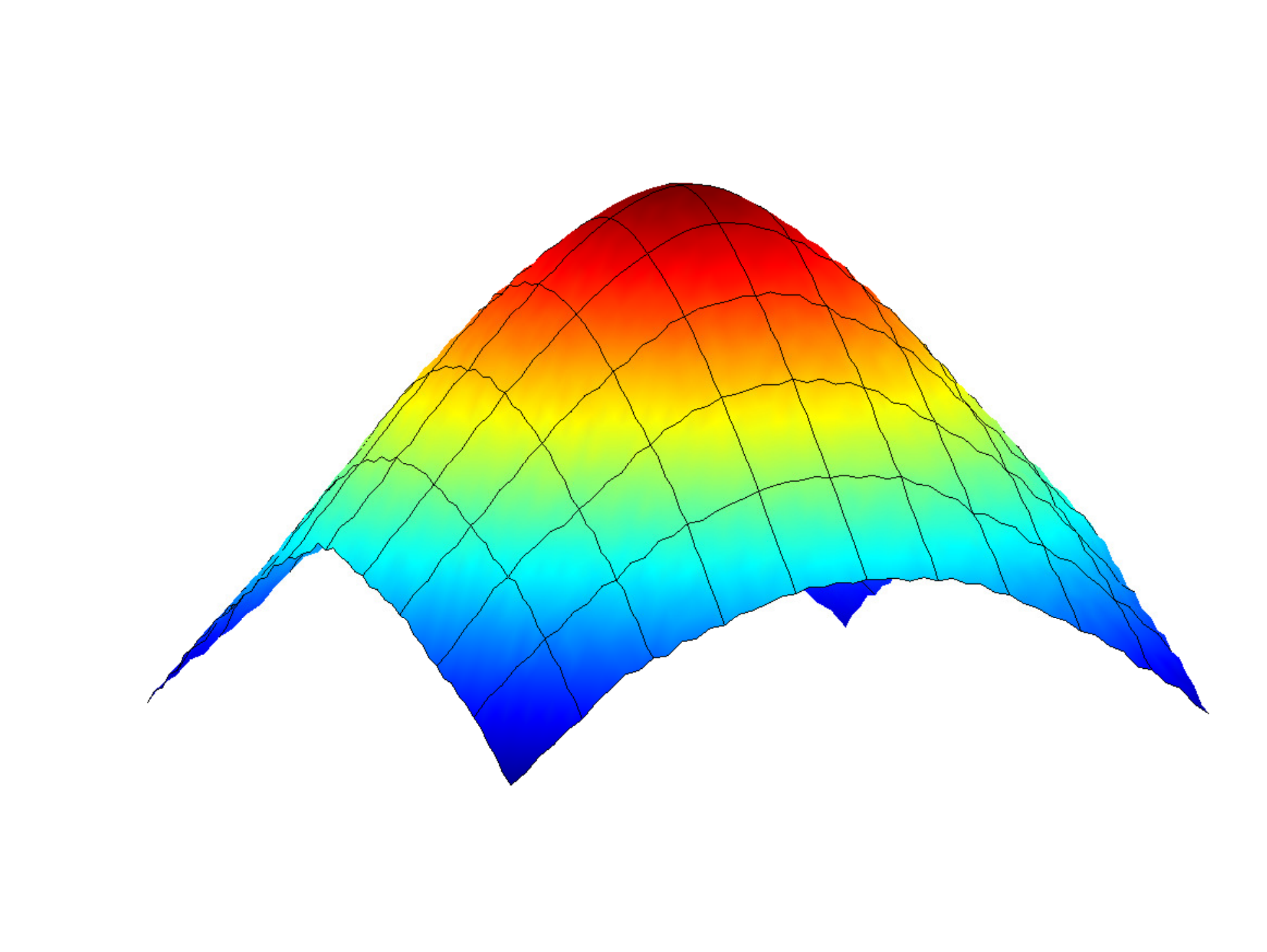}
                 \caption{\centering Bregmanised $\mathrm{TGV}$: central part zoom}
\end{subfigure}
\end{center}
\caption{Surface plots of the images in Figure \ref{radial_3}. Notice how high values of $p$, here for $p=7$, can preserve the sharp spike in the middle of the image.}
\label{radial_3_surf}
\end{figure}

\begin{figure}[h]
\begin{center}
\begin{subfigure}[t]{3.8cm}
                \centering                                                  
                \includegraphics[width=3.8cm]{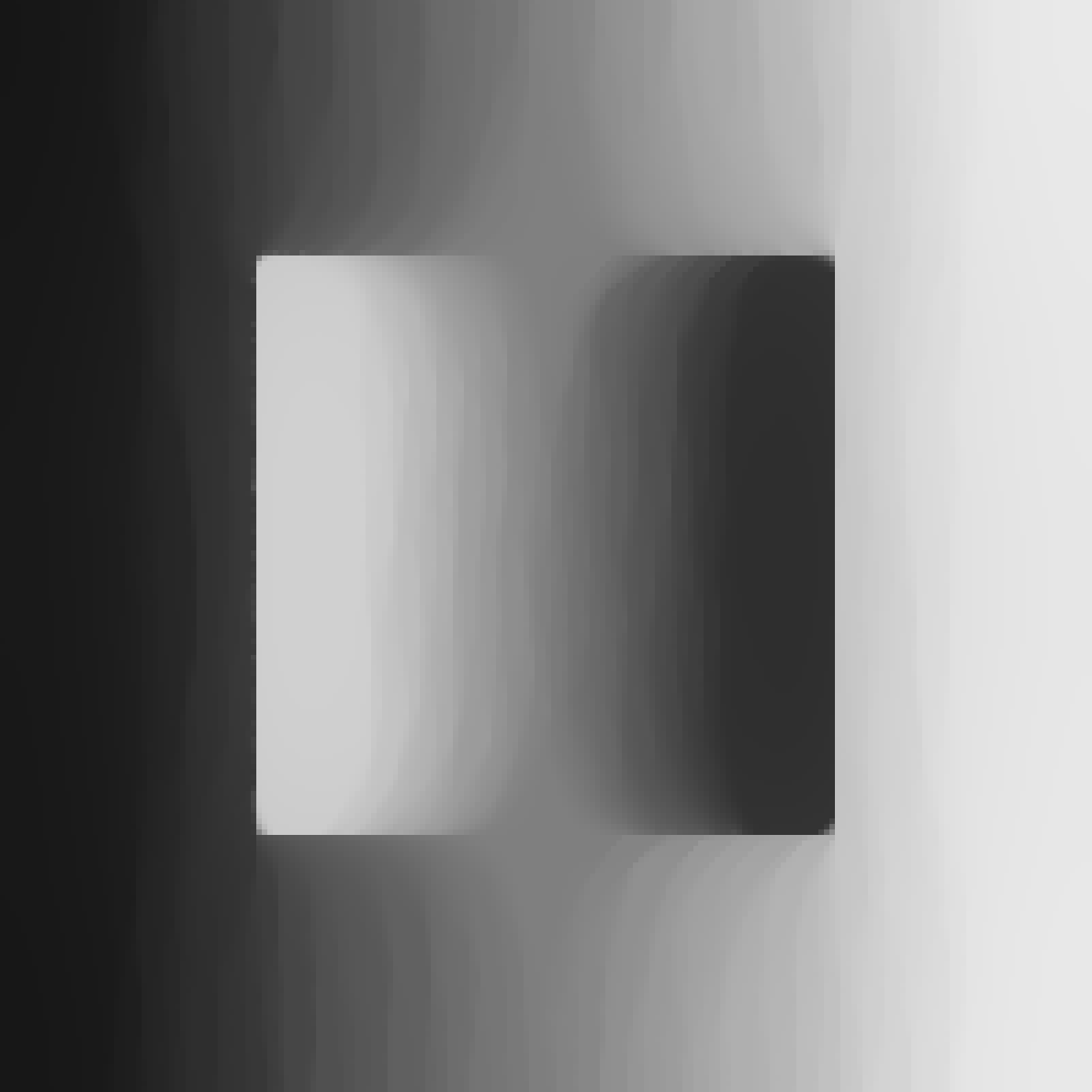}
                \caption{\centering Solution $u+v$ of \eqref{inf_conv}: $p=2$, $\alpha=0.8$, $\beta=120$, SSIM=0.9268} 
                \label{infimal_2D_1:a}
\end{subfigure}
\begin{subfigure}[t]{3.8cm}
                \centering                                                  
                \includegraphics[width=3.8cm]{no_stair_tvl_2_a_1_b_100-eps-converted-to.pdf}
                \caption{\centering $\mathrm{TVL}^{2}$: $\alpha=1$, $\beta=116$, SSIM=0.9433} 
                \label{infimal_2D_1:b}
\end{subfigure}
\begin{subfigure}[t]{3.8cm}
                \centering                                                  
                \includegraphics[width=3.8cm]{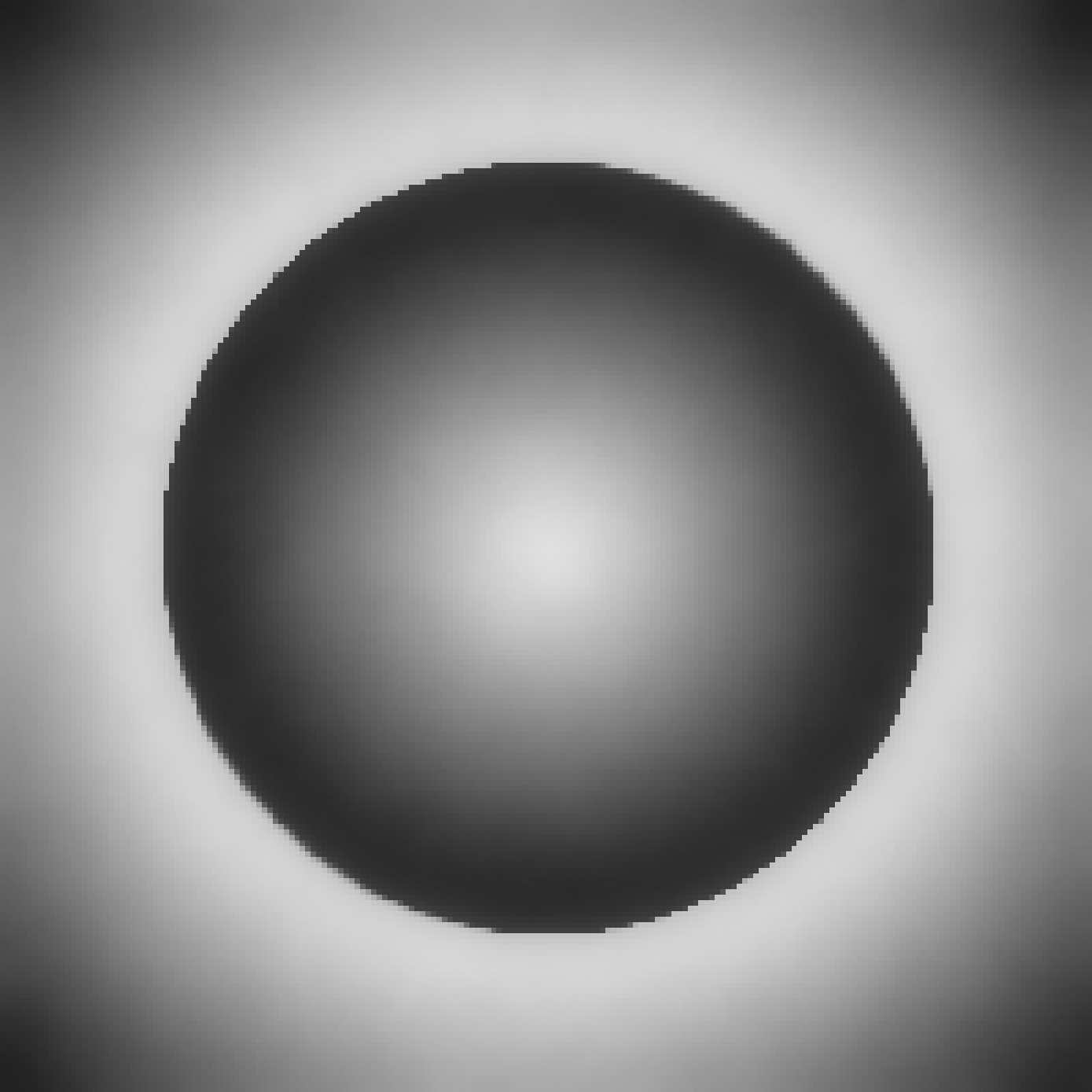}
                \caption{\centering Solution $u+v$ of \eqref{inf_conv} $p=2$: $\alpha=0.8$, $\beta=70$, SSIM=0.8994} 
                \label{infimal_2D_1:c}
\end{subfigure}
\begin{subfigure}[t]{3.8cm}
                \centering                                                  
                \includegraphics[width=3.8cm]{radial_p_2_a_0_8_b_35_0_8998-eps-converted-to.pdf}
                 \caption{\centering $\mathrm{TVL^{2}}$: $\alpha=0.8$, $\beta=79$, SSIM=0.8998} 
                 \label{infimal_2D_1:d}
\end{subfigure}
\end{center}
\caption{Comparison between the model \eqref{inf_conv} for $p=2$ and $\mathrm{TVL^{2}}$: Staircasing cannot be always eliminated.}
\label{infimal_2D_1}
\end{figure}

\begin{figure}[h]
\begin{center}
\hspace{-1cm}
\begin{subfigure}[t]{4.4cm}
                \centering                                                  
                \includegraphics[width=4.4cm]{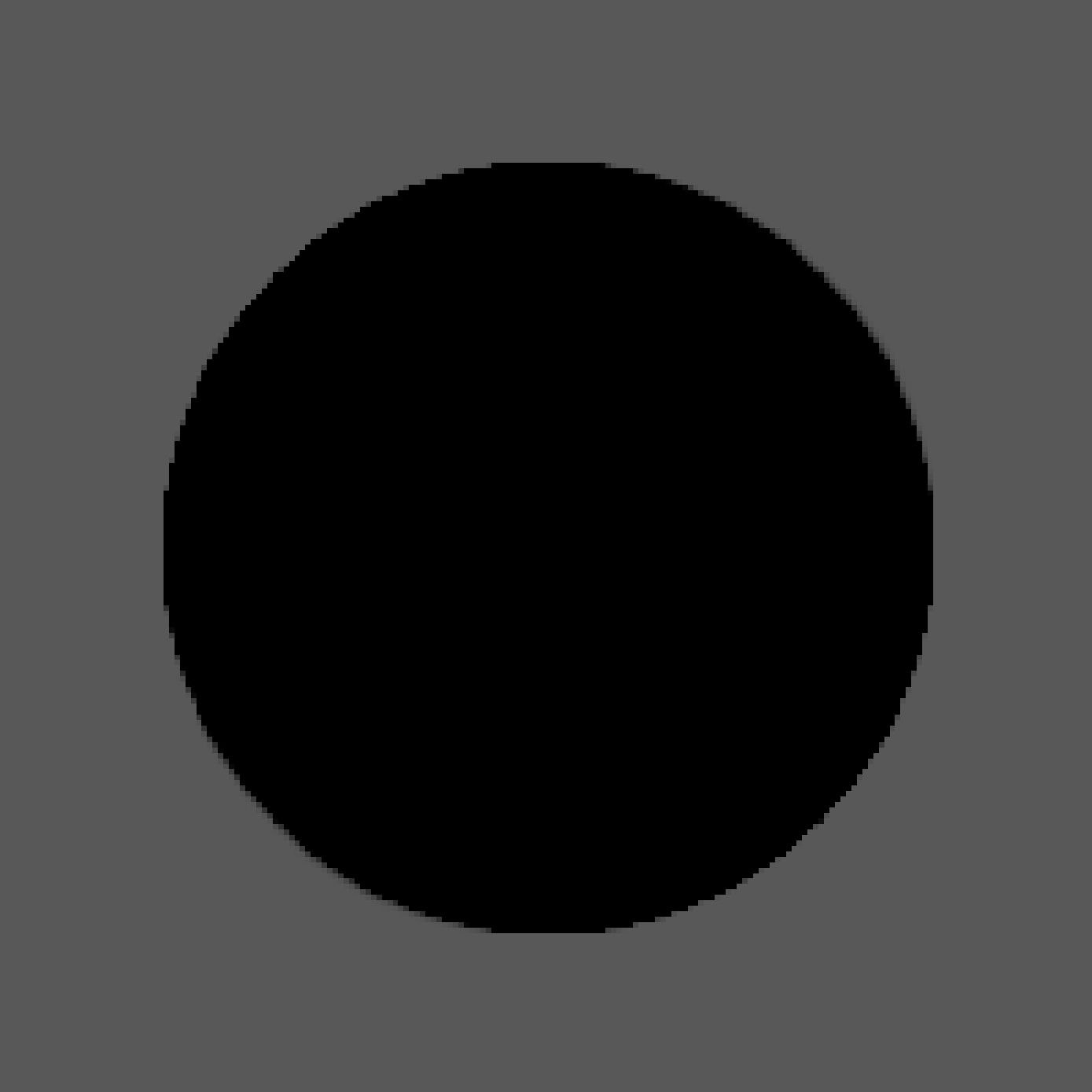}
                \caption{\centering Piecewise constant \newline \centering component $u$} 
\end{subfigure}
\begin{subfigure}[t]{4.4cm}
                \centering                                                  
                \includegraphics[width=4.4cm]{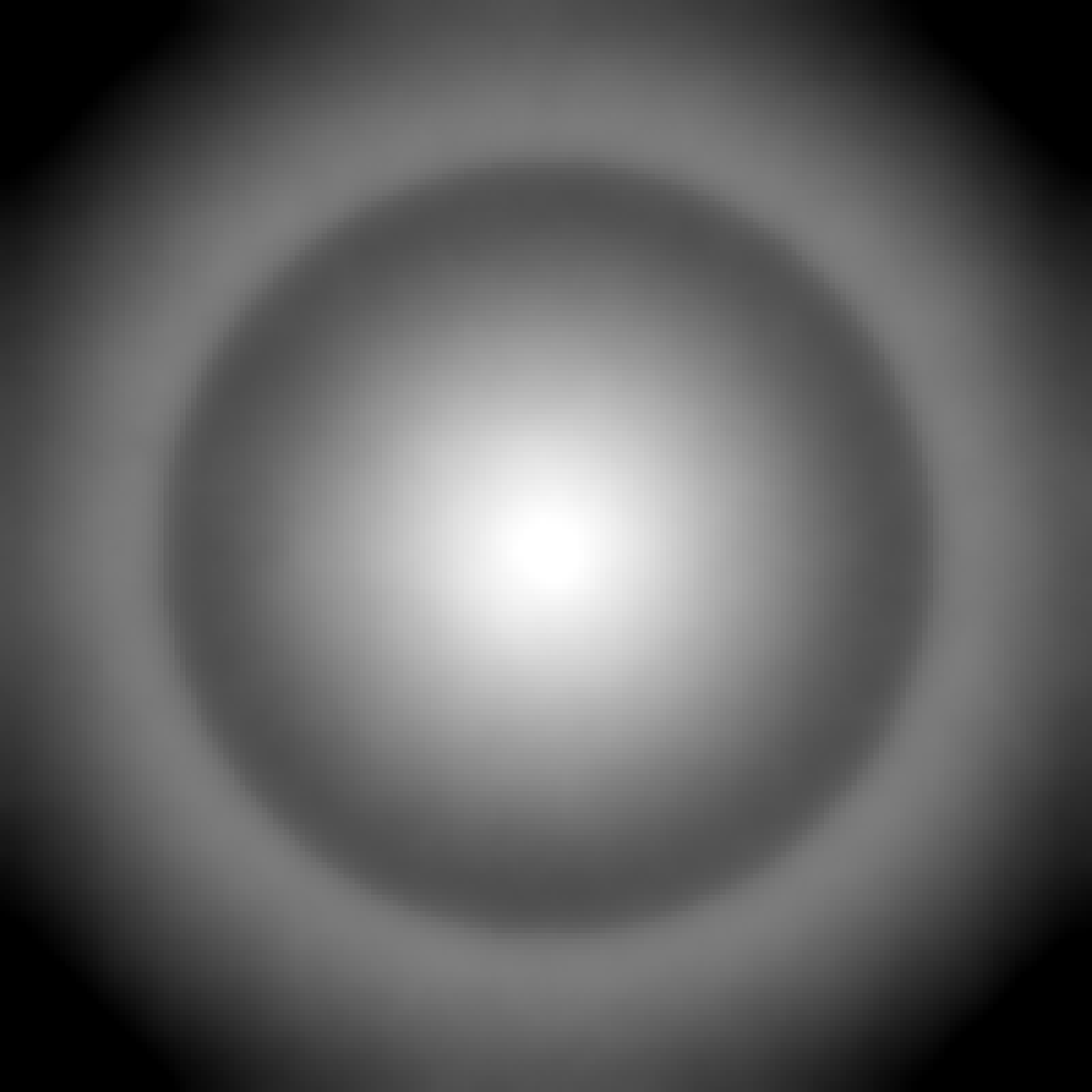}
                \caption{Smooth component $v$} 
\end{subfigure}
\begin{subfigure}[t]{4.4cm}
                \centering                                                  
                \includegraphics[height=4.4cm]{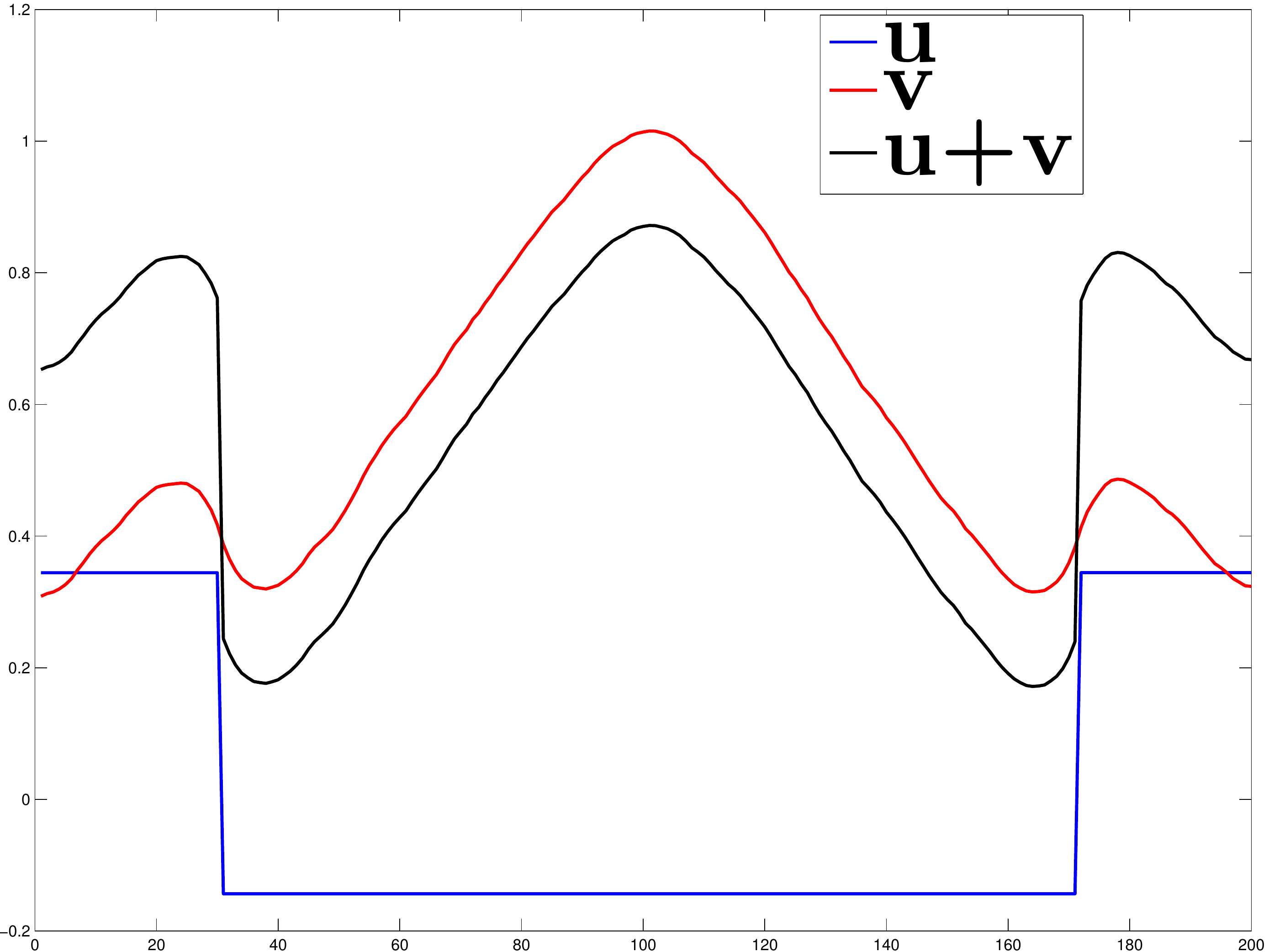}
                \caption{\centering Middle row profiles} 
\end{subfigure}
\end{center}
\caption{Geometric decomposition of the image in Figure \ref{infimal_2D_1:d} into a piecewise constant and smooth component, $u$ and $v$ respectively, by solving  \eqref{inf_conv}.}
\label{infimal_2D_2}
\end{figure}

We conclude with numerical results for the image decomposition approach of Section \ref{sec:infimal} which we solve again using the split Bregman algorithm. Recall that in dimension two, the solutions of \eqref{inf_conv} will not necessarily be the same with the ones of \eqref{tvlp_min}. In fact, we observe that  \eqref{inf_conv} cannot always eliminate the staircasing, see for instance Figure \ref{infimal_2D_1}. Even though, we can easily eliminate the staircasing both in the square and in the circle by applying $\mathrm{TVL}^{p}$ regularisation, Figures \ref{infimal_2D_1:b} and \ref{infimal_2D_1:d}, we cannot obtain equally satisfactory results by solving  \eqref{inf_conv}. While using the latter we can get rid of the staircasing in the circle, Figure \ref{infimal_2D_1:c}, this is not possible for the square, Figure \ref{infimal_2D_1:a}, where we observe -- after extensive experimentation -- that no values of $\alpha$ and $\beta$ lead to a staircasing elimination. This is analogous to the difference between $\mathrm{TGV}^{2}$ and the $\mathrm{TV}$--$\mathrm{TV}^{2}$ infimal convolution of Chambolle--Lions \cite{ChambolleLions}.  

However, as we mentioned before, the strength of the formulation \eqref{inf_conv} lies on its ability to efficiently decompose an image into piecewise constant and smooth parts. We show that in Figure \ref{infimal_2D_2}, for the image in Figure \ref{infimal_2D_1:c}. 

\section{Conclusion}
\label{sec:conclusion}

We have introduced a novel first-order, one-homogeneous $\mathrm{TV}$--$\mathrm{L}^{p}$ infimal convolution type functional for variational image regularisation. The $\mathrm{TVL}^{p}$ functional constitutes a very general class of regularisation functionals exhibiting diverse smoothing properties for different choices of $p$. In the case $p=2$ the well-known Huber $\mathrm{TV}$ regulariser is recovered. 

We studied the corresponding one dimensional denoising problem focusing on the structure of its solutions. We computed exact solutions of this problem for the case $p=2$ for simple one dimensional data. Hence, as an additional novelty in our paper we presented exact solutions of the one dimensional Huber $\mathrm{TV}$ denoising problem.
 
 Numerical experiments for several values of $p$ indicate that our model leads to an elimination of the staircasing effect. We show that we can further enhance our results by increasing the contrast via a Bregman iteration scheme and thus obtaining results of similar quality to those of $\mathrm{TGV}^{2}$. Furthermore, as $p$ increases the structure of the solutions changes from piecewise smooth to piecewise linear and the model, in contrast to $\mathrm{TGV}^{2}$, is capable of preserving sharp spikes in the reconstruction. This observation motivates a more detailed study of the $\mathrm{TVL}^{p}$ functionals for large $p$ and in particular for the case $p=\infty$.
 
 This concludes the first part of the study of the $\mathrm{TV}$--$\mathrm{L}^{p}$ model for $p< \infty$. The second part \cite{partII}, is devoted to the $p=\infty$ case. There we explore further, both in an analytical and an experimental level,  the capability of the $\mathrm{TVL}^{\infty}$ model to promote affine and  spike-like structures in the reconstructed image and we discuss several applications.

\subsection*{Acknowledgements}
The authors acknowledge support of the Royal Society International Exchange Award Nr. IE110314. This work is further supported by the King Abdullah University for Science and Technology (KAUST) Award No. KUK-I1-007-43, the EPSRC first grant Nr. EP/J009539/1 and the EPSRC grant Nr. EP/M00483X/1.
MB acknowledges further support by ERC via Grant EU FP 7-ERC Consolidator Grant 615216 LifeInverse. KP acknowledges further support by the Cambridge Centre for Analysis (CCA) and the  Engineering and Physical Sciences Research Council (EPSRC). EP acknowledges support by Jesus College, Cambridge and Embiricos Trust Scholarship.

\bibliographystyle{amsplain}
\bibliography{references}

\appendix

\section{Radon Measures and functions of bounded variation}\label{sec:preliminaries}


In what follows $\Omega\subset\re^{d}$ is an open, bounded set with Lipschitz boundary whose Lebesgue measure is denoted by $|\Omega|$. We denote by $\mathcal{M}(\Omega,\re^{d})$ (and $\mathcal{M}(\Omega)$ if d=1) the space of finite Radon measures on $\Omega$. The total variation measure of $\mu\in\mathcal{M}(\Omega,\re^{d})$ is denoted by $|\mu|$, while we denote the polar decomposition of $\mu$ by $\mu=\mathrm{sgn}(\mu)|\mu|$, where $\mathrm{sgn}(\mu)=1$ $|\mu|$-almost everywhere.

Recall that the \emph{Radon norm} of a $\mathbb{R}^{d}$-valued distribution $\mathcal{T}$ on $\Omega$ is defined as 
\[\|\mathcal{T}\|_{\mathcal{M}}:= \sup \left\{  \scalprod{}{u}{\phi}:\; \phi\in C_{c}^{\infty}(\Omega,\mathbb{R}^{d}),\; \norm{\infty}{\phi}\leq1 \right \}.\]
It can be shown that $\|\mathcal{T}\|_{\mathcal{M}}<\infty$ if and only if $\mathcal{T}$ can be represented by a measure $\mu\in \mathcal{M}(\Omega,\re^{d})$ and in that case $\|\mu\|_{\mathcal{M}}= |\mu|(\Omega)$. 

A function $u\in \mathrm{L}^{1}(\Omega)$ is a function of bounded variation if its distributional derivative $Du$ is representable by a finite Radon measure.   We denote by $\mathrm{BV}(\Omega)$, the space of functions of bounded variation which is a Banach space under the norm $$\norm{\mathrm{BV(\mathrm{\Omega})}}{u}:=\norm{\mathrm{L}^{1}(\Omega)}{u}+\norm{\mathcal{M}}{Du}.$$ 
The term 
\[\norm{\mathcal{M}}{Du}=\sup \left \{\int_{\Omega}u\, \mathrm{div}\phi\,dx: \phi\in C_{c}^{\infty}(\Omega,\mathbb{R}^{d}),\;\|\phi\|_{\infty}\le 1 \right\},\]
is called the total variation of $u$, also commonly denoted by $\mathrm{TV}(u)$. From the Radon--Nikodym theorem, the measure $\mu$ can be decomposed into an absolutely continuous and a singular part with respect to the Lebesgue measure $\mathcal{L}^{d}$, that is $Du=D^{a}u+D^{s}u$. Here, $D^{a}u=\nabla u\mathcal{L}^{d}$, i.e.,  $\nabla u$ denotes the Radon--Nikodym derivative of $D^{a}u$ with respect to $\mathcal{L}^{d}$. When $d=1$, $\nabla u$ is simply denoted by $u'$. 

We will also use the following basic inequality regarding inclusions of $\mathrm{L}^{p}$ spaces
\begin{equation}\label{Lp_inclusion}
\norm{\mathrm{L}^{p_{1}}(\Omega)}{h}\leq |\Omega|^{\frac{1}{p_{1}}-\frac{1}{p_{2}}}\norm{\mathrm{L}^{p_{2}}(\Omega)}{h}, \quad 1\le p_{1}<p_{2}\le\infty.
\end{equation} 

Unless otherwise stated $q$ denotes the H\"older conjugate of the exponent $p$, i.e.,
\begin{equation}
q=
\begin{cases}
 \frac{p}{p-1}& \mbox{ if }p\in(1,\infty),\\
1 & \mbox{ if }p=\infty.
\end{cases}
\label{conj_expon}
\end{equation}
Regarding the subdifferential of the Radon norm we have that it can be  characterised,  at least for $C_{0}$ functions, as follows \cite{Bredies} 
\begin{equation}
\partial\norm{\mathcal{M}}{\cdot}(\mu)\cap C_{0}(\Omega)=\Sgn(\mu)\cap C_{0}(\Omega),
\label{sign2}
\end{equation}
where here $\Sgn(\mu)$  denotes the set-valued sign
\begin{equation}
\Sgn(\mu)=\left\{v\in\mathrm{L^{\infty}(\Omega)}\cap\mathrm{L^{\infty}(\Omega,|\mu|)}:\norm{\infty}{v}\leq1, v=\sgn(\mu), |\mu| - a.e.\right\}.
\label{sign1}
\end{equation} 
\end{document}